\pdfoutput=1
\RequirePackage{ifpdf}
\ifpdf 
\documentclass[pdftex]{sigma}
\else
\documentclass{sigma}
\fi

\numberwithin{equation}{section}

\newtheorem{Theorem}{Theorem}[section]
\newtheorem*{Theorem*}{Theorem}
\newtheorem{Corollary}[Theorem]{Corollary}
\newtheorem{Lemma}[Theorem]{Lemma}
\newtheorem{Proposition}[Theorem]{Proposition}
 \theoremstyle{definition}
\newtheorem{Definition}[Theorem]{Definition}

\newtheorem{Example}[Theorem]{Example}
\newtheorem{Remark}[Theorem]{Remark} %
\newtheorem{con}[Theorem]{Conjecture}

\usepackage{ytableau}
\ytableausetup{aligntableaux=top,mathmode,boxsize=1em}
\def\young(#1){\ytableaushort{#1}}
\def\yng(#1){\tiny {\ydiagram{#1}}}

\newcommand{\bbC}{\mathbb{C}}

\newcommand{\bbZ}{\mathbb{Z}}

\newcommand{\beq}{\begin{equation}}
\newcommand{\eeq}{\end{equation}}

\newcommand{\floor}[1]{\lfloor#1\rfloor}

\newcommand{\Nk}{{\mathsf N}}

\def\cb#1#2#3{\left[\begin{array}{@{}c@{}}#1\\#2\end{array}\right]_{#3}}
\def\ctwo#1{{ \binom{#1}{2} }}

\newcommand{\be}{\begin{equation}}
\newcommand{\ee}{\end{equation}}
\newcommand{\ba}{\begin{eqnarray}}
\newcommand{\ea}{\end{eqnarray}}
\newcommand{\ha}{{\frac 12}}
\newcommand\Frac[2]{\genfrac{}{}{0pt}{}{#1}{#2}}

\newcommand{\vt}{\vartheta}
\newcommand{\vx}{\vec x}
\newcommand{\vlambda}{{\vec\lambda}}
\newcommand{\bC}{{\mathbb C}}
\newcommand{\bZ}{{\mathbb Z}}
\newcommand{\bN}{{\mathbb N}}
\def\SS{{\mathcal{H}_{\mathrm S}}}
\def\Bor{\mathcal{B}}

\def\glN{{\widehat{\mathfrak{gl}}_N}}
\def\sres#1{\mathsf{R}(#1)}

\newcommand{\ua}{a}
\newcommand{\vb}{b}
\newcommand{\wc}{c}
\newcommand{\tkappa}{ {\kappa} }
\newcommand{\tkappafactor}[1]{ {\kappa^{#1}} }
\newcommand{\fNekN}[7]{N^{{\rm Poch }(#1\vert #2)}_{#3#4}(#5\vert #6,#7)}
\newcommand{\fNekNs}[7]{\mathsf{N}^{(#1\vert #2)}_{#3#4}(#5\vert #6,#7)}
\newcommand{\NekZ}[2]{Z^{\rm Poch }_{#1}({#2})}
\newcommand{\NekZs}[2]{{\mathcal Z}_{#1}({#2})}
\newcommand{\NekZp}[2]{Z^{{\rm Poch, pure}}_{#1}({#2})}
\newcommand{\NekZps}[2]{{\mathcal Z}^{{\rm pure}}_{#1}({#2})}
\newcommand{\NekZsc}[3]{{\mathcal Z}_{#1, #2}({#3})}
\newcommand{\NekZpsc}[3]{{\mathcal Z}^{{\rm pure}}_{#1, #2}({#3})}

\newcommand{\ffactor}[4]{f^{(#1\vert #2)}_{#3,#4}} 
\newcommand{\gfactor}[5]{g^{(#1\vert #2)}_{#3,#4}\! \left(#5\right)}
\newcommand{\fweight}[2]{\vert #1\vert_{#2}}

\newcommand{\Kahlerm}[3]{Q^{#1}_{#2,#3}}
\newcommand{\dbar}[1]{ { \vmargin{.3ex}{0ex}{ \overline{d}_{#1} } } }

\makeatletter
\newcommand{\doublehat}[1]{%
\begingroup%
 \let\macc@kerna\z@%
 \let\macc@kernb\z@%
 \let\macc@nucleus\@empty%
 \hat{\mathchoice%
 {\raisebox{.3ex}{\vphantom{\ensuremath{\displaystyle #1}}}}%
 {\raisebox{.3ex}{\vphantom{\ensuremath{\textstyle #1}}}}%
 {\raisebox{.16ex}{\vphantom{\ensuremath{\scriptstyle #1}}}}%
 {\raisebox{.14ex}{\vphantom{\ensuremath{\scriptscriptstyle #1}}}}%
 \smash{\hat{#1}}}%
\endgroup%
}
\makeatother

\makeatletter
\def\vmargin@#1#2#3{
\setbox0=\hbox{#3}%
\rule[#1]{0pt}{\ht0}%
\lower\dp0\hbox{\rule[-#2]{0pt}{\dp0}}%
\box0%
}
\def\vmargin#1#2#3{
\mathchoice
{\vmargin@{#1}{#2}{$\displaystyle #3$}}
{\vmargin@{#1}{#2}{$\textstyle #3$}}
{\vmargin@{#1}{#2}{$\scriptstyle #3$}}
{\vmargin@{#1}{#2}{$\scriptscriptstyle #3$}}
}
\makeatother

\begin{document}

\allowdisplaybreaks

\newcommand{\arXivNumber}{2510.27142}

\renewcommand{\PaperNumber}{115}

\FirstPageHeading

\ShortArticleName{Non-Stationary Difference Equation and Affine Laumon Space III}

\ArticleName{Non-Stationary Difference Equation\\ and Affine Laumon Space III:\\
Generalization to $\boldsymbol{\widehat{\mathfrak{gl}}_N}$}

\Author{Hidetoshi~AWATA~$^{\rm a}$, Koji~HASEGAWA~$^{\rm b}$, Hiroaki~KANNO~$^{\rm ac}$,
Ryo~OHKAWA~$^{\rm d}$, \newline
Shamil~SHAKIROV~$^{\rm e}$, Jun'ichi~SHIRAISHI~$^{\rm f}$ and Yasuhiko~YAMADA~$^{\rm g}$}

\AuthorNameForHeading{H.~Awata, K.~Hasegawa, H.~Kanno, R.~Ohkawa, S.~Shakirov, J.~Shiraishi and Y.~Yamada}

\Address{$^{\rm a)}$~Graduate School of Mathematics, Nagoya University, Nagoya 464-8602, Japan}
\EmailD{\href{mailto:awata@math.nagoya-u.ac.jp}{awata@math.nagoya-u.ac.jp}}

\Address{$^{\rm b)}$~Mathematical Institute, Tohoku University, Sendai 980-8578, Japan}
\EmailD{\href{mailto:kojihas2@gmail.com}{kojihas2@gmail.com}}

\Address{$^{\rm c)}$~Kobayashi-Maskawa Institute,
Nagoya University, Nagoya 464-8602, Japan}
\EmailD{\href{mailto:kanno@math.nagoya-u.ac.jp}{kanno@math.nagoya-u.ac.jp}}

\Address{$^{\rm d)}$~Research Institute for Mathematical Sciences, Kyoto University, Kyoto 606-8502, Japan}
\EmailD{\href{mailto:ohkawa@kurims.kyoto-u.ac.jp}{ohkawa@kurims.kyoto-u.ac.jp}}

\Address{$^{\rm e)}$~Institute for Information Transmission Problems, Moscow, Russia}
\EmailD{\href{mailto:shakirov.work@gmail.com}{shakirov.work@gmail.com}}

\Address{$^{\rm f)}$~Graduate School of Mathematical Sciences, University of Tokyo,\\
\hphantom{$^{\rm h)}$}~Komaba, Tokyo 153-8914, Japan}
\EmailD{\href{mailto:shiraish@ms.u-tokyo.ac.jp}{shiraish@ms.u-tokyo.ac.jp}}

\Address{$^{\rm g)}$~Department of Mathematics, Kobe University, Rokko, Kobe 657-8501, Japan}
\EmailD{\href{mailto:yamaday@math.kobe-u.ac.jp}{yamaday@math.kobe-u.ac.jp}}

\ArticleDates{Received November 07, 2025, in final form October 23, 2026; Published online November 15, 2026}

\Abstract{In a series of papers, we have considered a non-stationary difference equation which was originally discovered for the deformed Virasoro conformal block. The equation involves mass parameters and, when they are tuned appropriately, the equation is regarded as a quantum KZ equation for \smash{$U_q\bigl(A_{1}^{(1)}\bigr)$}. We introduce a~$\widehat{\mathfrak{gl}}_N$ generalization of the non-stationary difference equation. The Hamiltonian is expressed in terms of $q$-commuting variables and allows both factorized forms and a normal ordered form. By specializing the mass parameters appropriately, the Hamiltonian can be identified with the $R$-matrix of the symmetric tensor representation of \smash{$U_q\bigl(A_{N-1}^{(1)}\bigr)$}, which in turn comes from the 3D (tetrahedron) $R$-matrix. We conjecture that the affine Laumon partition function of type \smash{$A_{N-1}^{(1)}$} gives a solution to our $\widehat{\mathfrak{gl}}_N$ non-stationary difference equation. As a check of our conjecture, we work out the four-dimensional limit and find that the non-stationary difference equation reduces to the Fuji--Suzuki--Tsuda system.}


\Classification{14H70; 33D80; 81R12}

\begin{flushright}
\begin{minipage}{80mm}
\it Dedicated to the memory of Masatoshi Noumi
\end{minipage}
\end{flushright}

\newpage

\section{Introduction}

In \cite{AHKOSSY1,AHKOSSY2}, we have explored various aspects
of the non-stationary difference equation;
\begin{equation}\label{Shakirov-eq}
\SS T_{qtQ,x}^{-1}T_{t,\Lambda}^{-1} \cdot \Psi(\Lambda,x)=\Psi(\Lambda,x),
\qquad \Psi(\Lambda,x) = \sum_{m,n \geq 0} c_{m,n} x^m (\Lambda/x)^n, \quad c_{0,0}=1,
\end{equation}
which was first introduced in \cite{Shakirov:2021krl}.
The Hamiltonian has mass parameters $d_i$ and is given by
\begin{align}
&\SS = \frac{1}{\varphi(qx)\varphi(\Lambda/x)} \cdot \Bor \cdot
\frac{\varphi(\Lambda)\varphi\bigl(q^{-1} d_1d_2d_3d_4\Lambda\bigr)}{\varphi(-d_1x)\varphi(-d_2x)\varphi(-d_3 \Lambda/x)\varphi(-d_4 \Lambda/x)}
\nonumber \\
&\hphantom{\SS =}{} \cdot \Bor\cdot \frac{1}{\varphi\bigl(q^{-1}d_1d_2x\bigr)\varphi(d_3d_4 \Lambda/x)},\label{Shakirov-Hamiltonian}
\end{align}
where $\varphi(z) := (z;q)_\infty$, $\Bor$ is the $q$-Borel transformation and $T_{\alpha,z}$ denotes the shift operator $z \to \alpha z$.
Other notations used throughout the paper are summarized in Section~\ref{notations}
at the end of the introduction.
The non-stationary difference equation \eqref{Shakirov-eq} is related to the quantized discrete Painlev\'e VI equation \cite{AHKOSSY1}.
Namely, the Hamiltonian \eqref{Shakirov-Hamiltonian} is equivalent
to the Hamiltonian of the discrete Painlev\'e VI equation given by \cite{Hasegawa},
in the sense that they have the same adjoint action on the canonical variables $(F,G)$ with $FG=q^{-1}GF$.
On the other hand, if we tune two of the mass parameters, say $d_2=q^{-m}$, $d_3=q^{-n}$, $m,n \in \mathbb{Z}_{\geq 0}$,
the equation~\eqref{Shakirov-eq}
can be also identified with the quantum Knizhnik--Zamolodchikov ($q$-KZ) equation for \smash{$U_q\bigl(\widehat{\mathfrak{gl}}_2\bigr)$} with generic spins.
Based on this fact we can prove that the $K$-theoretic Nekrasov partition function\footnote{There is
a variety of the $K$-theoretic Nekrasov partition functions on the affine Laumon space (see, for example,~\cite{Ohkawa-Shiraishi}).
Among them, we consider the partition function with fundamental matter multiplets.}
coming from the affine Laumon space provides a solution to the equation \eqref{Shakirov-eq} \cite{AHKOSSY2}.

In this paper, we propose a $\glN$ generalization of the non-stationary Hamiltonian~\eqref{Shakirov-Hamiltonian}.
For explicit expressions, see Definitions~\ref{def:glN Hamiltonian}--\ref{def:normal-ordered} below.
One of the significant differences from the $N=2$ case is that the arguments of $\varphi$ become $q$-commutative.
Let us introduce two sets of $q$-commutative variables
$(\hat{\mathsf{u}}_i,\check{\mathsf{u}}_i)~(i \in \mathbb{Z}/ N \mathbb{Z})$
with the following commutation relations;
\begin{equation}\label{q-com-1}
\hat{\mathsf{u}}_i \hat{\mathsf{u}}_{j} = q^{\delta_{i, j-1}- \delta_{i-1,j}} \hat{\mathsf{u}}_{j} \hat{\mathsf{u}}_{i},
\qquad \check{\mathsf{u}}_i \check{\mathsf{u}}_{j} = q^{\delta_{i-1,j}-\delta_{i, j-1}}\check{\mathsf{u}}_{j} \check{\mathsf{u}}_{i},
\end{equation}
and
\begin{equation}\label{q-com-2}
\hat{\mathsf{u}}_i \check{\mathsf{u}}_j = q^{2 \delta_{i,j} - \delta_{i,j+1} - \delta_{i, j-1}} \check{\mathsf{u}}_j \hat{\mathsf{u}}_i,
\end{equation}
where $\delta_{i,j}$ is the Kronecker delta modulo $N$.
Note that the matrix which appears in the power of $q$ is the Cartan matrix of \smash{${A}_{N-1}^{(1)}$}.
To write down the Hamiltonian of the non-stationary $\glN$ difference equation
with $N$ commutative variables $x_i~(i \in \mathbb{Z}/ N\mathbb{Z})$,
we employ the following representation of the algebra generated by $(\hat{\mathsf{u}}_i,\check{\mathsf{u}}_i)$;
\begin{equation}\label{x-rep}
\hat{x}_i := \alpha_i x_i q^{\vartheta_i - \vartheta_{i-1}}, \qquad \check{x}_i := \beta_i x_i q^{-\vartheta_i + \vartheta_{i-1}},
\end{equation}
where $\vartheta_i := x_i \frac{\partial}{\partial x_i}$ and $\alpha_i$, $\beta_i$ are arbitrary scaling parameters.
Since the index of $x_i$ is in $\mathbb{Z}/ N\mathbb{Z}$, we will identify $x_0$ with $x_N$ throughout the paper.
From $p_i x_j = q^{\delta_{i,j}} x_j p_i$ with $p_i := q^{\vartheta_i}$,
we see that $\hat{x}_i$ and $\check{x}_i$ satisfy the commutation relations~\eqref{q-com-1} and \eqref{q-com-2}.
The non-stationary $\glN$ Hamiltonian has $3N$ parameters $b_i$, $d_i$, $\overline{d}_i$ $(i \in \mathbb{Z}/ N\mathbb{Z})$.
It also involves the quantum deformation parameter $q$ and the shift parameter $t^{-1}= \kappa^{N}$.
In the supersymmetric gauge theory, $x_i$ are instanton expansion parameters, $b_i$ are Coulomb moduli and
$(d_i,\overline{d}_i)$ are mass parameters. The equivariant parameters $(q,t)$ come from
the torus action $(z_1, z_2) \to (q z_1, \kappa z_2)$ on~$\mathbb{C}^2$.

\begin{Definition}[non-stationary $\glN$ Hamiltonian]\label{def:glN Hamiltonian}
Let
\begin{equation}\label{dfn:Delta}
\Delta :=\displaystyle{\sum_{i=1}^N} (\vartheta_i^2 - \vartheta_i \vartheta_{i-1})
= \frac{1}{2}\displaystyle{\sum_{i=1}^N} (\vartheta_i - \vartheta_{i-1})^2.
\end{equation}
We define
\begin{equation}\label{eq:glN Hamiltonian}
\mathcal{H}^{\glN}(x_i; b_i, d_i, \overline{d}_i, q, \kappa)
= q^{\frac{1}{2}\Delta} \cdot \mathcal{A}_L \cdot \mathcal{A}_C \cdot \mathcal{A}_R
\cdot q^{\frac{1}{2}\Delta} \cdot \mathsf{T}, \\
\end{equation}
where
\begin{equation}
\mathsf{T} := \prod_{i=1}^N T_{\frac{\kappa b_i}{b_{i+1}}, x_i}.
\end{equation}
The shift operator $\mathsf{T}$ acts on $x_i$ by $x_i \to \frac{\kappa b_i}{b_{i+1}}x_i$ and
hence on $\Lambda:= x_1 x_2 \cdots x_N$ by $\Lambda \to \kappa^N \Lambda = t^{-1} \Lambda$.
The middle block of the Hamiltonian is defined by
\begin{equation}
\mathcal{A}_C := \prod_{k=1}^{N} \frac{1}{\varphi(d_k x_k)\varphi(\overline{d}_k x_k)},
\end{equation}
where $\varphi(z) := (z;q)_\infty$. Other blocks $\mathcal{A}_L$ and $\mathcal{A}_R$ are
given by Definitions \ref{factorized1}, \ref{factorized2} and \ref{def:normal-ordered} below.
\end{Definition}

There are three equivalent definitions of $\mathcal{A}_L= \mathcal{A}_L^{(i)}$
and $\mathcal{A}_R=\mathcal{A}_R^{(i)}$ with $i = \mathrm{s}, \mathrm{h}, \mathrm{n}$,
which is one of the remarkable consequences of the fact that $\mathcal{A}_L$ and $\mathcal{A}_R$ involve
the $q$-commutative variables $\hat{x}_i$ and $\check{x}_i$, respectively.
To define $\mathcal{A}_L$ and $\mathcal{A}_R$,
we choose the scaling parameters of $\hat{x}_i$ and $\check{x}_i$ as $\alpha_i = d_i \overline{d}_i$ and $\beta_i =1$
and set \smash{$D_N := \prod_{k=1}^{N} d_k \overline{d}_k$}.

\begin{Definition}[factorized form of simple root type]\label{factorized1}
\begin{align}
\label{eq:A_L}
&\mathcal{A}_L^{(\mathrm{s})}:= \frac{1}{G_L(\check{x})}\frac{1}{\varphi(-\check{x}_{0})}G_L(\check{x})
\frac{1}{\varphi(-\check{x}_{N-1})} \cdots \frac{1}{\varphi(-\check{x}_{2})} \frac{1}{\varphi(-\check{x}_{1})} \varphi(\Lambda),
\\
&\mathcal{A}_R^{(\mathrm{s})}:= \varphi\bigl(q^{1-N} D_N \Lambda\bigr) \frac{1}{\varphi(-\hat{x}_1)}
\frac{1}{\varphi(-\hat{x}_2)} \cdots \frac{1}{\varphi(-\hat{x}_{N-1})}
G_R(\hat{x}) \frac{1}{\varphi(-\hat{x}_0)} \frac{1}{G_R(\hat{x})},
\label{eq:A_R}
\end{align}
where $G_L(\check{x}) := \varphi(-\check{x}_{1}) \cdots \varphi(-\check{x}_{N-2})$ and
$G_R(\hat{x}) := \varphi(-\hat{x}_{N-2}) \cdots \varphi(-\hat{x}_1)$ are twisting factors by the adjoint action.
\end{Definition}

\begin{Definition}[factorized form of higher root type]\label{factorized2}
\begin{align}
\label{eq:A_L-reduction}
&\mathcal{A}_L^{(\mathrm{h})}:= e_q(-\check{x}_0) e_q(-\check{x}_0\check{x}_1) \cdots e_q(-\check{x}_0 \cdots \check{x}_{N-2}) \cdot
e_q(-\check{x}_{N-1}) \cdots e_q(-\check{x}_{1}) \cdot \varphi(\Lambda), \\
&\mathcal{A}_R^{(\mathrm{h})}:= \varphi\bigl(q^{1-N} D_N \Lambda\bigr) \cdot
e_q(-\hat{x}_1) \cdots e_q(-\hat{x}_{N-1}) \nonumber\\
& \hphantom{\mathcal{A}_R^{(\mathrm{h})}:=}{} \cdot e_q(-\hat{x}_{N-2} \cdots \hat{x}_0) \cdots
e_q(-\hat{x}_1 \hat{x}_0) e_q(-\hat{x}_0),
\label{eq:A_R-reduction}
\end{align}
where $e_q(z) = \varphi(z)^{-1}$ denotes the $q$-exponential function (see Section~\ref{notations}).
\end{Definition}
\begin{Definition}[normal ordered form]\label{def:normal-ordered}
\begin{equation}\label{eq:normal-ordered}
(\mathcal{A}_L^{(\mathrm{n})})^{-1} := {:}\prod_{i=1}^N \frac{1}{\varphi(\check{x}_i)}{:}, \qquad
\mathcal{A}_R^{(\mathrm{n})}:= {:}\prod_{i=1}^N \varphi(\hat{x}_i){:},
\end{equation}
where $: \ :$ denotes the normal ordering, which is a linear operator acting on a formal power series in $x_i$.
It is defined as follows:
for any analytic function $F(x,\theta)$ in $2N$ commutative variables $x=\{x_i\}$, $\theta=\{\theta_i\}$,
we first define ${:}F(x,\theta){:}$ by the following action on a monomial $x^{\nu}=\prod_{i=1}^N x_i^{\nu_i}$:%
\begin{equation}\label{def:normal-ordering}
{:}F(x,\theta){:} x^\nu=F(x,\nu)x^\nu,
\end{equation}
and extend it linearly to the space of formal power series.
We call the symbol ${:}\bullet{:}$ normal ordering. By an abuse of notation
the Euler derivatives $\vartheta_i$ are used for the commutative variables $\theta_i$
inside the normal ordering symbol ${:}\bullet{:}$,
for example,\begin{equation*}
{:}q^{\vartheta_i} x_i{:} x^\nu ={:}x_i q^{\vartheta_i}{:} x^\nu = x_i q^{\nu_i} \cdot x^\nu.
\end{equation*}
\end{Definition}

In Section~\ref{sec:2}, we will prove the equivalence of three forms of the Hamiltonian.
In Section~\ref{Dynkin-auto}, we show the pentagon identity implies the equivalence of
two factorized forms of the Hamiltonian; \smash{$\mathcal{A}_L^{(\mathrm{s})}=\mathcal{A}_L^{(\mathrm{h})}$}
and \smash{$\mathcal{A}_R^{(\mathrm{s})}=\mathcal{A}_R^{(\mathrm{h})}$}.
On the other hand, in Section~\ref{subsec:normal}, we employ the $q$-binomial theorem
to prove the equivalence to the normal ordered Hamiltonian: \smash{$\mathcal{A}_L^{(\mathrm{h})}=\mathcal{A}_L^{(\mathrm{n})}$}
and \smash{$\mathcal{A}_R^{(\mathrm{h})}=\mathcal{A}_R^{(\mathrm{n})}$}.
These equalities among the Hamiltonians mean the identities as linear operators
on the space of formal power series in~$x_i$.

\subsection{Several remarks}

{\bf 1.1.1.}
The arguments of the middle block $\mathcal{A}_C$ are commutative variables $x_i$.
We note that $\hat{x}_1 \cdots \hat{x}_N = q^{-1} D_N \Lambda$, $\check{x}_1 \cdots \check{x}_N = q\Lambda$, $\hat{x}_N \cdots \hat{x}_1 = q^{1-N} D_N \Lambda$ and
$\check{x}_N \cdots \check{x}_1 = q^{N-1} \Lambda$ are central elements in the algebra.

{\bf 1.1.2.}
Compared with the $\widehat{\mathfrak{gl}}_2$ Hamiltonian~\eqref{Shakirov-Hamiltonian},
the $q$-Borel transformation \smash{$q^{\frac{1}{2}\Delta}$} is moved
to both of the end positions. See Appendix~\ref{App-A} for the agreement of the $N=2$ case
of~\eqref{eq:glN Hamiltonian} and~\eqref{Shakirov-Hamiltonian}. In $\glN$ case,
if \smash{$q^{\frac{1}{2}\Delta}$} is put between the blocks $\mathcal{A}_i$,
the formulas for $\mathcal{A}_i$ will become more involved (see Proposition~\ref{Prop 2.10}).

{\bf 1.1.3.}
In terms of the Hamiltonian of normal ordered form, the Schr\"odinger equation for the wave function $\psi$ can be written
the following way:
\begin{align}
&\bigl(\mathcal{A}_L^{(\mathrm{n})}\bigr)^{-1} q^{-\frac{1}{2}\Delta} \psi
= \mathcal{A}_C \cdot \mathcal{A}_R^{(\mathrm{n})} \cdot q^{\frac{1}{2}\Delta} \mathsf{T} \psi,
\\
&\psi= \sum_{\theta_1, \dots, \theta_N =0}^\infty
c_{\theta_1, \dots, \theta_N} x_1^{\theta_1} \cdots x_N^{\theta_N}, \qquad c_{0,\dots,0}=1.
\end{align}
By the gauge transformation of the form $\psi \to \prod_{i} x_i^{\beta_i}\cdot \psi$
with an appropriate scaling of $x_i$ we can eliminate the parameters $b_i$ in the shift operator $\mathsf{T}$
so that the dependence on $b_i$ only appears in the wave function $\psi$. See \cite[Remark 2.9]{AHKOSSY1}
for an explicit example in the case $N=2$.

{\bf 1.1.4.}
In Definition~\ref{factorized1} of simple root type,
the arguments $x_0, x_1, \dots, x_{N-1}$ of the function~$\varphi(z)$ correspond
to the simple roots of the affine algebra \smash{$A_{N-1}^{(1)}$} and $\Lambda := x_0x_1 \cdots x_{N-1}$ corresponds to the null root.
On the other hand in Definition~\ref{factorized2} of higher root type,
the Hamiltonian involves the $q$-exponential factors corresponding to the higher roots.
Note that~\smash{$\mathcal{A}_L^{(\mathrm{h})}$} and~\smash{$\mathcal{A}_R^{(\mathrm{h})}$}
involve $N-1$ $q$-exponentials which correspond to the simple roots of~$A_{N-1}$
and $N-1$ $q$-exponentials with variables for higher roots of the affine algebra.
The factorized Hamiltonian of higher root type is more convenient to see the relation
to the universal $R$-matrix of \smash{$U_q\bigl(A_{N-1}^{(1)}\bigr)$}~\cite{AHKOSSY}.

{\bf 1.1.5.}
An interesting feature of the Hamiltonian of simple root type is
that the factor corresponding to the last variable $x_N=x_0$ is
twisted by the adjoint action of $G_L(\check{x}_i)$ or $G_R(\hat{x}_i)$,
which is the product of $\varphi$ with variables $x_1, \dots, x_{N-2}$.\footnote{When $N=2$ this is empty.}
In contrast to the non-affine $\mathfrak{gl}_N$ case,
the cyclic symmetry of the Hamiltonian in $x_1, \dots, x_N$ is required for the affine $\glN$ case,
which is non-trivial, since the Hamiltonian involves the $q$-commuting variables.
We note that the twisting guarantees the desired cyclic symmetry of the Hamiltonian.
In Section~\ref{classical-AL-AR}, we give
a classical analogue of the twisting in \smash{$\mathcal{A}_L^{(\mathrm{s})}$} and \smash{$\mathcal{A}_R^{(\mathrm{s})}$}.
More generally, due to the pentagon identity for $\varphi(z)^{-1}$ with $q$-commutative variables,
the Hamiltonian~\smash{$\mathcal{H}^{\glN}$} is invariant under the automorphism of the Dynkin diagram of $A_{N-1}^{(1)}$
(see Corollary~\ref{Cor2.8}).

{\bf 1.1.6.}
The $q$-commuting variables $\hat{x}_i$ and $\check{x}_i$ appear in the arguments of the $q$-exponential function $\varphi(z)= e_q(z)^{-1}$.
By using
\begin{equation}
\operatorname{Ad} \bigl( q^{\frac{1}{2}(\vartheta_i - \vartheta_{i-1})^2}\bigr) (\alpha_i x_i)^n = q^{\frac{n}{2}} (\hat{x}_i)^n,
\qquad
\operatorname{Ad} \bigl( q^{-\frac{1}{2}(\vartheta_i - \vartheta_{i-1})^2}\bigr) (\beta_i x_i)^n = q^{-\frac{n}{2}} (\check{x}_i)^n,
\end{equation}
(see \eqref{adjoint-onpower}), we can replace the $q$-exponential functions with $q$-commuting variables by
$\varphi(x_i)= e_q(x_i)^{-1}$ with commuting variables $x_i$.
For example, when $N=3$, the left block \smash{$\mathcal{A}_L^{(\mathrm{s})}$} can be written
as follows:\footnote{The arguments of the left block $\mathcal{A}_L$ are $\check{x}_i$ and we have set $\beta_i=1$
in the definition of $\mathcal{A}_L$.}
\begin{align}
& q^{\frac{1}{2}\Delta} \cdot \mathcal{A}_L^{(\mathrm{s})} \cdot \varphi(\Lambda)^{-1}
 = q^{\frac{1}{2}(-\vartheta_1 \vartheta_2 + \vartheta_1 \vartheta_3 + \vartheta_2 \vartheta_3)}
\varphi\bigl(q^{\frac{1}{2}}x_1\bigr)^{-1} q^{-\vartheta_1 \vartheta_3} \varphi\bigl(q^{\frac{1}{2}}x_3\bigr)^{-1}
q^{\vartheta_1 \vartheta_3 -\vartheta_2 \vartheta_3} \varphi\bigl(q^{\frac{1}{2}}x_1\bigr) \nonumber\\
& \hphantom{q^{\frac{1}{2}\Delta} \cdot \mathcal{A}_L^{(\mathrm{s})} \cdot \varphi(\Lambda)^{-1} =}{} \cdot q^{\vartheta_1\vartheta_2} \varphi\bigl(q^{\frac{1}{2}}x_2\bigr)^{-1} q^{-\vartheta_1 \vartheta_2} \varphi\bigl(q^{\frac{1}{2}}x_1\bigr)^{-1}
q^{\frac{1}{2}(-\vartheta_1 \vartheta_3 + \vartheta_1 \vartheta_2 + \vartheta_2 \vartheta_3)} \cdot q^{\frac{1}{2}\Delta}.\label{eq:A_Lcom}
\end{align}
Thus, the expense of eliminating $q$-commuting variables from the arguments of $\varphi(z)$ is
the scattered insertion of the operators of the form $q^{\hbox{\tiny quadratic in}~\vartheta_i}$
between the $q$-exponential functions.
Note that the position of the $q$-Borel transformation \smash{$q^{\frac{1}{2}\Delta}$} is changed
from the left of \smash{$\varphi(\hat{x}_i)$} to the right of \smash{$\varphi\bigl(q^{\frac{1}{2}}x_i\bigr)$} with commuting variables~$x_i$.
For general $\glN$ case, see Section~\ref{subsec:other-form}.

\subsection{Affine Laumon partition function: Conjecture}

In \cite{AHKOSSY2}, we proved that the affine Laumon partition function of type~\smash{$A_1^{(1)}$}
provides a solution to the non-stationary difference equation~\eqref{Shakirov-eq}.
In general, the affine Laumon partition function of type \smash{$A_{N-1}^{(1)}$} is defined as follows.

\begin{Definition}[affine Laumon partition function]\label{def:affine-Laumon}
The affine Laumon partition function of type $A_{N-1}^{(1)}$ is
a summation over $N$-tuples of partitions $\vec{\lambda} = \bigl(\lambda^{(1)}, \dots, \lambda^{(N)}\bigr)$:
\begin{align*}
&\mathcal{Z}_{\mathrm{AL}}^{\glN}
\left( \left.\left.\begin{array}{@{}ccc@{}}a_1,\dots, a_N \\b_1, \dots ,b_N\\c_1,\dots, c_N\end{array}
\right| \mathsf{x}_1, \dots, \mathsf{x}_N \right|q, \kappa\right) \\
&\qquad{} = \sum_{\vec{\lambda}}
\prod_{i,j=1}^N \frac{\Nk^{(j-i\vert N)}_{\varnothing,\lambda^{(j)}}(a_i/b_j |q, \kappa)
\Nk^{(j-i\vert N)}_{\lambda^{(i)},\varnothing}(b_i/c_j|q, \kappa)}
{\Nk^{(j-i\vert N)}_{\lambda^{(i)},\lambda^{(j)}}(b_i/b_j |q, \kappa)} \cdot \mathsf{x}_1^{k_1(\vec{\lambda})} \cdots \mathsf{x}_N^{k_N(\vec{\lambda})},
\end{align*}
where \smash{$\mathsf{N}_{\lambda, \mu}^{(k\vert N)} (u \vert q, \kappa)$} is
the orbifolded Nekrasov factor with color $k$ (see \cite[Definition 6.3]{AHKOSSY1}):
\begin{align*}
&\Nk^{(k|N)}_{\lambda,\mu}(u|q,\kappa)=
\Nk^{(k)}_{\lambda,\mu}(u|q,\kappa)\\
&\quad=
 \prod_{\substack{j\geq i\geq 1 \\ j-i \equiv k \,({\rm mod}\,N)}}
\bigl[u q^{-\mu_i+\lambda_{j+1}} \kappa^{-i+j};q\bigr]_{\lambda_j-\lambda_{j+1}}
\cdot\!\!\!
\prod_{\substack{\beta\geq \alpha \geq 1 \\ \beta-\alpha \equiv -k-1 \,({\rm mod}\,N)}}\!\!\!
\bigl[u q^{\lambda_{\alpha}-\mu_\beta} \kappa^{\alpha-\beta-1};q\bigr]_{\mu_{\beta}-\mu_{\beta+1}},
\end{align*}
with
\begin{align*}
[u;q]_n&=u^{-n/2}q^{-n(n-1)/4} (u;q)_n\\
&= \bigl(u^{-1/2}-u^{1/2}\bigr)\bigl(q^{-1/2}u^{-1/2}-q^{1/2}u^{1/2}\bigr)\cdots
\bigl(q^{-(n-1)/2}u^{-1/2}-q^{(n-1)/2}u^{1/2}\bigr).
\end{align*}
The powers of the expansion parameters $\mathsf{x}_i$ are given by the number of boxes with a fixed color:%
\begin{equation}\label{box-counting}
k_i(\vec{\lambda}) = \sum_{\alpha + \beta \equiv i+1} \bigl\vert \lambda^{(\alpha)}\bigr\vert_{\beta}, \qquad
\bigl\vert \lambda^{(\alpha)} \bigr\vert_\beta := \sum_{k \in \mathbb{Z}} \lambda^{(\alpha)}_{\beta+Nk},
\end{equation}
where we denote the components of $\lambda^{(\alpha)}$ by
\smash{$\bigl(\lambda^{(\alpha)}_1 \geq \lambda^{(\alpha)}_2 \geq \cdots\bigr)$} and set
\smash{$\lambda^{(\alpha)}_{i}=0$} for $i \leq 0$.
\end{Definition}

Now we are ready to present our main claim in this paper.

\begin{con}\label{con:1.6}
The affine Laumon partition function provides a solution to the non-stationary difference equation:
\begin{gather}
\mathcal{H}^{\glN}(x_i; b_i, d_i, \overline{d}_i, q, \kappa) \psi = \psi, \qquad \psi= \sum_{\theta_1, \dots, \theta_N =0}^\infty
c_{\theta_1, \dots, \theta_N} x_1^{\theta_1} \cdots x_N^{\theta_N}, \qquad c_{0,\dots,0}=1,
\label{eq:EigenEq}
\end{gather}
where $\psi$ is the $\glN$ Laumon partition function in the following parametrization:
\[
\psi = \mathcal{Z}_{\mathrm{AL}}^{\glN}
\left( \left.\left.\begin{array}{@{}ccc@{}} \frac{q\kappa b_N}{d_N}, \frac{q\kappa b_1}{d_1}, \dots, \frac{q\kappa b_{N-1}}{d_{N-1}}
\\b_1, b_2, \dots, b_N \\ \frac{b_1}{\vmargin{.3ex}{0ex}{\overline{d}_{1}}},
\frac{b_2}{\vmargin{.3ex}{0ex}{\overline{d}_{2}}}, \dots, \frac{b_N}{\vmargin{.3ex}{0ex}{\overline{d}_{N}}} \end{array}
\right| \sqrt{\frac{b_2d_1 \overline{d}_{1}}{q \kappa b_1}}x_1, \sqrt{\frac{b_3d_2 \overline{d}_{2}}{q \kappa b_2}}x_2 ,
\dots, \sqrt{\frac{b_1d_N \overline{d}_{N}}{q \kappa b_N}}x_N \right| q, \kappa \right).
\]
\end{con}

Firstly, let us make a remark on the cases $N=1,2$.
The ${\widehat{\mathfrak{gl}}_2}$ case of the conjecture was proved in~\cite{AHKOSSY2}.
We can see it is also valid for $N=1$ as follows:
Dropping the indices, we simply write
$x=x_1$, $d=d_1$, $\overline{d}=\overline{d}_1$, etc.
The Hamiltonian is simplified to
\begin{equation}
\mathcal{H}^{\widehat{\mathfrak{gl}}_1}= \frac{ \varphi(x)\varphi(d\overline{d}x)}{\varphi(d x)\varphi(\overline{d}x)}T_{\kappa,x}
=\exp\left(- \sum_{n=1}^\infty \frac{1}{n}
\frac{(1-d^n)(1-\overline{d}^n)}{(1-q^n)} x^n\right)T_{\kappa,x},
\end{equation}
since
\[
\Delta=0, \qquad \mathsf{T}= T_{\kappa,x} ,\qquad
 \mathcal{A}_C=\frac{1}{\varphi(dx)\varphi(\overline{d}x)} ,\qquad
 \mathcal{A}_L=\varphi(x) , \qquad \mathcal{A}_R=\varphi(d\overline{d}x).
 \]
Hence, it is easy to see that the equation and the solution read
\begin{equation}\label{gl1case}
\mathcal{H}^{\widehat{\mathfrak{gl}}_1} \psi=\psi,\qquad
\psi=\exp\left(- \sum_{n=1}^\infty \frac{1}{n}
\frac{(1-d^n)(1-\overline{d}^n)}{(1-q^n)(1-\kappa^n)} x^n\right).
\end{equation}
On the other hand, we have an impressive (double infinite product) expression for the \smash{${\widehat{\mathfrak{gl}}_1}$}
affine Laumon partition function
\begin{equation}\label{gl1partition-function}
\mathcal{Z}^{\widehat{\mathfrak{gl}}_1}_{\rm AL}
\left( \begin{array}{@{\,}c@{\,}} a \\ b\\c \end{array} \Biggl|\, \mathsf{x}\, \Biggl|q,\kappa\right)=
\exp\left( \sum_{n=1}^\infty \frac{1}{n}
\frac{[b^n/c^n][a^n/q^n\kappa^n b^n]}{[q^n][\kappa^n]}\mathsf{x}^n\right),
\end{equation}
where we used the symbol $[x]=x^{-1/2}-x^{1/2}$.
As for a proof of the identity \eqref{gl1partition-function}, see, e.g., \cite[Proposition 4.17]{Ohkawa-Shiraishi}.
Note that in idem., the Nekrasov partition function is defined by the
ordinary Pochhammer symbol $(a;q)_n$ (as equation~(72) in idem.), instead of the shifted product of
hyperbolic sine functions $[a;q]_n=[a][qa]\cdots \bigl[q^{n-1}a\bigr]$ defined in Definition~1.5.
Comparison of these is achieved by applying Proposition~B.1.
We conclude that the solution $\psi$ to the equation $\mathcal{H}^{\widehat{\mathfrak{gl}}_1} \psi=\psi$
is given by the affine Laumon function as
\begin{equation}
\mathcal{Z}^{\widehat{\mathfrak{gl}}_1}_{\rm AL}
\left( \begin{array}{@{\,}c@{\,}} q\kappa b/d \\ b\\b/\overline{d}\end{array} \Biggl|\,
\sqrt{\frac{d\overline{d}}{q\kappa}}
x\, \Biggl|q,\kappa\right)=
\exp\left(- \sum_{n=1}^\infty \frac{1}{n}
\frac{(1-d^n)(1-\overline{d}^n)}{(1-q^n)(1-\kappa^n)} x^n\right)=\psi.
\end{equation}

Secondly, some remarks are in order concerning the verification of Conjecture~\ref{con:1.6} for $N \geq 3$.
For simplicity, we set
\[
{\mathcal A}_L^{-1} q^{-\frac{1}{2}\Delta}\psi-
{\mathcal A}_C {\mathcal A}_R q^{\frac{1}{2}\Delta}\mathsf{T}\psi
=\sum_{\theta_1,\dots,\theta_N =0}^{\infty} e_{\theta_1,\dots,\theta_N}x_1^{\theta_1} \cdots x_N^{\theta_N}.
\]
By straightforward brute-force calculations using {\sc Mathematica},
we can show that $ e_{\theta_1,\dots,\theta_N}=0$ at least for small $N$ and
lower total degree \smash{$\mathsf{d}=\sum_{i=1}^{N} \theta_i$}.
For $N=3$, we have checked that $e_{\theta_1,\theta_2,\theta_3}=0$ for $0\leq \mathsf{d}\leq 5$,
treating all eleven parameters $d_i$, $\overline{d}_i$, $b_i$ $(i=1,2,3)$, $q$ and $\kappa$ as indeterminates.
For $6 \leq \mathsf{d}$, however, such a task becomes more computationally demanding and practically intractable.
Instead, substituting random integers for all eleven parameters,
we observe that $e_{\theta_1,\theta_2,\theta_3}=0$ for $6\leq \mathsf{d}\leq 15$, which means up to instanton number five.
Similarly, for $N=4$, we find that $e_{\theta_1,\theta_2,\theta_3,\theta_4}=0$ for $0 \leq \mathsf{d}\leq 9$,
when we specialize the fourteen parameters in the same manner; for $N=5$
the same method works for $0 \leq \mathsf{d}\leq 7$.

\subsection[Mass truncation and relation to the R matrix]{Mass truncation and relation to the $\boldsymbol{R}$ matrix}

In ${\widehat{\mathfrak{gl}}_2}$ case \cite{AHKOSSY2}, we introduced the mass parameter truncation where half of the mass parameters are
set to the form $q^{-n}~(n \in \mathbb{Z}_{\geq 0})$. After the mass parameter truncation the non-stationary difference equation
\eqref{Shakirov-eq} is identified with the quantum KZ equation. Namely, if we remove the shift operator $\mathsf{T}$ from the Hamiltonian,
it gives the (finite-dimensional) $R$-matrix of \smash{$U_q\bigl(A_1^{(1)}\bigr)$} with generic spins.
Based on the normal ordered Hamiltonian~\eqref{def:normal-ordered}, we can show the same story for $\glN$ case.
It is quite remarkable the resulting finite-dimensional $R$-matrix of \smash{$U_q\bigl(A_{N-1}^{(1)}\bigr)$} is related the three-dimensional (tetrahedron) $R$-matrix~\cite{Kuniba-book}.
In the formula of the components of the three-dimensional $R$-matrix there appears a basic building block $\Phi_q$ defined by \eqref{def:Phi_q}
(see \cite[Section~13.5]{Kuniba-book}).
We find the same function in our formula of the components of the $R$-matrix (see Corollary~\ref{Cor:R-component}).\looseness=1

\subsection{Four-dimensional limit}
The four-dimensional (cohomological) version of the affine Laumon partition function of type~${\widehat{\mathfrak{gl}}_2}$
satisfies a quantization of the differential Painlev\'e equation $P_{\rm VI}$ (see, e.g.,~\cite{AFKMY} and references therein).
In \cite{AHKOSSY1}, we have seen how the non-stationary difference equation~\eqref{Shakirov-eq} provides a way to
up grade the story to five-dimensional/$q$-difference version.
In four-dimensional/differential situation, the generalization to~$\glN$ case was also considered in \cite{Yamada:2010rr}
where a quantization of a particular kind of higher-rank generalization of $P_{\rm VI}$ (called Fuji--Suzuki--Tsuda system)
was studied as the relevant equation.
One can check that the four-dimensional limit of Conjecture~\ref{con:1.6} is consistent with the result in~\cite{Yamada:2010rr}.

\subsection{Organization of the paper}
The present paper is organized as follows:
In Section~\ref{sec:2}, by using the pentagon identity for the $q$-exponential function $e_q(z) = \varphi(z)^{-1}$
and the $q$-binomial theorem we prove that there are three equivalence forms of the $\glN$ Hamiltonian~\eqref{eq:glN Hamiltonian};
two kinds of the factorized form and the normal ordered form. Each form has its own advantage.
We also show that the Hamiltonian is invariant under the action of the Dynkin automorphisms of \smash{$A_{N-1}^{(1)}$}.
We consider the mass truncation in Section~\ref{sec:mass-truncation}. Namely, we tune half of the mass parameters in the Hamiltonian so that we can
extract finite-dimensional blocks of the $R$-matrix. We find an interesting relation to the tetrahedron (3D) $R$-matrix.
In \cite{AHKOSSY2}, Conjecture~\ref{con:1.6} for $N=2$ was proved by rewriting
the affine Laumon partition function in the form of the Jackson integral. In Section~\ref{sec:affine-Laumon}, we show that
it is straightforward to apply the same method to $N \geq 3$, though it does not lead to
a~proof of Conjecture~\ref{con:1.6} at the moment.
Finally, in Section~\ref{sec:4d-limit}, we show that a~four-dimensional limit of our system is nothing but the Fuji--Suzuki--Tsuda system, which is consistent with the conjecture. Some of technical details and miscellaneous topics are collected in appendices.\looseness=1

\subsection{Notations and convention}\label{notations}
We will use the following notations throughout the paper~\cite{GR}:
\begin{equation}
\label{inf-prod}
\varphi(x):= (x;q)_\infty = \prod_{n=0}^\infty (1-x q^n)
=\exp\left(-\sum_{n=1}^\infty \frac{1}{n} \frac{1}{1-q^n} x^n \right), \qquad |x| < 1, \quad |q|<1.
\end{equation}
The $q$-shifted factorial is defined by
\begin{equation}\label{q-factorial}
(x;q)_n = \frac{(x;q)_\infty}{(xq^n ; q)_\infty}.
\end{equation}
The following formula is useful:
\begin{equation}\label{q-factorial-inversion}
(x;q)_n=(-x)^n q^{\frac{1}{2}n(n-1)}\frac{1}{(qx^{-1};q)_{-n}}, \qquad n \in \mathbb{Z}.
\end{equation}

We employ the formulas of two $q$-exponential functions \cite{GR}:
\begin{align}
\label{q-exp1}
&e_q(z)= \sum_{n=0}^\infty \frac{z^n}{(q;q)_n} = \varphi(z)^{-1}, \qquad |z| <1, \quad |q| <1, \\
\label{q-exp2}
&E_q(z)= \sum_{n=0}^\infty \frac{q^{\frac{1}{2}n(n-1)}z^n}{(q;q)_n} = \varphi(-z), \qquad |z| <1, \quad |q| <1.
\end{align}
Finally, the $q$-binomial coefficients are defined by
\begin{equation}\label{q-binom}
\cb{n}{k}{q}:= \frac{(q;q)_n}{(q;q)_{k} (q;q)_{n-k}}.
\end{equation}
We have
\begin{equation}
(x;q)_n = \sum_{k=0}^n \cb{n}{k}{q} (-x)^k q^{\frac{1}{2}k(k-1)}.
\end{equation}

The partition function on the gauge theory side is computed by the localization for the torus action.
On the four-dimensional space-time the action is $\mathbb{R}^4 \simeq \mathbb{C}^2 \ni (z_1, z_2) \longrightarrow (q_1z_1, q_2z_2)$.
In this paper, we regard the equivariant parameters\footnote{The factor $1/N$ in the definition of $q_2$ comes from the $\mathbb{Z}_N$ orbifold
action on $z_2$, which is an effective way of introducing a surface defect at the divisor $z_2=0$.}
\begin{equation}
q_1 := e^{\epsilon_1}, \qquad q_2 := \kappa = t^{-\frac{1}{N}} = e^{\frac{\epsilon_2}{N}},
\end{equation}
as the canonical parameters of the theory. They are natural parameters of the quantum toroidal algebras.
We simply denote $q=q_1$ unless otherwise mentioned.

\section[Non-stationary widehat\{gl\}\_N difference equation]{Non-stationary $\boldsymbol{\glN}$ difference equation}
\label{sec:2}
\subsection[Pentagon identity and Dynkin automorphisms of A\_\{N-1\}\^{}\{(1)\}]{Pentagon identity and Dynkin automorphisms of $\boldsymbol{A_{N-1}^{(1)}}$}
\label{Dynkin-auto}

By using the pentagon identity for the $q$-exponential function $e_q(z)= \varphi(z)^{-1}$,
we can recast the~blocks~\smash{$ \mathcal{A}_L^{(\mathrm{s})}$} and \smash{$\mathcal{A}_R^{(\mathrm{s})}$}
of the Hamiltonian~\eqref{eq:glN Hamiltonian} of factorized form of simple root type
so that the correspondence to the factorization of the universal $R$-matrix becomes clear~\cite{AHKOSSY}.
The pentagon identity also allows us to see that the $\glN$ Hamiltonian is actually symmetric in variables $x_i$.

\begin{Proposition}[\cite{Kirillov}]
For $q$-commutative variables $a$, $b$ with $ab=q ba$,
the $q$-exponential function $e_q(z)= \varphi(z)^{-1}$ satisfies the pentagon identity:
\begin{equation}\label{pentagon}
e_q(-a) e_q(-b) = e_q(-b) e_q(-ba) e_q(-a).
\end{equation}
\end{Proposition}

Since the $q$-commutative variables $\hat{x}_i$ and $\check{x}_i$
satisfy $\hat{x}_i \hat{x}_j = q^{\delta_{i,j-1}-\delta_{i-1,j}}\hat{x}_j \hat{x}_i$ and
$\check{x}_i \check{x}_j = q^{\delta_{i-1,j}-\delta_{i,j-1}}\check{x}_j \check{x}_i$,
we obtain
\begin{align}
\label{pentagon1}
&e_q(-\hat{x}_i) e_q(-\hat{x}_{i+1})
= e_q(-\hat{x}_{i+1}) e_q(-\hat{x}_{i+1}\hat{x}_i) e_q(-\hat{x}_i), \\
\label{pentagon2}
&e_q(-\check{x}_{i+1}) e_q(-\check{x}_i)
= e_q(-\check{x}_i) e_q(-\check{x}_i\check{x}_{i+1}) e_q(-\check{x}_{i+1}).
\end{align}
\begin{Lemma}\label{lem:adjoint-reduction}
For any $N \geq 3$,
\begin{align}
& e_q(-\hat{x}_{N-2})^{-1} \cdots e_q(-\hat{x}_{1})^{-1} e_q(-\hat{x}_{0}) e_q(-\hat{x}_{1}) \cdots e_q(-\hat{x}_{N-2}) \nonumber\\
& \qquad = e_q(-\hat{x}_{N-2}\cdots \hat{x}_0) \cdots e_q(-\hat{x}_1\hat{x}_0) e_q(-\hat{x}_0) , \label{adjoint-reduction1}\\
& e_q(-\check{x}_{N-2}) \cdots e_q(-\check{x}_{1}) e_q(-\check{x}_{0}) e_q(-\check{x}_{1})^{-1} \cdots e_q(-\check{x}_{N-2})^{-1} \nonumber\\
& \qquad = e_q(-\check{x}_0) e_q(-\check{x}_0\check{x}_1) \cdots e_q(-\check{x}_0 \cdots \check{x}_{N-2}).\label{adjoint-reduction2}
\end{align}
\end{Lemma}
\begin{proof}
We show \eqref{adjoint-reduction2} by induction. When $N=3$, 
the pentagon identity \eqref{pentagon2} with $i=0$ implies $e_q(-\check{x}_{1}) e_q(-\check{x}_{0}) e_q(-\check{x}_{1})^{-1}
= e_q(-\check{x}_0) e_q(-\check{x}_0\check{x}_1)$.
Now suppose \eqref{adjoint-reduction2} is true for $N=k$. We note that $\check{x}_{k-1}$ commutes with $\check{x}_0, \check{x}_1, \dots,\check{x}_{k-3}$
and $\check{x}_{k-1} \check{x}_{k-2} = q \check{x}_{k-2} \check{x}_{k-1}$. Hence,
\begin{align*}
& e_q(-\check{x}_{k-1}) \cdots e_q(-\check{x}_{1}) e_q(-\check{x}_{0}) e_q(-\check{x}_{1})^{-1} \cdots e_q(-\check{x}_{k-1})^{-1} \\
&\qquad = e_q(-\check{x}_0) e_q(-\check{x}_0\check{x}_1) \cdots e_q(-\check{x}_0 \cdots \check{x}_{k-3}) e_q(-\check{x}_{k-1})
e_q(-\check{x}_0 \cdots \check{x}_{k-2})e_q(-\check{x}_{k-1})^{-1} \\
&\qquad = e_q(-\check{x}_0) e_q(-\check{x}_0\check{x}_1) \cdots e_q(-\check{x}_0 \cdots \check{x}_{k-3})
e_q(-\check{x}_0 \cdots \check{x}_{k-2})e_q(-\check{x}_0 \cdots \check{x}_{k-1}),
\end{align*}
where for the first equality we have used the assumption of induction.
We see that \eqref{adjoint-reduction2} is also true for $N=k+1$.
Similarly, we can check~\eqref{adjoint-reduction1} for $\hat{x}_i$.
For example, the pentagon identity~\eqref{pentagon1} with $i=0$ implies
$e_q(-\check{x}_{1})^{-1} e_q(-\check{x}_{0}) e_q(-\check{x}_{1})
= e_q(-\check{x}_1\check{x}_0)e_q(-\check{x}_0)$.
Note that compared with $\check{x}_i$,
the cyclic ordering of $q$-commuting variables is reversed.
\end{proof}

By Lemma \ref{lem:adjoint-reduction}, we can reduce the left block \smash{$\mathcal{A}_L^{\mathrm{s}}$}
and the right block \smash{$\mathcal{A}_R^{\mathrm{s}}$} of the non-sta\-tion\-ary~$\glN$ Hamiltonian~\eqref{eq:glN Hamiltonian}
into the factorized form of higher root type.

\begin{Proposition}\label{A_LR-reduction}
We can recast the factorized Hamiltonian of simple root type to the form of higher root type:
\begin{align*}
&\mathcal{A}_L^{(\mathrm{s})}= \mathcal{A}_L^{(\mathrm{h})}
 = e_q(-\check{x}_0) e_q(-\check{x}_0\check{x}_1) \cdots e_q(-\check{x}_0 \cdots \check{x}_{N-2}) \cdot
e_q(-\check{x}_{N-1}) \cdots e_q(-\check{x}_{1}) \cdot \varphi(\Lambda), \\
&\mathcal{A}_R^{(\mathrm{s})}= \mathcal{A}_R^{(\mathrm{h})}
 = \varphi\bigl(q^{1-N} D_N \Lambda\bigr) \cdot
e_q(-\hat{x}_1) \cdots e_q(-\hat{x}_{N-1}) \nonumber\\
&\hphantom{\mathcal{A}_R^{(\mathrm{s})}= \mathcal{A}_R^{(\mathrm{h})} =}{}
\cdot e_q(-\hat{x}_{N-2} \cdots \hat{x}_0) \cdots
e_q(-\hat{x}_1 \hat{x}_0) e_q(-\hat{x}_0).
\end{align*}
\end{Proposition}

Let us introduce the cyclic shift $\pi(\check{x}_i)=\check{x}_{i+1}~(i \in \mathbb{Z}/ N\mathbb{Z})$ and similarly for $\hat{x}_i$;
$\pi$~is an automorphism of the algebra. Using the pentagon identity \eqref{pentagon}, we can show the
Hamiltonian enjoys the cyclic symmetry \smash{$\pi\bigl(\mathcal{H}^{\glN}\bigr) = \mathcal{H}^{\glN}$}.
It is enough to prove $\pi(\mathcal{A}_L) = \mathcal{A}_L$ and ${\pi(\mathcal{A}_R) = \mathcal{A}_R}$, since other parts of
$\mathcal{H}^{\glN}$ are manifestly symmetric under the cyclic permutation.

\begin{Proposition}\label{Prop:cyclic-sym}
We have
\[
\pi(\mathcal{A}_L) = \mathcal{A}_L, \qquad \pi(\mathcal{A}_R) = \mathcal{A}_R,
\]
and the Hamiltonian of the non-stationary $\glN$ difference equation has the cyclic symmetry in the variables $x_i$.
\end{Proposition}

\begin{proof}
Recall that Lemma \ref{lem:adjoint-reduction} is derived by applying the pentagon identity $N-2$ times
for the $q$-commuting pairs of variables $(\check{x}_0, \check{x}_1), (\check{x}_0\check{x}_1, \check{x}_2),
\dots , (\check{x}_0\check{x}_1\cdots\check{x}_{N-3}, \check{x}_{N-2})$.
Our strategy is to apply the pentagon identity $N-2$ times for the $q$-commuting pairs
$(\check{x}_1, \check{x}_2), (\check{x}_1\check{x}_2, \check{x}_3),
\dots ,\allowbreak (\check{x}_1\check{x}_2\cdots\check{x}_{N-2}, \check{x}_{N-1})$.
After the first step of applying the pentagon identity for $e_q(-\check{x}_2) e_q(-\check{x}_1)$, we have
\begin{align*}
&\mathcal{A}_L= e_q(-\check{x}_1) e_q(-\check{x}_{N-2}) \cdots e_q(-\check{x}_{3}) e_q(-\check{x}_{1}\check{x}_{2})
e_q(-\check{x}_0) e_q(-\check{x}_{1}\check{x}_{2})^{-1} \\
& \hphantom{\mathcal{A}_L=}{}
 \cdot e_q(-\check{x}_{3})^{-1} \cdots e_q(-\check{x}_{N-2})^{-1} e_q(-\check{x}_{N-1}) \cdots e_q(-\check{x}_{3})
e_q(-\check{x}_{1}\check{x}_{2})e_q(-\check{x}_{2}).
\end{align*}
Then after the second step of applying the pentagon identity for $e_q(-\check{x}_3) e_q(-\check{x}_1\check{x}_2)$, we have
\begin{align*}
&\mathcal{A}_L= e_q(-\check{x}_1)e_q(-\check{x}_{1}\check{x}_{2}) e_q(-\check{x}_{N-2}) \cdots e_q(-\check{x}_{4})
e_q(-\check{x}_{1}\check{x}_{2}\check{x}_{3})
e_q(-\check{x}_0)\\
& \hphantom{\mathcal{A}_L=}{}
 \cdot e_q(-\check{x}_{4})^{-1} \cdots e_q(-\check{x}_{N-2})^{-1} e_q(-\check{x}_{N-1}) \cdots e_q(-\check{x}_{4})
e_q(-\check{x}_{1}\check{x}_{2}\check{x}_{3})
e_q(-\check{x}_{3})e_q(-\check{x}_{2}).
\end{align*}
We repeatedly apply the pentagon identity in a similar manner. After the $N-2$ steps, we arrive
\begin{align*}
&\mathcal{A}_L= e_q(-\check{x}_1)e_q(-\check{x}_{1}\check{x}_{2}) \cdots e_q(-\check{x}_{1}\cdots \check{x}_{N-2}) e_q(-\check{x}_0)\\
& \hphantom{\mathcal{A}_L=}{}
 \cdot e_q(-\check{x}_{1}\check{x}_{2}\cdots \check{x}_{N-1}) e_q(-\check{x}_{N-1}) \cdots e_q(-\check{x}_2).
\end{align*}
Since $\check{x}_0$ and $\check{x}_1\check{x}_2\cdots\check{x}_{N-1}$ are commuting, this completes the proof of the first identity.
The second identity can be proved in a similar way.
\end{proof}

\newcommand{\Dshift}{\pi}
\newcommand{\Dreflection}[1]{s_{#1}}
\newcommand{\RA}[2]{ \mathcal{A}^{#1}_{#2}}
\newcommand{\tRA}[2]{\mathcal{A}^{#1}_{#2}}
\newcommand{\RB}[1]{\mathcal{B}_{#1}}
\newcommand{\tRB}[1]{\mathcal{B}^{#1}}

We have proved that the \smash{$\glN$} Hamiltonian is invariant under the shift $\Dshift(\hat x_i) := \hat x_{i+1}$.
By the pentagon identity, we can also check the invariance under the automorphisms of the Dynkin diagram of $A_{N-1}^{(1)}$ for $N \geq 3$.
Let $\Dshift$ and $\Dreflection{j}$ be
the automorphism $\Dshift(ab)=\Dshift(a)\Dshift(b)$ and
the anti-automorphisms $\Dreflection{j}(ab)=\Dreflection{j}(b)\Dreflection{j}(a)$, respectively,
such that
\begin{gather}
\Dshift(\hat x_i) := \hat x_{i+1},\qquad
\Dreflection{j}(\hat x_i) := \hat x_{2j-i}.
\end{gather}
Since $\Dreflection{\frac n2} = \Dshift^{n}\circ\Dreflection{0}$ for any $n\in\bZ$,
the group generated by $\Dshift$ and $\Dreflection{\frac n2}$'s $(n\in\bZ)$
is generated by $\Dshift$ and $\Dreflection{0}$, i.e.,
$\langle \Dshift, \Dreflection{\frac n2} \rangle_{n\in\bZ}=\langle \Dshift, \Dreflection{0} \rangle$.
Note that the automorphism of the Dynkin diagram of \smash{$A^{(1)}_{N-1}$},
which is isomorphic to the dihedral group, is generated by $\Dshift$ and $\Dreflection{0}$
with
\begin{gather}
\Dshift^N = \Dreflection{0}^2={\rm id},
\qquad
\Dreflection{0}=\Dshift\circ\Dreflection{0}\circ\Dshift.
\end{gather}

For $i\in\bZ/N\bZ$, let us look at the following quantities:
\begin{align}
&{\RA{i-1}{i}}
:=
e_q(-\hat{x}_{i+1}) \cdots e_q(-\hat{x}_{i+N-2})
\cdot
e_q(-\hat{x}_{i+N-1})
\cdot
e_q(-\hat{x}_{i+N-2})^{-1}\cdots e_q(-\hat{x}_{i+1})^{-1}\nonumber\\
& \hphantom{{\RA{i-1}{i}}:=}{}
\cdot
e_q(-\hat{x}_{i})
 \cdot
e_q(-\hat{x}_{i+1}) \cdots e_q(-\hat{x}_{i+N-2})
\end{align}
and
\begin{align}
&{\RB {i}}
:=
e_q(-\hat{x}_{i+1}) \cdots e_q(-\hat{x}_{i+N-2}) e_q(-\hat{x}_{i+N-1})
\nonumber\\
&\hphantom{{\RB {i}}:=}{}
\cdot
e_q(-\hat{x}_{i+N-2} \cdots \hat{x}_{i+1}\hat{x}_{i}) \cdots e_q(-\hat{x}_{i+1}\hat{x}_{i}) e_q(-\hat{x}_{i}),
\nonumber\\
&{\tRB {i}}
:=
e_q(-\hat{x}_{i+N})
e_q(-\hat{x}_{i+N}\hat{x}_{i+N-1}) \cdots
e_q(-\hat{x}_{i+N} \hat{x}_{i+N-1}\cdots \hat{x}_{i+2})
\nonumber\\
& \hphantom{{\tRB {i}}:=}{}
\cdot
e_q(-\hat{x}_{i+1}) e_q(-\hat{x}_{i+2}) \cdots e_q(-\hat{x}_{i+N-1}).
\end{align}
Then
\begin{alignat}{3}
&\Dshift\bigl({\RA{i-1}{i}}\bigr)
=
{\tRA{i}{i+1}},
\qquad &&
\Dreflection{j}\bigl({\RA{i-1}{i}}\bigr)
=
{\tRA{2j-i}{2j-i+1}},&\nonumber\\
&\Dshift({\RB {i}})
=
{\RB {i+1}},
\qquad &&
\Dreflection{j}({\RB {i}})
=
{\tRB {2j-i}},&\nonumber\\
&\Dshift({\tRB {i}})
=
{\tRB {i+1}},
\qquad&&
\Dreflection{j}({\tRB {i}})
=
{\RB {2j-i}}.&
\end{alignat}
Therefore,
\smash{$\bigl\{ {\RA{i-1}{i}} \bigr\}_{i\in\bZ/N\bZ}$} and
$\{ {\RB{i}}, {\tRB{i}} \}_{i\in\bZ/N\bZ}$ are
invariant and transitive under the Dynkin automorphism group
$\langle \Dshift, \Dreflection{0} \rangle$.

By iteratively using the pentagon identity \eqref{pentagon2},
we have (see the proof of Lemma~\ref{lem:adjoint-reduction}) the following.

\begin{Lemma}
For any $N \geq 3$ and $1\leq j-i\leq N-2$,
\begin{align}
&e_q(-\hat{x}_i) \cdot
e_q(-\hat{x}_{i+1}) \cdots e_q(-\hat{x}_j) \nonumber\\
&\qquad =
e_q(-\hat{x}_{i+1}) \cdots e_q(-\hat{x}_{j})
\cdot
e_q(-\hat{x}_j \cdots \hat{x}_{i+1}\hat{x}_i) \cdots e_q(-\hat{x}_{i+1}\hat{x}_i) e_q(-\hat{x}_i)
\label{eq:Lem1}
\end{align}
and
\begin{align}
&e_q(-\hat{x}_i) \cdots e_q(-\hat{x}_{j-1})
\cdot e_q(-\hat{x}_j) \nonumber\\
& \qquad =
e_q(-\hat{x}_j) e_q(-\hat{x}_j\hat{x}_{j-1}) \cdots e_q(-\hat{x}_j \hat{x}_{j-1}\cdots \hat{x}_i)
\cdot
e_q(-\hat{x}_i) \cdots e_q(-\hat{x}_{j-1}).
\label{eq:Lem2}
\end{align}
\end{Lemma}
Then, the relations \eqref{eq:Lem1} and \eqref{eq:Lem2} imply the following result,
which generalizes Proposition~\ref{Prop:cyclic-sym}.

\begin{Lemma}
For any integer $N \geq 3$,
${\RA{i-1}{i}}={\tRB {i-1}}={\RB {i+1}}={\tRA{i}{i+1}}$.
\end{Lemma}

\begin{proof}
By using
$(\ref{eq:Lem2})$ with $(i,j)=(i+1,i+N-1)$,
${\RA{i-1}{i}}={\tRB {i-1}}$.
By using $(\ref{eq:Lem1})$ with $(i,j)=(i+1,i+N-2)$,
\begin{align}
&{\tRB {i-1}}
=
e_q(-\hat{x}_{i+N-1}) e_q(-\hat{x}_{i+N-1}\hat{x}_{i+N-2})
\cdots
e_q(-\hat{x}_{i+N-1} \hat{x}_{i+N-2}\cdots \hat{x}_{i+2})
\nonumber\\
&\hphantom{{\tRB {i-1}}=}{}
\cdot
e_q(-\hat{x}_{i+N-1} \hat{x}_{i+N-2}\cdots \hat{x}_{i+1})
\cdot
e_q(-\hat{x}_{i})
\cdot
e_q(-\hat{x}_{i+2}) \cdots e_q(-\hat{x}_{i+N-2})
\nonumber\\
&\hphantom{{\tRB {i-1}}=}{}
\cdot
e_q(-\hat{x}_{i+N-2} \cdots \hat{x}_{i+2}\hat{x}_{i+1})
\cdots e_q(-\hat{x}_{i+2}\hat{x}_{i+1}) e_q(-\hat{x}_{i+1}).
\label{eq:A}
\end{align}
Acting the anti-automorphism $\Dreflection{i}$ on the above equations,
we have
${\tRA{i}{i+1}}={\RB {i+1}}$ and
\begin{align}
&{\RB {i+1}}
=
e_q(-\hat{x}_{i+N-1}) e_q(-\hat{x}_{i+N-1}\hat{x}_{i+N-2})
\cdots
e_q(-\hat{x}_{i+N-1} \hat{x}_{i+N-2}\cdots \hat{x}_{i+2})
\nonumber\\
&\hphantom{{\RB {i+1}}=}{}
\cdot
e_q(-\hat{x}_{i+2}) \cdots e_q(-\hat{x}_{i+N-2})
\cdot
e_q(-\hat{x}_{i+N})
\cdot
e_q(-\hat{x}_{i+N-1} \cdots \hat{x}_{i+2}\hat{x}_{i+1})
\nonumber\\
&\hphantom{{\RB {i+1}}=}{}
\cdot
e_q(-\hat{x}_{i+N-2} \cdots \hat{x}_{i+2}\hat{x}_{i+1})
 \cdots
e_q(-\hat{x}_{i+2}\hat{x}_{i+1}) e_q(-\hat{x}_{i+1}).
\end{align}
Since
$e_q(-\hat{x}_{i+2}) \cdots e_q(-\hat{x}_{i+N-2})$,
$e_q(-\hat{x}_{i+N})$
and
$e_q(-\hat{x}_{i+N-1} \cdots \hat{x}_{i+2}\hat{x}_{i+1})$
commute each other,\footnote{Note that
$[\, \hat{x}_{k}\, ,\ \hat{x}_{\ell} \, ]=0$ for $k \neq \ell \pm 1$ (mod $N$) and
$[\, \hat{x}_{i+N-1} \cdots \hat{x}_{i+2}\hat{x}_{i+1} \, , \, \hat{x}_{i+j} \, ]=0$
for
$j\neq \pm 1$ (mod $N$).}
 we obtain the lemma.
\end{proof}

By this lemma,
${\RB {i+1}}={\tRA{i}{i+1}}={\RA{j}{j+1}}={\tRB {j}}$
for any $i,j\in\bZ/N\bZ$.
Thus,
${\RA{i-1}{i}}={\RB {j}}={\tRB {k}}$
for any $i,j,k\in\bZ/N\bZ$.
Therefore, we finally obtain the following.

\begin{Proposition}\label{Dynkin-inv}
For any integer $N \geq 3$, and for any $i\in\bZ/N\bZ$,
${\RA{i-1}{i}}={\RB {i}}={\tRB {i}}$.
Hence, ${\RB {i}}={\tRB {i}}$ is invariant under the Dynkin automorphism group
$\langle \Dshift, \Dreflection{0} \rangle$,
i.e.,
$ \Dshift({\RB {i}}) = \Dreflection{0}({\RB {i}}) = {\RB {i}}$.
\end{Proposition}

\begin{Corollary}\label{Cor2.8}
The non-stationary Hamiltonian enjoys the full invariance under the Dynkin automorphism of \smash{$A_{N-1}^{(1)}$}.
\end{Corollary}

The original definition of the right block $\mathcal{A}_R$ of $\glN$ Hamiltonian
employs \smash{$\mathcal{A}_0^{N-1}$} in Proposition~\ref{Dynkin-inv} (see Definition~\ref{def:glN Hamiltonian}).
On the other hand, it is
\begin{equation}\label{another_expression_of_A_R}
\tRB {0}=
e_q(-\hat{x}_0)e_q(-\hat{x}_0\hat{x}_{N-1})\cdots
 e_q(-\hat{x}_0\hat{x}_{N-1}\cdots\hat{x}_2)
\cdot e_q(-\hat{x}_1)\cdots e_q(-\hat{x}_{N-1})
\end{equation}
that naturally appears in the Hamiltonian
constructed from the universal $R$-matrix of \smash{$U_q \bigl(A_{N-1}^{(1)}\bigr)$} (see \cite{AHKOSSY}).
We have focused on the right block $\mathcal{A}_R$ of the non-stationary Hamiltonian.
Similarly, we can confirm the invariance under the Dynkin automorphism of the left block
$\mathcal{A}_L$ with $q$-commutative variables $\check{x}_i$.
The dihedral group invariance of the remaining parts of the Hamiltonian
is trivial. Hence, we obtain Corollary~\ref{Cor2.8}.

\subsection{Normal ordered form of the Hamiltonian}
\label{subsec:normal}

In this subsection, we prove the equivalence of the Hamiltonian of factorized form and of normal ordered form
(see Definitions \ref{factorized1}, \ref{factorized2} and \ref{def:normal-ordered}).

\begin{Proposition}\label{F=N}
We can recast the factorized Hamiltonian of higher root type to the normal ordered form
\[
\mathcal{A}_R^{(\mathrm{h})} = \mathcal{A}_R^{(\mathrm{n})}, \qquad
\mathcal{A}_L^{(\mathrm{h})} = \mathcal{A}_L^{(\mathrm{n})}.
\]
\end{Proposition}
Recall that $\hat{x}_i = d_i \overline{d}_i x_i q^{\vartheta_i - \vartheta_{i-1}}$
and $\check{x}_i := x_i q^{-\vartheta_i + \vartheta_{i-1}}$.\footnote{We have fixed the scaling parameters as $\alpha_i = d_i \overline{d}_i$ and $\beta_i =1$ to write down the building block of the Hamiltonian.}
We are going to show the agreement of the factorized form of
the building block \smash{$\mathcal{A}_R^{(\mathrm{h})}$} of higher root type
and the corresponding normal ordered form \smash{$\mathcal{A}_R^{(\mathrm{n})}$}.
The agreement of \smash{$\mathcal{A}_L^{(\mathrm{h})}$} and \smash{$\mathcal{A}_L^{(\mathrm{n})}$} is proved similarly.
To prove Proposition~\ref{F=N} for $\mathcal{A}_R$,
we need Proposition~\ref{Prop;2.9} below, which relies on the $q$-multinomial formula.
\begin{Lemma}
If $x_i x_j = q x_j x_i$ for $i<j$, we have
\begin{equation}\label{multi-nomial}
(x_1 + x_2 + \cdots + x_m)^n = \sum_{\substack{k_1,\dots,k_{m}\geq 0 \\ k_1+ k_2 + \cdots + k_m =n}}
\cb{n}{k_1, k_2, \dots, k_m}{q} x_m^{k_m} \cdots x_2^{k_2} x_1^{k_1},
\end{equation}
where
\[
\cb{n}{k_1, k_2, \dots, k_m}{q}:= \frac{[n]_q!}{[k_1]_q! [k_2]_q! \cdots [k_m]_q!}
= \frac{(q;q)_n}{(q;q)_{k_1} (q;q)_{k_2} \cdots (q;q)_{k_m}},
\]
denotes the $q$-multinomial coefficients.
\end{Lemma}

When $m=2$, \eqref{multi-nomial} is given in \cite[Exercise~1.35]{GR} (see also references there).
The general case is proved by induction on $m$ as follows.

\begin{proof}
Assume that \eqref{multi-nomial} holds for some arbitrary $m\in\bZ_{\geq 2}$, then
\begin{align*}
\frac{\left(x_1+\cdots+
x_{m+1}\right)^n}{(q;q)_n}
&=\frac{\left((x_1+\cdots+x_m)+x_{m+1}\right)^n}{(q;q)_n}
 = \sum_{ \Frac{k,\ell\geq 0}{k+\ell=n} }
 \frac{x_{m+1}^{k}}{(q;q)_{k}} \frac{(x_1+\cdots+x_m)^{\ell}}{(q;q)_{\ell}}
\\
&=\sum_{ \Frac{k,\ell\geq 0}{k+\ell=n} }
 \sum_{ \Frac{k_1,\dots,k_{m}\geq 0}{k_1+\cdots+k_{m}=\ell} }
 \frac{x_{m+1}^{k}}{(q;q)_{k}} \frac{x_m^{k_m}}{(q;q)_{k_m}}
 \cdots \frac{x_2^{k_2}}{(q;q)_{k_2}} \frac{x_1^{k_1}}{(q;q)_{k_1}}
\\
&= \sum_{ \Frac{k_1,\dots,k_{m+1}\geq 0}{k_1+\cdots+k_{m+1}=n} }
 \frac{x_{m+1}^{k_{m+1}}}{(q;q)_{k_{m+1}}} \frac{x_m^{k_m}}{(q;q)_{k_m}}
 \cdots \frac{x_2^{k_2}}{(q;q)_{k_2}} \frac{x_1^{k_1}}{(q;q)_{k_1}}.\tag*{\qed}
\end{align*}\renewcommand{\qed}{}
\end{proof}

\begin{Proposition}\label{Prop;2.9}
 We have
\begin{equation}\label{P-2.9}
\varphi(\hat{x}_{N-1}\cdots \hat{x}_1 \hat{x}_0) \cdot \widetilde{\mathcal{A}}_R
= {:}\varphi(\hat{x}_1)\varphi(\hat{x}_2)\cdots \varphi(\hat{x}_N){:},
\end{equation}
where
\begin{equation}
\widetilde{\mathcal{A}}_R := e_q(-\hat{x}_1) e_q(-\hat{x}_2) \cdots e_q(-\hat{x}_{N-1})
e_q(-\hat{x}_{N-2} \cdots \hat{x}_1 \hat{x}_0) \cdots e_q(-\hat{x}_1 \hat{x}_0) e_q(-\hat{x}_0).
\end{equation}
\end{Proposition}

\begin{proof}
First note that for $\mathbf{i}=(i_1,i_2,\dots, i_N) \in \mathbb{Z}_{\geq 0}^N$,
\begin{align}
{\hat{x}_1}^{i_1} \cdots {\hat{x}_N}^{i_N}
&=\left(x_1 \frac{p_1}{p_N}\right)^{i_1} \left(x_2 \frac{p_2}{p_1}\right)^{i_2} \cdots \left(x_N \frac{p_N}{p_{N-1}}\right)^{i_N}
=q^{-i_1 i_N} \prod_{a=1}^N q^{\frac{i_a(i_a-1)}{2}} {:}\prod_{a=1}^N {\hat{x}_a}^{i_a}{:}.
\end{align}
Hence, by using the expansion formulas \eqref{q-exp1} and \eqref{q-exp2}, we have
\begin{align}
 {:}\varphi(-\hat{x}_1)\varphi(-\hat{x}_2)\cdots \varphi(-\hat{x}_N){:}
={}&\sum_{\mathbf{i} \in \mathbb{Z}_{\geq 0}^N} \left( \prod_{a=1}^N \frac{q^{\frac{i_a(i_a-1)}{2}}}{(q;q)_{i_a}}\right)
 {:}{\hat{x}_1}^{i_1} \cdots {\hat{x}_N}^{i_N}{:} \nonumber\\
={}&\sum_{\mathbf{i} \in \mathbb{Z}_{\geq 0}^N} \frac{(q^{i_N}\hat{x}_1)^{i_1}}{(q;q)_{i_1}} \frac{{\hat{x}_2}^{i_2}}{(q;q)_{i_2}}\cdots \frac{{\hat{x}_N}^{i_N}}{(q;q)_{i_N}} \nonumber \\
={}&\sum_{i_N \geq 0} p_1^{i_N}\frac{1}{\varphi(\hat{x}_1)} \frac{1}{\varphi(\hat{x}_2)} \cdots \frac{1}{\varphi(\hat{x}_{N-1})}
\frac{(p_1^{-1}\hat{x}_N)^{i_N}}{(q;q)_{i_N}} \nonumber \\
={}&\frac{1}{\varphi(\hat{x}_1)} \frac{1}{\varphi(\hat{x}_2)} \cdots \frac{1}{\varphi(\hat{x}_{N-1})}
\left( \sum_{i_N \geq 0} A^{i_N} \frac{(p_1^{-1}\hat{x}_N)^{i_N}}{(q;q)_{i_N}} \right),\label{eq:Hasegawa-sum}
\end{align}
where $A$ is defined by
\begin{equation}
A:=\varphi(\hat{x}_{N-1}) \cdots \varphi(\hat{x}_1) \cdot p_1 \cdot \frac{1}{\varphi(\hat{x}_1)}\cdots \frac{1}{\varphi(\hat{x}_{N-1})}.
\end{equation}
We can decompose $A$ as follows:
\begin{align}
A&=\varphi(\hat{x}_{N-1}) \cdots \varphi(\hat{x}_2) (1-\hat{x}_1) \frac{1}{\varphi(\hat{x}_2)} \cdots \frac{1}{\varphi(\hat{x}_{N-1})} p_1\nonumber \\
&=\varphi(\hat{x}_{N-1}) \cdots \varphi(\hat{x}_3) (1-\hat{x}_1+\hat{x}_2\hat{x}_1) \frac{1}{\varphi(\hat{x}_3)} \cdots \frac{1}{\varphi(\hat{x}_{N-1})} p_1\nonumber \\
& \quad \ \vdots \nonumber \\
&= A_0+A_1+\cdots+ A_{N-1},
\end{align}
where
\begin{equation}
A_0=p_1, \quad A_1=-\hat{x}_1 p_1, \quad A_2=\hat{x}_2\hat{x}_1 p_1, \quad \dots, \quad A_{N-1}=(-1)^{N-1} \hat{x}_{N-1} \cdots \hat{x}_1 p_1.
\end{equation}

To compute the sum on the right-hand side of \eqref{eq:Hasegawa-sum}
with $\hat{x}_N=\hat{x}_0$, we note the following:
\begin{enumerate}\itemsep=0pt
\item[(i)]
Since $A_i A_{i+1}=q A_{i+1} A_i$ ($i=0,\dots, N-1$), we can apply $q$-multinomial formula \eqref{multi-nomial}:
\begin{align}
 \frac{ A^{i_N} }{(q;q)_{i_N}}
= \sum_{\substack{k_0+\cdots+k_{N-1}=i_N \\ k_0,k_2, \dots, k_{N-1} \geq 0}}
\frac{A_{N-1}^{k_{N-1}}}{(q;q)_{k_{N-1}}}\cdots \frac{{A_0}^{k_0}}{(q;q)_{k_0}}.
\end{align}

\item[(ii)]
$A_i$ and $p_1^{-1} \hat{x}_0$ are commutative for $i=0, \dots, N-2$.

\item[(iii)]
Since
$A_{N-1} p_1^{-1} \hat{x}_0=q p_1^{-1} \hat{x}_0 A_{N-1}$, we have
\begin{align}
A_{N-1}^{k_{N-1}} \bigl(p_1^{-1} \hat{x}_0\bigr)^{k_{N-1}}&=q^{\frac{1}{2}k_{N-1}(k_{N-1}-1)}(A_{N-1} p_1^{-1} \hat{x}_0)^{k_{N-1}} \nonumber \\
&=q^{\frac{1}{2}k_{N-1}(k_{N-1}-1)} ((-1)^{N-1}\hat{x}_{N-1} \hat{x}_{N-2} \cdots \hat{x}_1 \hat{x}_0)^{k_{N-1}}.
\end{align}
\end{enumerate}
Hence, we have
\begin{align}
&\sum_{i_N\geq 0} \frac{ A^{i_N} }{(q;q)_{i_N}}\bigl(p_1^{-1} \hat{x}_0\bigr)^{i_N}\nonumber\\
&\qquad{} =\sum_{k_0,k_2, \dots, k_{N-1} \geq 0} \frac{A_{N-1}^{k_{N-1}} \bigl(p_1^{-1} \hat{x}_0\bigr)^{k_{N-1}}}{(q;q)_{k_{N-1}} }
 \frac{\bigl(A_{n-2} p_1^{-1} \hat{x}_0\bigr)^{k_{N-2}}}{(q;q)_{k_{N-2}} }
 \cdots \frac{\bigl(A_{0} p_1^{-1} \hat{x}_0\bigr)^{k_{0}}}{(q;q)_{k_{0}} } \nonumber \\
&\qquad{} =\varphi\bigl((-1)^{N}\hat{x}_{N-1}\cdots \hat{x}_1 \hat{x}_0\bigr)
\frac{1}{\varphi((-1)^{N-2}\hat{x}_{N-2}\cdots \hat{x}_1\hat{x}_0)}\cdots
\frac{1}{\varphi(-\hat{x}_1\hat{x}_0)}
\frac{1}{\varphi(\hat{x}_0)}.
\end{align}
We finally obtain
\begin{align}
 &{:}\varphi(-\hat{x}_1)\varphi(-\hat{x}_2)\cdots \varphi(-\hat{x}_N){:} \nonumber\\
&\qquad{} = \frac{1}{\varphi(\hat{x}_1)} \frac{1}{\varphi(\hat{x}_2)} \cdots \frac{1}{\varphi(\hat{x}_{N-1})}
 \varphi((-1)^{N}\hat{x}_{N-1}\cdots \hat{x}_1 \hat{x}_0) \nonumber\\
& \qquad\quad \times \frac{1}{\varphi((-1)^{N-2}\hat{x}_{N-2}\cdots \hat{x}_1\hat{x}_0)}\cdots
\frac{1}{\varphi(-\hat{x}_1\hat{x}_0)}
\frac{1}{\varphi(\hat{x}_0)}.
\end{align}
By replacing $\hat{x}_i$ with $-\hat{x}_i$, this implies the desired relation.
\end{proof}

Since Proposition~\ref{A_LR-reduction} implies that the left-hand side of~\eqref{P-2.9} is
equal to $\mathcal{A}_R^{(\mathrm{s})} = \mathcal{A}_R^{(\mathrm{h})}$,
Proposition~\ref{F=N} for $\mathcal{A}_R$ follows from Proposition~\ref{Prop;2.9}.
Note that $\hat{x}_{N-1}\cdots \hat{x}_1 \hat{x}_0 = q^{1-N} D_N \Lambda$ is central.

\subsection[Classical analogue of A\_L\^{}\{(\mathrm{s})\} and A\_R\^{}\{(\mathrm{s})\}]{Classical analogue of $\boldsymbol{\mathcal{A}_L^{(\mathrm{s})}}$ and $\boldsymbol{\mathcal{A}_R^{(\mathrm{s})}}$}
\label{classical-AL-AR}

This subsection is an interesting detour. Logically, it is not necessary
for the following sections and may be skipped. But we would like to make a remark on the
factorization of the classical cyclic matrix, which is instructive for understanding
\smash{$\mathcal{A}_L^{(\mathrm{s})}$} and \smash{$\mathcal{A}_R^{(\mathrm{s})}$} in the $\glN$ Hamiltonian.
For $0\leq i\leq n-1$ and $x \in {\mathbb C}$, let $J_i(x)$ be the $n \times n$ elementary Jacobi matrix defined as
\begin{align}
&J_i(x)=\exp(x e_i)={\bf 1}+x e_i, \nonumber \\
&e_i=E_{i,i+1}, \qquad 1\leq i\leq n-1, \qquad e_0=z E_{n,1},
\end{align}
where ${\bf 1}={\bf 1}_n$ is the identity matrix and $E_{i,j}$ is the matrix unit: $(E_{i,j})_{k,l}=\delta_{i,k}\delta_{j,l}$.
We define the matrix $X$ by
\begin{equation}
X={\bf 1}+\sum_{i=0}^{n-1} x_i e_i=\left[
\begin{array}{@{}cccccc@{}}
 1 & x_1&&& \\
 & 1 & x_2 \\
 &&\ddots&\ddots\\
 &&& 1 & x_{n-1} \\
 x_0 z && & & 1 \\
\end{array}
\right],
\end{equation}
which is manifestly cyclic and plays fundamental role in tropical/geometric crystal and discrete integrable systems.

We have the following factorization of the cyclic matrix $X$, where $X^{\pm 1}$
may be viewed as the classical analog of ${\mathcal A}_L^{(\mathrm{s})}$ and ${\mathcal A}_R^{(\mathrm{s})}$
which enjoy the cyclic symmetry.
\begin{Lemma}
The matrix $X$ is decomposed as
\begin{equation}\label{classical-X}
X=g J_{0}(x_0) d_{n}(v z) g^{-1} \cdot J_{n-1}(x_{n-1})\cdots J_{2}(x_2) J_{1}(x_1),
\end{equation}
where
\begin{align}
g=J_{n-2}(x_{n-2})\cdots J_{2}(x_2)J_{1}(x_1),\qquad
 d_i(x)={\bf 1}+x E_{i,i}, \qquad v=(-1)^{n-1} \prod_{i=0}^{n-1}x_i.
\end{align}
\end{Lemma}
\begin{proof}
A straightforward matrix computation.
\end{proof}

Note that the first factor in \eqref{classical-X} can be written various ways as
 \begin{align}
g J_{0}(x_0) d_{n}(v z) g^{-1}
 &=\prod_{j=0}^{n-1}\bigl({\bf 1}+ (-1)^{j} x_0 x_1 \cdots x_{j} z E_{n,j+1} \bigr)\nonumber \\
& ={\bf 1}+\sum_{j=0}^{n-1} (-1)^j x_0 x_1 \cdots x_{j} z E_{n,j+1} \nonumber \\
& =\left[
\begin{array}{@{}cccc|c@{}}
 &&{\bf 1}_{n-1} &&{\bf 0} \\
 \hline
 x_0 z &-x_0x_1z & \cdots & (-1)^{n-2}x_0 \cdots x_{n-2}z& vz
\end{array}
\right].
 \end{align}

\subsection[Other forms of widehat\{gl\_N\} Hamiltonian]{Other forms of $\boldsymbol{\glN}$ Hamiltonian}
\label{subsec:other-form}
\newcommand{\hathat}[1]{\hat{\hat{#1}}}

The $\glN$ Hamiltonian involves the $q$-exponential function with $q$-commutative variables~$\hat{x}_i$ and~$\check{x}_i$. We can recast it in such a form that the arguments of the $q$-exponential function
are commutative variables $x_i$ by moving the position of the $q$-Borel transformation.

Let $\vartheta:= x\frac{\partial}{\partial x}$.
Since
$\vartheta^k x = x (1+\vartheta)^k$ for any natural number $k\in\bN$,
we have
$q^{\vartheta} x = x q^{1+\vartheta}$ and
\smash{$q^{\ha\vartheta(\vartheta-1)} x
= x q^{\ha(1+\vartheta)\vartheta}
= x q^{\vartheta+\ha\vartheta(\vartheta-1)}$}.
Therefore,
$q^{c\vartheta} x^n =(q^c x)^n q^{c\vartheta}$ and
$\smash{q^{\ha\vartheta(\vartheta-1)} x^n}
= \smash{(x q^\vartheta)^n q^{\ha\vartheta(\vartheta-1)}}$
for any integer $n\in\bZ$ and $c\in\bC$.

Assume $N\geq 3 $. For $N=3$,
$f_{N-2}\cdots f_2 f_1$ stands for $f_1$.
Let us set $\alpha_i=1$ for simplicity, so that $\hat{x}_i = x_i q^{\vartheta_i - \vartheta_{i-1}}$.

\begin{Proposition}\label{Prop 2.10}
Let \smash{${\doublehat{x}}_i := x_i q^{\ha(\vartheta_{i+1} - \vartheta_{i-1})}$}. Then we have
\begin{align}
&
e_q(-\hat x_{1})e_q(-\hat x_{2})
\cdots
e_q(-\hat x_{N-1})
e_q(-\hat x_{N-2})^{-1}
\cdots
e_q(- \hat x_{2})^{-1}
e_q(-\hat x_{1})^{-1}
\nonumber\\
&\qquad \quad {}\cdot
e_q(-\hat x_{0})
e_q(-\hat x_{1})
\cdots
e_q(-\hat x_{N-2})
q^{\ha\sum_{i=1}^N\vartheta_i(\vartheta_i-\vartheta_{i-1})}
\label{eq:prophat}
\\
&\qquad =
q^{\ha\sum_{i=1}^N\vartheta_i(\vartheta_i-\vartheta_{i-1}-1)}
e_q(-\doublehat x_{1})e_q(-\doublehat x_{2})
\cdots
e_q(-\doublehat x_{N-1})
\nonumber\\
&
\qquad\quad \cdot
e_q(-\doublehat x_{N-2})^{-1}
\cdots
e_q(-\doublehat x_{2})^{-1}
e_q(-\doublehat x_{1})^{-1}
\cr
&
\qquad\quad \cdot
e_q(-\doublehat x_{0})
e_q(-\doublehat x_{1})
\cdots
e_q(-\doublehat x_{N-2})
q^{\ha\sum_{i=1}^N\vartheta_i}
\label{eq:prophathat}
\\
&\qquad =
q^{\ha\sum_{i=1}^N\vartheta_i(\vartheta_i-1)}
q^{-\vartheta_{0}\vartheta_{1}} e_q(-x_{1})
q^{-\vartheta_{1}\vartheta_{2}} e_q(-x_{2})
\cdots
q^{-\vartheta_{N-2}\vartheta_{N-1}} e_q(-x_{N-1})
\cr
&\qquad\quad \cdot
q^{\vartheta_{N-1}\vartheta_{N-2}} e_q(-x_{N-2})^{-1}
q^{\vartheta_{N-2}\vartheta_{N-3}} 
\cdots
e_q(-x_{2})^{-1}
q^{\vartheta_{2}\vartheta_{1}} e_q(-x_{1})^{-1}
q^{\vartheta_{1}\vartheta_{0}-\vartheta_{N-1}\vartheta_{0}}
\nonumber\\
&\qquad\quad \cdot
e_q(-x_{0})
q^{-\vartheta_{0}\vartheta_{1}} e_q(-x_{1})
q^{-\vartheta_{1}\vartheta_{2}} 
\cdots
e_q(-x_{N-2})
q^{-\vartheta_{N-2}\vartheta_{N-1}}
q^{\ha\sum_{i=1}^N\vartheta_i(1+\vartheta_{i+1})} .
\label{eq:prop}
\end{align}
\end{Proposition}

\begin{proof}
By the lemma below,
each term in the Taylor series of (\ref{eq:prophat})--(\ref{eq:prop})
in $\hat x_i$, $\doublehat x_i$ and $x_i$'s
coincides each other.
\end{proof}

\begin{Lemma}
For any integers $\ell_i, m_i, n_i\in\bZ$ $(i\in\bZ)$,
\begin{align}
&
\hat x_{1}^{\ell_{1}}\hat x_{2}^{\ell_{2}}
\cdots
\hat x_{N-1}^{\ell_{N-1}}
\cdot
\hat x_{N-2}^{m_{N-2}}
\cdots
\hat x_{2}^{m_{2}}
\hat x_{1}^{m_{1}}
\cdot
\hat x_{0}^{n_{0}}
\hat x_{1}^{n_{1}}
\cdots
\hat x_{N-2}^{n_{N-2}}
q^{\ha\sum_{i=1}^N\!\vartheta_i(\vartheta_i-\vartheta_{i-1})}
\label{eq:lemmahat}
\\
&=
q^{\ha\sum_{i=1}^N\!\vartheta_i(\vartheta_i-\vartheta_{i-1}-1)}
\doublehat x_{1}^{\ell_{1}}\doublehat x_{2}^{\ell_{2}}
\cdots
\doublehat x_{N-1}^{\ell_{N-1}}
\cdot
\doublehat x_{N-2}^{m_{N-2}}
\cdots
\doublehat x_{2}^{m_{2}}
\doublehat x_{1}^{m_{1}}
\cdot
\doublehat x_{0}^{n_{0}}
\doublehat x_{1}^{n_{1}}
\cdots
\doublehat x_{N-2}^{n_{N-2}}
q^{\ha\sum_{i=1}^N\!\vartheta_i}\!\!\!\!\!
\label{eq:lemmahathat}
\\
&=
q^{\ha\sum_{i=1}^N\!\vartheta_i(\vartheta_i-1)}
q^{-\vartheta_{0}\vartheta_{1}} x_{1}^{\ell_{1}}
q^{-\vartheta_{1}\vartheta_{2}} x_{2}^{\ell_{2}}
\cdots
q^{-\vartheta_{N-2}\vartheta_{N-1}} x_{N-1}^{\ell_{N-1}}
\nonumber\\
&\quad\cdot
q^{\vartheta_{N-1}\vartheta_{N-2}} x_{N-2}^{m_{N-2}}
q^{\vartheta_{N-2}\vartheta_{N-3}} 
\cdots
x_{2}^{m_{2}}
q^{\vartheta_{2}\vartheta_{1}} x_{1}^{m_{1}}
q^{\vartheta_{1}\vartheta_{0}-\vartheta_{N-1}\vartheta_{0}}
\nonumber\\
&\quad{}\cdot
x_{0}^{n_{0}}
q^{-\vartheta_{0}\vartheta_{1}} x_{1}^{n_{1}}
q^{-\vartheta_{1}\vartheta_{2}}
\cdots
x_{N-2}^{n_{N-2}}
q^{-\vartheta_{N-2}\vartheta_{N-1}}
q^{\ha\sum_{i=1}^N\vartheta_i(1+\vartheta_{i+1})} .\label{eq:lemma}
\end{align}
\end{Lemma}

\begin{proof}
With the formulas
\begin{align}
&q^{\ha\sum_{j=1}^N\vartheta_j(\vartheta_j-1)}
x_i^n
=
\bigl(x_i q^{\vartheta_i}\bigr)^n
q^{\ha\sum_{j=1}^N\vartheta_j(\vartheta_j-1)} ,
\\
&q^{-\vartheta_{i-1}\vartheta_{i}} x_{i}^{n}
=
\bigl(x_i q^{-\vartheta_{i-1}}\bigr)^n
q^{-\vartheta_{i-1}\vartheta_{i}},
\end{align}
we can move $q^{\ha\sum_{i=1}^N\vartheta_i(\vartheta_i-1)}$ and
$q^{-\vartheta_{i-1}\vartheta_{i}}$ in (\ref{eq:lemma}) to the right,
which yields the equality of~(\ref{eq:lemmahat}) and~(\ref{eq:lemma}).

Similarly, we can move $q^{\ha\sum_{i=1}^N\vartheta_i\vartheta_{i+1}} $ and
$q^{-\vartheta_{i}\vartheta_{i+1}} $ in (\ref{eq:lemma}) to the left with the formulas
\begin{align}
&x_i^n
q^{\ha\sum_{j=1}^N\vartheta_j\vartheta_{j+1}}
=
q^{\ha\sum_{j=1}^N\vartheta_j\vartheta_{j+1}}
\bigl(x_i q^{-\ha(\vartheta_{i+1}+\vartheta_{i-1})}\bigr)^n ,
\\
&x_{i}^{n}
q^{-\vartheta_{i}\vartheta_{i+1}}
=
q^{-\vartheta_{i}\vartheta_{i+1}}
\bigl(x_i q^{\vartheta_{i+1}}\bigr)^n,
\end{align}
which gives the equality of $(\ref{eq:lemmahathat})$ and $(\ref{eq:lemma})$.
\end{proof}

Since
\begin{align}
q^{\vartheta_i(\vartheta_i-\vartheta_{i-1}-\vartheta_{i+1}-1)}
x_i
&=
x_i
q^{(1+\vartheta_i)(\vartheta_i-\vartheta_{i-1}-\vartheta_{i+1})}
\cr
&=
x_i
q^{(\vartheta_i-\vartheta_{i-1})+(\vartheta_i-\vartheta_{i+1})}
q^{\vartheta_i(\vartheta_i-\vartheta_{i-1}-\vartheta_{i+1}-1)} ,
\end{align}
we obtain
\be
q^{\sum_{j=1}^N\vartheta_j(\vartheta_j-\vartheta_{j-1}-1)}
x_i^n
=
\bigl(x_i q^{(\vartheta_i-\vartheta_{i-1})+(\vartheta_i-\vartheta_{i+1})}\bigr)^n
q^{\sum_{j=1}^N\vartheta_j(\vartheta_j-\vartheta_{j-1}-1)} .
\ee
Therefore, we have
\be
q^{\ha\sum_{j=1}^N\vartheta_j(\vartheta_j-\vartheta_{j-1}-1)}
\doublehat x_{i}^n
=
\hat x_{i}^n
q^{\ha\sum_{j=1}^N\vartheta_j(\vartheta_j-\vartheta_{j-1}-1)} ,
\ee
which also yields the equality of $(\ref{eq:lemmahat})$ and $(\ref{eq:lemmahathat})$.

\begin{Remark}
The equation $(\ref{eq:prophat})$ equals to
$\varphi\bigl(q^{1-N}D_N\Lambda\bigr)^{-1}{\mathcal A}_R^{(s)}q^{\ha\Delta}$.
As $e_{q^{-1}}(x)=e_q(qx)^{-1}$,
by replacing $q$ and $x_i$'s with $1/q$ and $x_i/q$'s, respectively,
the equations $(\ref{eq:prophat})$ and $(\ref{eq:prop})$ reduce to
\smash{$\varphi(\Lambda)\bigl({\mathcal A}_L^{(s)}\bigr)^{-1}q^{-\ha\Delta}$}
with ${\mathcal A}_L^{(s)}$ in $(\ref{eq:A_L})$ and $(\ref{eq:A_Lcom})$,
 respectively.
\end{Remark}

\section[Mass truncation and finite-dimensional R matrix]{Mass truncation and finite-dimensional $\boldsymbol{R}$ matrix}
\label{sec:mass-truncation}

In this section we study the $\glN$ equation \eqref{eq:EigenEq}
by imposing a truncation condition \eqref{mass-truncation} on mass parameters.
We show that the normal ordered Hamiltonian \eqref{def:normal-ordered} gives rise to
the finite-dimensional $R$-matrix for the symmetric representation of \smash{$U_q\bigl(A_{N-1}^{(1)}\bigr)$}.
It is quite remarkable that the resulting $R$-matrix is also related the three-dimensional (tetrahedron) $R$-matrix \cite{Kuniba-book}.
Namely, in the formula of the components of the three-dimensional $R$-matrix there appears
a~basic building block $\Phi_q$ defined by~\eqref{def:Phi_q} (see~\cite[Section~13.5]{Kuniba-book}).
Exactly the same function is obtained in our formula of the components of the $R$-matrix (see Corollary~\ref{Cor:R-component}).
We note that the mechanism for deriving these formulas is common to both cases.

\subsection{The mass truncation}

We can recast the normal ordered form of the non-stationary $\glN$ equation as follows:
\begin{equation}
{:}\prod_{i=1}^N \frac{\varphi(\hat{x}_i)}{\varphi(d_i x_i)}{:}\, q^{\frac{1}{2}\Delta}\cdot \mathsf{T} \psi
= {:}\prod_{i=1}^N \frac{\varphi(\overline{d}_i x_i)}{\varphi(\check{x}_i)}{:}\, q^{-\frac{1}{2} \Delta} \cdot \psi.
\end{equation}
Here the normal ordering $: \quad :$ is defined as
\begin{equation}
{:}F(x,\vartheta){:}\, x^\nu=F(x,\nu)x^\nu,
\end{equation}
for any commutative function $F(x,\vartheta)=F(\{x_a\}, \{\vartheta_a\})$ and monomial $x^{\nu}=\prod_{a=1}^N x_a^{\nu_a}$.
By the $q$-binomial theorem
\begin{equation}\label{q-binomial}
\sum_{n=0}^\infty \frac{(a;q)_n}{(q;q)_n} z^n = \frac{\varphi(az)}{\varphi(z)},
\end{equation}
we have
\begin{align}
& \sum_{\alpha_1, \dots, \alpha_N \geq 0} {:} \prod_{i=1}^N (d_i x_i)^{\alpha_i}
\frac {(\overline{d}_i q^{\vartheta_i - \vartheta_{i-1}};q)_{\alpha_i}}{(q;q)_{\alpha_i}} {:}\,
q^{\frac{1}{2}\Delta}\cdot \mathsf{T} \psi \nonumber\\
& \qquad = \sum_{\alpha_1, \dots, \alpha_N \geq 0} {:} \prod_{i=1}^N x_i^{\alpha_i}
 \frac {(\overline{d}_i q^{\vartheta_i - \vartheta_{i-1}};q)_{\alpha_i}}{(q;q)_{\alpha_i}
(q^{\vartheta_i - \vartheta_{i-1}})^{\alpha_{i}}} {:}\,
q^{-\frac{1}{2}\Delta}\cdot \psi.\label{normal-ordered-form}
\end{align}

Let us impose the mass truncation condition,
\begin{equation}\label{mass-truncation}
\overline{d}_i = q^{-m_i}, \qquad m_i \in \bbZ_{\geq 0}, \qquad 1 \leq i \leq N.
\end{equation}
Set $M:= m_1 + m_2 + \cdots +m_N$. Note that $\overline{d}_1 \overline{d}_2 \cdots \overline{d}_N =q^{-M}$.
Under the condition \eqref{mass-truncation}, the coefficient for $x_i^{\alpha_i}$ in \eqref{normal-ordered-form}
vanishes for $\alpha_i \geq 1+m_i-\vartheta_i + \vartheta_{i-1}$.
By using
\begin{equation}
\sum_{\alpha=0}^{n} \frac{(q^{-n};q)_\alpha}{(q;q)_\alpha} z^{\alpha} = (q^{-n}z;q)_n,
\end{equation}
we obtain the following statement.

\begin{Proposition}\label{terminated-eq}
After the mass truncation \eqref{mass-truncation}, the non-stationary $\glN$ equation becomes\footnote{
$\mu_i := d_i$ are the remaining mass parameters after the mass truncation. When $N=2$, this should be compared with
equation~(2.12) in~\cite{AHKOSSY2}.}
\[
{:} \prod_{i=1}^N \bigl(q^{-m_i + \vartheta_i - \vartheta_{i-1}} \mu_i x_i;q\bigr)_{m_i -\vartheta_i + \vartheta_{i-1}} {:}\, q^{\frac{1}{2}\Delta}\cdot \mathsf{T} \psi
= {:} \prod_{i=1}^N (q^{-m_i} x_i;q)_{m_i - \vartheta_i + \vartheta_{i-1}} {:}\, q^{-\frac{1}{2} \Delta} \cdot \psi,
\]
for the terminated function
\begin{equation}\label{terminated}
\psi = \sum_{\substack{\theta_1, \dots, \theta_N \geq 0, \\ \theta_i - \theta_{i-1} \leq m_i}}
c_{\theta_1, \dots, \theta_N} x_1^{\theta_1} \cdots x_N^{\theta_N}.
\end{equation}
\end{Proposition}

Recall that we identify $\Lambda = x_1x_2 \cdots x_N$ as the parameter of the instanton expansion.
If we fix $\theta_N \geq 0$ and regard it as the instanton number,
the number of terms in the terminated expansion~\eqref{terminated} is finite.
They are labeled by the set $S(m) \subset \mathbb{Z}^{N-1}$ defined by
\begin{gather*}
S(m) := \{ (\theta_1, \dots, \theta_{N-1}) \mid
\theta_1 \leq m_1, \theta_2 \!-\! \theta_1\hspace{-0.5pt} \leq\hspace{-0.5pt} m_2, \dots, \theta_{N-1}\!-\! \theta_{N-2}\hspace{-0.5pt} \leq\hspace{-0.5pt} m_{N-1}, -\theta_{N-1\hspace{-0.5pt}} \leq\hspace{-0.5pt} m_N \}.
\end{gather*}
For example, when $N=3$,
\[
S(m) = \bigl\{ (\theta_1, \theta_2) \in \mathbb{Z}^2 \mid \theta_1 \leq m_1, \theta_2 - \theta_1 \leq m_2, - \theta_2 \leq m_3 \bigr\},
\]
and the allowed $(\theta_1, \theta_2)$ in $\mathbb{Z}^2$-lattice is bounded
by the triangle (see Figure~\ref{Fig:shifted triangle} in Appendix~\ref{App:Combi}).
For general $N$, we have \smash{$\vert S(m) \vert = \binom{M+N-1}{N-1} = \binom{M+N-1}{M}$}.

Let us make a shift $\widetilde{\theta}_i := \theta_i + m_{i+1}+ \cdots + m_N$ so that
the defining conditions for $S(m)$ can be written by
$0 \leq \widetilde{\theta}_{N-1} \leq \widetilde{\theta}_{N-2} \leq \cdots \leq
\widetilde{\theta}_2 \leq \widetilde{\theta}_1 \leq M$.
See also Appendix~\ref{App:Combi} for the meaning of such a shift.
Then define $i_1 := M -\widetilde{\theta}_1, i_2 := \widetilde{\theta}_1 - \widetilde{\theta}_2,
\cdots, i_{N-1}:= \widetilde{\theta}_{N-2} - \widetilde{\theta}_{N-1}, i_N := \widetilde{\theta}_{N-1}$,
then $\mathbf{i}=(i_1, i_2, \dots, i_N)$ belongs to the set
\begin{equation}\label{IM-def}
I_M :=\bigl\{ {\bf i}=(i_1,i_2, \dots, i_N) \in \bbZ_{\geq 0}^N \mid
 i_1+i_2+\cdots +i_N=M \bigr\}.
\end{equation}
In fact we can define a bijection between $S(m)$ and $I_M$ as follows:
let us define $z_k := x_1 x_2 \cdots x_k$, $k=1, \dots, N-1$ and set
\[
z_1^{\alpha_1} z_2^{\alpha_2} \cdots z_{N-1}^{\alpha_{N-1}} = x_1^{\widetilde{\theta}_1}
\cdots x_{N-1}^{\widetilde{\theta}_{N-1}}.
\]
Then we have $\alpha_k = i_{k+1} ,k=1, \dots, N-1$ and $\alpha_1 + \alpha_2 + \cdots +\alpha_{N-1} \leq M$.
Introducing $\alpha_N := M - (\alpha_1 + \alpha_2 + \cdots +\alpha_{N-1}) = i_1$, we see that the elements of $S(m)$ are
in one to one correspondence with the monomials $z_1^{\alpha_1} \cdots z_{N-1}^{\alpha_{N-1}} z_N^{\alpha_N},~
z_N = \Lambda$ with homogeneous degree~$M$.

\subsection[R matrix as a connection matrix]{$\boldsymbol{R}$ matrix as a connection matrix}

\def\bfbeta{\boldsymbol\beta}
\def\bfgamma{\boldsymbol\gamma}

We study the finite-dimensional matrix $R$ arising from the $\glN$ Hamiltonian by the mass truncation.
For $N, M \in \bbZ_{\geq 0}$, we have introduced $I^N_M = I_M$ in the last subsection (see \eqref{IM-def}).
Note that $|I_{M}|=\binom{N+M-1}{M}$.
For variables $\mathbf{z}=(z_1, z_2, \dots, z_{N})$ and
parameters\footnote{These are the remaining mass parameters after the mass truncation.}
$(\mu_1, \dots, \mu_N)$, we define the polynomials $B_{k, {\bf i}}$ ($k=1,2$) as
\begin{align}\label{eq:z-basis}
&B_{1,{\bf i}}(\mathbf{z},\Lambda)=\prod_{a=1}^N \left(\mu_a \frac{z_{a+1}}{z_a};q \right)_{i_a} z_a^{i_a}, \qquad
B_{2,{\bf i}}(\mathbf{z},\Lambda)=\prod_{a=1}^N \left(\frac{z_{a}}{z_{a+1}};q \right)_{i_a} z_{a+1}^{i_a}.
\end{align}
Here and in the followings, we always put $z_{N+1}=\Lambda z_1$ regarding $\Lambda$ as a free parameter.\footnote{We do not assume, for example, $\Lambda = x_1 x_2 \cdots x_N$.}
Note that we can change the normalization of the base polynomials~\eqref{eq:z-basis} freely keeping the main
structure of the matrix $R$. See the remark at the end of the section.
For generic $\Lambda$, both $\{B_{1,{\bf i}}(\mathbf{z},\Lambda) \mid {\bf i} \in I_M\}$ and
$\{B_{2,{\bf i}}(\mathbf{z},\Lambda) \mid {\bf i} \in I_M\}$ form a basis of the homogeneous polynomials
of degree $M$ in $\bbC[z_1, \dots, z_N]$. Hence we have a relation
\begin{equation}\label{eq:BRB}
B_{1,{\bf i}}(\mathbf{z},\Lambda)=\sum_{{\bf j} \in I_M} R_{{\bf i},{\bf j}}(\Lambda) B_{2,{\bf j}}(\mathbf{z},\Lambda).
\end{equation}
The coefficients $R(\Lambda)_{{\bf i},{\bf j}}$ are polynomial in $\mu_a$ and rational in $\Lambda$ and $q$.

Since the size of $R$-matrix is $|I_M|$, one can determine $R_{{\bf i},{\bf j}}(\Lambda)$ by specializing
\eqref{eq:BRB} at $|I_M|$ points.
It is convenient to choose such $|I_M|$ reference points
$z_{{\bf k}}$ as follows:
\begin{equation}\label{eq:zjsp}
z_{{\bf k},1}=1, \qquad z_{{\bf k},a}=q^{k_1+\cdots+k_{a-1}}, \qquad 1<a\leq N,
\end{equation}
with ${\bf k} = (k_1, k_2, \dots, k_N) \in I_M$.
We will solve the matrix equation
\[
B_{1,{\bf i}}(z_{{\bf k}},\Lambda)=\sum_{{\bf j} \in I_M} R_{{\bf i},{\bf j}}(\Lambda) B_{2,{\bf j}}(z_{{\bf k}}, \Lambda).
\]
As we will see in Proposition~\ref{prp:invB}, the inverse of the matrix $\big(B_{2,{\bf i}}(z_{{\bf k}},\Lambda)\big)_{{\bf i},{\bf k}}$
is obtained explicitly, hence one can derive an explicit formulae of $R_{{\bf i},{\bf j}}(\Lambda)$.

To describe the inversion formulae, we prepare some notations. Let $n \geq 1$.
For any sequences of integers ${\bf i}=(i_1, i_2, \dots, i_n)$, ${\bf j}=(j_1, j_2, \dots, j_n)\in \bbZ^n$
of length $n$, we put
\begin{align}
|{\bf i}|=\sum_{a=1}^n i_a, \qquad
\overline{\bf i}=(i_1, \dots, i_{n-1}), \qquad
\langle {\bf i}, {\bf j} \rangle=\sum_{1\leq a<b\leq n} i_a j_b.
\end{align}
For $\bfbeta, \bfgamma \in \bbZ^{n}$ and $\lambda, \mu \in \bbC$, we define\footnote{
See equations~(13.49) and~(13.50) in \cite{Kuniba-book} (or equations~(19) and (20) in \cite{Kuniba:2018qzk}).}
\begin{align}\label{def:Phi_q}
&\Phi_q({\bfgamma}|{\bfbeta}; \lambda, \mu)
=q^{\langle \bfbeta-\bfgamma, \bfgamma \rangle} \Big(\frac{\mu}{\lambda}\Big)^{|\bfgamma|}
\frac{(\lambda;q)_{|\bfgamma|}(\frac{\mu}{\lambda};q)_{|\bfbeta|-|\bfgamma|}}{(\mu;q)_{|\bfbeta|}}
\prod_{a=1}^{n} \left[{ \Frac{\beta_a}{\gamma_a} }\right]_q,
\end{align}
where the $q$-binomial coefficients are define by \eqref{q-binom}.

Note that $\Phi_q({\bfgamma}|{\bfbeta}; \lambda, \mu)=0$ unless $\bfgamma \leq \bfbeta$ (i.e., $\forall a$ $\gamma_a\leq \beta_a$).
The function $\Phi_q$ and the function~\smash{$A_{{\bf i},{\bf j}}^{{\bf a}, {\bf b}}$} defined below
 (see \eqref{eq:APhiPhi}), which is quadratic in $\Phi_q$,
originate in the study of the three-dimensional $R$-matrix, where it was shown that the trace reduction of the three-dimensional~$R$-matrix gives fundamental examples of the quantum $R$-matrix of \smash{$U_q\bigl(A_{N-1}^{(1)}\bigr)$} with higher ``spin'' representations.
See \cite[Chapter~13]{Kuniba-book} and references therein.

\begin{Proposition}
For any ${\bf i}, {\bf k} \in \bbZ_{\geq 0}^n$ and $a,b,c \in \bbC$, the function $\Phi_q$ satisfies
the transition property\footnote{
A similar formula
$\sum_{{\bf j}} \Phi_q({\bf i}|{\bf j};b,c)\Phi_q({\bf j}|{\bf k};a,b)=\Phi_q({\bf i}|{\bf k};a,c)$ also seems to be true.}
\begin{equation}\label{eq:Phi-trans}
\sum_{{\bf i}\leq {\bf j}\leq {\bf k}} \Phi_q({\bf i}|{\bf j};a,b)\Phi_q({\bf j}|{\bf k};b,c)=\Phi_q({\bf i}|{\bf k};a,c).
\end{equation}
\end{Proposition}

\begin{proof}
Let ${\bf i}, {\bf j}, {\bf k} \in \bbZ_{\geq 0}^n$ and ${\bf i}_l$, ${\bf j}_l$, ${\bf k}_l$
be their truncations to the first $l$ components. Assuming $\Phi_q({\bf i}_l|{\bf k}_l;a,c)\neq 0$, we put
\begin{equation}
F_l=\frac{ \Phi_q({\bf i}_l|{\bf j}_l;a,b)\Phi_q({\bf j}_l|{\bf k}_l;b,c)}{\Phi_q({\bf i}_l|{\bf k}_l;a,c)} \quad (l\geq 1),
\qquad F_0=1.
\end{equation}
It is not difficult to see
\begin{align}
\frac{F_l}{F_{l-1}}=
&\frac{(u;q)_s(v;q)_{k-s}}{(q)_s(q)_{k-s}}v^s \frac{(q)_k}{(uv;q)_k}, \qquad l\geq 1,
\end{align}
where
\begin{align}
&u=\frac{b}{a} q^{\alpha}, \qquad v=\frac{c}{b} q^{\beta-\alpha},
\qquad s=j_l-i_l, \qquad k=k_l-j_l, \nonumber \\
&\alpha=\sum_{a=1}^{l-1}(j_a-i_a), \qquad \beta=\sum_{a=1}^{l-1}(k_a-j_a).
\end{align}
Then, the $q$-binomial formula \eqref{q-binomial} implies
\begin{align*}
\sum_{l=0}^{\infty}\sum_{s=0}^{k}\frac{(u)_s(v)_{k-s}}{(q)_s(q)_{k-s}}(vx)^s x^{k-s}
=\frac{(u v x;q)_\infty}{(vx;q)_{\infty}}\frac{(vx;q)_{\infty}}{(x;q)_{\infty}}
=\frac{(u v x;q)_\infty}{(x;q)_{\infty}}=\sum_{l=0}^{\infty} \frac{(uv;q)_k}{(q)_k} x^k.
\end{align*}
Comparing the coefficients of $x^l$, we have
\begin{equation}
\sum_{j_l=i_l}^{k_l} \frac{F_l}{F_{l-1}}=1, \quad {\rm i.e.} \quad
\sum_{j_l=i_l}^{k_l} F_l=F_{l-1} \quad (l \geq 1).
\end{equation}
By iterating this, the desired relation $\sum_{{\bf j}} F_n=F_0=1$ is obtained.
\end{proof}

\begin{Lemma}
The specializations $B_{1,{\bf i}}(z_{\bf j},\Lambda)$, $B_{2,{\bf i}}(z_{\bf j},\Lambda)$ are given as follows:
\begin{align}\label{eq:B1sp}
&B_{1,{\bf i}}(z_{\bf j},\Lambda)=q^{\langle {\bf c}, {\bf i} \rangle}\frac{\bigl(q^{-|{\bf c}|}\Lambda;q\bigr)_M}{(q;q)_M}
\prod_{a=1}^{N} (q;q)_{i_a} \cdot
\Phi_q(\overline{{\bf c}-{\bf i}-{\bf j}}|\overline{{\bf c}-{\bf j}}; q^{M-|{\bf c}|}\Lambda, q^{-|{\bf c}|}\Lambda),\\
\label{eq:B2sp}
&B_{2, {\bf i}}(z_{\bf j}, \Lambda)=N_{{\bf i}}(\Lambda)
\Phi_q(\overline{{\bf i}}|\overline{{\bf j}}; q^{-M},\Lambda^{-1}),
\quad
N_{{\bf i}}(\Lambda)=\frac{(\Lambda^{-1};q)_M \Lambda^M}{(q;q)_M}\prod_{a=1}^{N} (q;q)_{i_a},
\end{align}
where ${\bf i}, {\bf j} \in I_M$ and $\mu_a=q^{-c_a}$.
\end{Lemma}

\begin{proof}
A direct computation.
Note that the expression \eqref{eq:B1sp} is valid also for $c_a \in \bbC$ where $\Phi_q({\bfgamma}|{\bfbeta}; \lambda, \mu)$ with $\lambda=\mu q^M$ is expressed as
\begin{equation}
\Phi_q\bigl({\bfgamma}|{\bfbeta}; \mu q^M, \mu\bigr)
=q^{\langle \bfbeta-\bfgamma, \bfgamma \rangle-M{|\bfgamma|}}
\frac{\bigl(\mu q^{|\bfgamma|};q\bigr)_{M}}{(\mu;q)_M}\frac{\bigl(\frac{\mu}{\lambda};q\bigr)_{|\bfbeta|-|\bfgamma|}}{\bigl(\mu q^{|\bfgamma|};q\bigr)_{|\bfbeta|-|\bfgamma|}}
\prod_{a=1}^{n} \frac{\bigl(q^{\gamma_a+1};q\bigr)_{\beta_a-\gamma_a}}{(q;q)_{\beta_a-\gamma_a}},
 \end{equation}
by analytical continuation.
\end{proof}

From \eqref{eq:B2sp}, we see $B_{2, {\bf i}}(z_{\bf j}, \Lambda)=0$ unless $\overline{\bf i} \leq \overline{\bf j}$.
Moreover, the transition property \eqref{eq:Phi-trans} with $a=c$ implies the following.

\begin{Proposition} \label{prp:invB} Let $B_{{\bf i},{\bf j}}=B_{2, {\bf i}}(z_{\bf j}, \Lambda)$
and
\begin{equation}\label{eq:iB2sp}
B'_{{\bf i},{\bf j}}=
N_{{\bf j}}(\Lambda)^{-1} \Phi_q\bigl(\overline{{\bf i}}|\overline{{\bf j}}; \Lambda^{-1}, q^{-M}\bigr),
\end{equation}
then the matrix $(B'_{{\bf i},{\bf j}})$ gives the inverse of $(B_{{\bf i},{\bf j}})$.
\end{Proposition}

Thanks to \eqref{eq:B1sp} and the inversion formulae~\eqref{eq:iB2sp}, we have the following.

\begin{Corollary} \label{Cor:R-component} For ${\bf i}, {\bf j} \in I_M$
the coefficients $R_{{\bf i}, {\bf j}}(\Lambda)$ are given by
\begin{align}
&R_{{\bf i}, {\bf j}}(\Lambda)=C\frac{\bigl(\Lambda q^{-|{\bf c}|};q\bigr)_M}{\bigl(\Lambda q^{-M+1};q\bigr)_M}
\prod_{a=1}^N \frac{(q;q)_{i_a}}{(q;q)_{j_a}} A_{{\bf c}-{\bf j},{\bf j}}^{{\bf c}-{\bf i}, {\bf i}},
\qquad{\bf c}={\bf a}+{\bf b}={\bf i}+{\bf j},
\label{eq:RbyA}\\
&C=(-1)^M q^{\frac{M(|{\bf c}|-M+1)}{2} + \frac{1}{4}\sum_{a=1}^N \left(\ctwo{i_a}-\ctwo{j_a}\right)
+\frac{{\bf c}\cdot({\bf j}-2{\bf i})}{2}
+\frac{\langle {\bf j}-{\bf i},{\bf c} \rangle}{2}},
\qquad \mu_a=q^{-c_a}.
\end{align}
Here the function $A$ is given by\footnote{See equation~(13.51) in \cite{Kuniba-book}, where
$\bigl(q^2, z,l,m\bigr)_{\rm there} =\bigl(q,q^{-|{\bf c}|/2}\Lambda, |{\bf c}|-M,M\bigr)_{\rm here} $.}
\begin{align}\label{eq:APhiPhi}
&A_{{\bf i},{\bf j}}^{{\bf a}, {\bf b}}=q^{\frac{\langle{\bf i},{\bf j}\rangle-\langle{\bf b},{\bf a}\rangle}{2}}
\sum_{{\bf k} \in \bbZ^{N-1}}
\Phi_q \left({\bf k}|\overline{\bf j};\frac{1}{\Lambda },q^{-M}\right)
\Phi_q \bigl(\overline{\bf a}-{\bf k}|\overline{\bf c}-{\bf k};\Lambda q^{M-|{\bf c}|},\Lambda q^{-|{\bf c}|}\bigr).
\end{align}
\end{Corollary}

In \cite{Kuniba-book}, the same function \smash{$A_{{\bf i},{\bf j}}^{{\bf a}, {\bf b}}$} was introduced in the formula of
a trace reduction of the tetrahedron (3D) $R$-matrix \smash{$R=R_{i,j,k}^{a,b,c}$}:
\begin{equation}
R^{{\rm tr}_3}(z)_{{\bf i}, {\bf j}}^{{\bf a}, {\bf b}} :=
\sum_{c_1, \dots, c_N \geq 0} z^{c_1}
 R_{i_1,j_1,c_2}^{a_1,b_1,c_1}
R_{i_2,j_2,c_3}^{a_2,b_2,c_2}\cdots
R_{i_N,j_N,c_1}^{a_N,b_N,c_N},
\end{equation}
where ${\bf i}, {\bf j}, {\bf a}, {\bf b} \in I_{M}$.
By the weight conservation ${\bf c}:={\bf i}+{\bf j}={\bf a}+{\bf b}$,
one can regard \smash{$R^{{\rm tr}_3}(z)_{{\bf i}, {\bf j}}^{{\bf a}, {\bf b}}$} as
a function $R(z; {\bf c})_{\bf j,\bf b}$ of ${\bf j}, {\bf b} \in I_M$ and ${\bf c} \in {\bbZ_{\geq 0}}^N$.
The dependence on the parameter ${\bf c}$ seems to be analytically continued to polynomials in $q^{-c_i}$.

The following two results are relevant to our current problem.
\begin{Theorem}[{\cite[Theorem 13.3]{Kuniba-book}}]\label{Kuniba-1}
For $\mathbf{a}, \mathbf{i} \in I_\ell$ and $\mathbf{b}, \mathbf{j} \in I_m$,\footnote{In the present case we take $\ell=m=M$. In \cite{Kuniba-book},
 $I_\ell$, $I_m$ are denoted by $B_\ell$, $B_m$.} we have
\[
\Lambda_{\ell, m}(z,q)^{-1} R^{{\rm tr}_3}(z)_{\mathbf{i}, \mathbf{j}}^{\mathbf{a}, \mathbf{b}}
= \delta_{\mathbf{i}+ \mathbf{j}}^{\mathbf{a} + \mathbf{b}} A(z)_{\mathbf{i}, \mathbf{j}}^{\mathbf{a}, \mathbf{b}},
\]
where
\[
\Lambda_{\ell, m}(z,q) = (-1)^m q^{m(\ell +1)} \frac{\bigl(q^{-\ell-m}z;q^2\bigr)_m}{\bigl(q^{\ell-m}z;q^2\bigr)_{m+1}}.
\]
\end{Theorem}

\begin{Theorem}[{\cite[Theorem 13.10]{Kuniba-book}}]\label{Kuniba-2}
Up to normalization $R^{{\rm tr}_3}(z)$ coincides with the quantum $R$-matrix of
\smash{$U_q\bigl(A_{N-1}^{(1)}\bigr)$} as follows:
\[
R^{{\rm tr}_3}_{\ell, m} (z) = \mathcal{R}_{\ell\varpi_1, m \varpi_1}\bigl(z^{-1}\bigr),
\]
where $k \varpi_1$ stands for the $k$-th symmetric representation of \smash{$U_q\bigl(A_{N-1}^{(1)}\bigr)$}.
\end{Theorem}
Combining Corollary~\ref{Cor:R-component} and these two theorems,
we see that $R_{{\bf i}, {\bf j}}(\Lambda)$ is nothing but the $R$-matrix of \smash{$U_q\bigl(A_{N-1}^{(1)}\bigr)$}
 for the symmetric representations presented in~\cite{Bosnjak:2016oze}.

 \begin{Remark} In a similar manner, one can compute the relation between two $B_{2,{\bf i}}$ polynomials with
 different $\Lambda$ as
 \begin{equation}
B_{2, {\bf i}}(\mathbf{z},\Lambda)= \sum_{{\bf j}\in I_M}
\frac{N_{{\bf i}}(\Lambda)}{N_{{\bf j}}(\Lambda')}\Phi_q\left(\overline{\bf i}|\overline{\bf j};
\frac{1}{\Lambda'},
\frac{1}{\Lambda}\right) B_{2,{\bf j}}(\mathbf{z}, \Lambda').
\end{equation}
In view of this, the transition property \eqref{eq:Phi-trans} is obvious.
\end{Remark}

The coefficients $R_{{\bf i},{\bf j}}(\Lambda)$ explicitly given by \eqref{eq:RbyA} are to be related to
the Hamiltonian without the shift operator~$\mathsf{T}$ of the mass truncated \smash{$\glN$} equation in Proposition~\ref{terminated-eq}.
Actually they are related by a gauge transformation of the form $R \longrightarrow K^{-1} R L$, where $K$ and $L$ are
diagonal matrices, which corresponds to a change of the normalization of a basis of homogeneous
polynomials. The diagonal components of $K$ and $L$ come from the $q$-Borel transformation on monomials in $\mathbf{z}$
and the inversion of the $q$-factorial, hence they are of the form $K= \mathrm{diag} \bigl(q^{\gamma_i}\bigr)$
and $L= \mathrm{diag} \bigl(q^{\delta_i}\bigr)$, where $\gamma_i$ and $\delta_i$ are at most quadratic in the powers
$\mathbf{i} \in I_M$ of $\mathbf{z}$.

\section{Affine Laumon partition function as Jackson integral}\label{sec:affine-Laumon}

When we impose the mass truncation, the affine Laumon partition function is
represented as a~Jackson integral, which was a key to the relation to the $q$-KZ equation.
See \cite{Aganagic:2013tta} for the general method of representing Nekrasov partition functions
as Jackson integrals with the help of the truncation of Young diagrams to finite length.
In this section by recasting the affine Laumon partition function as a Jackson integral
we will show that it is identified with an $N+2$ point correlation function,
which provides a solution to the $q$-KZ equation of \smash{$U_q\bigl(\widehat{\mathfrak{sl}}_2\bigr)$} \cite{AMatsuo2,VarchenkoCMP}.
In the last section we have seen that after the mass parameter truncation
the Hamiltonian of the non-stationary $\glN$ difference equation
is related to the finite-dimensional $R$-matrix of \smash{$U_q\bigl(\widehat{\mathfrak{sl}}_N\bigr)$} (see Corollary~\ref{Cor:R-component}
and Theorems~\ref{Kuniba-1}, \ref{Kuniba-2}).
On the other hand what we will see in this section is the affine Laumon partition function, which is conjecture to
be a solution to the non-stationary \smash{$\glN$} difference equation, actually gives a solution to the $q$-KZ equation of
\smash{$U_q\bigl(\widehat{\mathfrak{sl}}_2\bigr)$} (see Proposition~\ref{Laumon= Jackson} and \cite{MIto, ItoNoumi}).
At the moment, unfortunately, we are not able to fill the gap between \smash{$U_q\bigl(\widehat{\mathfrak{sl}}_N\bigr)$} on the Hamiltonian side and
\smash{$U_q\bigl(\widehat{\mathfrak{sl}}_2\bigr)$} on the wave function side.
It is natural to expect a duality of~$q$\nobreakdash-KZ equations between \smash{$U_q\bigl(\widehat{\mathfrak{sl}}_2\bigr)$} and
\smash{$U_q\bigl(\widehat{\mathfrak{sl}}_N\bigr)$}. Towards a proof of our conjecture such a~duality may play an important role.

\subsection{Jackson integral of symmetric Selberg type}
The Jackson integral of symmetric Selberg type is\footnote{We have changed the notation $(n,m) \to (M, N)$ from \cite{MIto}.}
\begin{align}
\langle \phi, \xi \rangle:={}& \int_{0}^{\xi \infty}
\phi(z) \Phi_{M,N}(z) \Delta(z) \frac{{\rm d}_tz_1}{z_1} \wedge \cdots \wedge \frac{{\rm d}_tz_M}{z_M} \nonumber\\
={}&(1-t)^M \sum_{\nu \in \bbZ^M} [ \phi(z)\Phi_{M,N}(z) \Delta(z)]_{z_i=\xi_i t^{\nu_i}},\label{Jackson-Selberg}
\end{align}
where
\begin{equation}\label{weight-MN}
\Phi_{M,N}(z) := \prod_{i=1}^{M} \left[ z_i^\alpha \prod_{r=1}^{N} \frac{\bigl(t \mathsf{a}_r^{-1}z_i;t\bigr)_\infty}{(\mathsf{b}_rz_i;t)_\infty} \right]
\prod_{1 \leq j < k \leq M} z_j^{2\tau -1} \frac{\bigl(tq^{-1} z_k/z_j;t\bigr)_\infty}{(q z_k/z_j; t)_\infty},
\end{equation}
with $q=t^{\tau}$,
\begin{equation}
\Delta(z) := \prod_{1 \leq i < j \leq M} (z_i-z_j),
\end{equation}
and $\phi(z)$ is an arbitrary symmetric function on $(\mathbb{C}^\times)^{M}$ \cite{MIto}.
For later convenience, we also define
\begin{equation}
\Delta(q,z) := \prod_{1 \leq i < j \leq M} \bigl(z_i- q^{-1} z_j\bigr), \qquad \Delta(z) = \Delta(1,z).
\end{equation}
Here as in \cite{AHKOSSY2}, we have exchanged $q$ and $t$, because
the shift parameter of the quantum Knizhnik--Zamolodchikov ($q$-KZ) equation of our interest is $t$, while it was $q$ in \cite{MIto}.
In the context of two-dimensional conformal field theory, $N$ is the number of the vertex operators in the correlation function
and $(\mathsf{a}_r, \mathsf{b}_r)$ are evaluation parameters of each vertex operator of \smash{$U_q\bigl(A_1^{(1)}\bigr)$}.
On the other hand, each integration variable $z_i$ is associated with the screening operator inserted in the correlation function.
Hence, the Jackson integral~\eqref{Jackson-Selberg} corresponds to $N$-point correlation function (or $N+2$ point vacuum expectation value)
with the insertion of~$M$ screening operators.

Aomoto and Kato \cite{Aomoto-Kato1,Aomoto-Kato2} showed that the Jackson integral \eqref{Jackson-Selberg} satisfies
a $q$-difference system of rank \smash{$\binom{N+M-1}{N-1}$}.
In \cite{MIto} the case $N=2$ was studied and the $R$-matrix appearing in the $q$-difference system of rank $M+1$ is
identified with the $R$-matrix for the symmetric tensor representation of \smash{$U_q\bigl(A_1^{(1)}\bigr)$} obtained in \cite{Bosnjak:2016oze,Kuniba:2018qzk}. Later in Section~\ref{sec:Jackson-cocycle},
we will see that up to a~pseudo-constant the weight function $\Phi_{M,N}(z)$ arises from the affine Laumon partition function
after the mass parameter truncation (see~\eqref{weight-function}).
Moreover, the rank of a basis (see~\eqref{cocycle-base} and~\eqref{symmetrization})
of the cocycle function~$\phi(z)$ is exactly the same as that of the $q$-difference system of Aomoto--Kato.

\subsection{Affine Laumon partition function}

The affine Laumon partition function of type \smash{$A_{N-1}^{(1)}$} is expressed
as a summation over $N$-tuples of partitions $\vec{\lambda} = \bigl(\lambda^{(1)}, \dots, \lambda^{(N)}\bigr)$:
\begin{align}
&\mathcal{Z}_{\mathrm{AL}}^{\glN}
\left( \left.\left.\begin{array}{@{}ccc@{}}a_1,\dots, a_N \\b_1, \dots ,b_N\\c_1,\dots, c_N\end{array}
\right| \mathsf{x}_1, \dots, \mathsf{x}_N \right|q, \kappa \right) \nonumber\\
&\qquad = \sum_{\vec{\lambda}}
\prod_{i,j=1}^N \frac{\Nk^{(j-i\vert N)}_{\varnothing,\lambda^{(j)}}(a_i/b_j|q, \kappa)
\Nk^{(j-i\vert N)}_{\lambda^{(i)},\varnothing}(b_i/c_j|q, \kappa)}
{\Nk^{(j-i\vert N)}_{\lambda^{(i)},\lambda^{(j)}}(b_i/b_j|q, \kappa)} \cdot \mathsf{x}_1^{k_1(\vec{\lambda})} \cdots \mathsf{x}_N^{k_N(\vec{\lambda})},\label{glnAL}
\end{align}
where $\mathsf{N}_{\lambda, \mu}^{(k\vert N)} (u \vert q, \kappa)$ is the orbifolded Nekrasov factor (see Definition~\ref{def:affine-Laumon}).
The powers of the expansion parameters $\mathsf{x}_i$ are given by the number of boxes with a fixed color (see \eqref{box-counting}).

\subsubsection{Exchange symmetry of the orbifolded Nekrasov factor}

The orbifolded Nekrasov factor $\mathsf{N}_{\lambda, \mu}^{(k\vert N)} (u \vert q,t)$ has the exchange symmetry.
We employ the infinite product form of the orbifolded Nekrasov factor
obtained in \cite[Appendix~F]{AHKOSSY2}\footnote{Originally the orbifolded Nekrasov factor is defined
in terms of the finite $q$-shifted factorial $[u;q]_n$. But to see the relation to the Jackson integral it is
convenient to employ the $t=\kappa^{-\frac{1}{N}}$-shifted factorial $[u;t]_\infty$.}
\beq\label{inf-prod-form}
\mathsf{N}_{\lambda, \mu}^{(k\vert N)} \bigl(u \vert q, t^{-\frac{1}{N}}\bigr)
= \prod_{i,j=1}^\infty
\frac{\bigl[uq^{j-i} t^{1 -\frac{k}{N} + \floor{\frac{\mu_{i}^\vee + k -\lambda_j^\vee}{N}}} ; t\bigr]_\infty}
{\bigl[uq^{j-i-1} t^{1 -\frac{k}{N} + \floor{\frac{\mu_{i}^\vee + k -\lambda_j^\vee}{N}}} ; t\bigr]_\infty}
\frac{\bigl[uq^{j-i-1} t^{1 -\frac{k}{N}}; t\bigr]_\infty} {\bigl[uq^{j-i} t^{1 -\frac{k}{N}}; t\bigr]_\infty},
\eeq
where $[u;t]_\infty$ is a regularized version of $[u;q]_n$ as $n \to \infty$, which is defined by
\beq\label{sinh-reg}
[u;t]_\infty := \frac{(u;t)_\infty}{\vartheta_{t^{1/2}}\bigl(- u^{1/2}\bigr)}
= \frac{\bigl(u^{1/2}; t^{1/2}\bigr)_\infty}{\bigl(-t^{1/2}u^{-1/2}; t^{1/2}\bigr)_\infty}.
\eeq
with $\vartheta_p(z):= (z;p)_\infty (pz^{-1};p)_\infty$.
One can check
\begin{align}
\frac{[u;q]_\infty}{[q^n u ; q]_\infty} &=
\frac{\bigl(u^{1/2} ; q^{1/2}\bigr)_\infty}{\bigl(-q^{1/2} u^{-1/2}; q^{1/2}\bigr)_\infty}
\frac{\bigl(-q^{(1-n)/2} u^{-1/2}; q^{1/2}\bigr)_\infty}{\bigl(q^{n/2} u^{1/2} ; q^{1/2}\bigr)_\infty}
\nonumber\\
&= \bigl(u^{1/2}; q^{1/2}\bigr)_n \bigl(-q^{(1-n)/2} u^{-1/2}; q^{1/2}\bigr)_n = [u; q]_n.\label{sinh-ratio}
\end{align}

\begin{Proposition}
The orbifolded Nekrasov factor has the following symmetry:
\begin{equation}
\mathsf{N}_{\lambda, \mu}^{(k\vert N)} (u \vert q, \kappa)
= \pm \mathsf{N}_{\mu, \lambda}^{(N-k-1\vert N)} (q\kappa/u \vert q, \kappa),
\qquad \kappa = t^{-1/N},
\end{equation}
where the sign factor depends on the pair $(\lambda, \mu)$.
\end{Proposition}
\begin{proof}
By using \eqref{sinh-ratio} with $q$ replaced by $t$, we can recast \eqref{inf-prod-form} in the following form:
\beq
\mathsf{N}_{\lambda, \mu}^{(k\vert N)} (u \vert q, \kappa)
= \prod_{i,j=1}^\infty
\frac{\bigl[uq^{j-i-1} t^{1 -\frac{k}{N}}; t\bigr]_{n_{ij}}} {\bigl[uq^{j-i} t^{1 -\frac{k}{N}}; t\bigr]_{n_{ij}}},
\eeq
where we set $n_{ij}=\bigl\lfloor \frac{\mu_{i}^\vee + k -\lambda_j^\vee}{N}\bigr\rfloor \in \mathbb{Z}$ and it is understood
\begin{equation}\label{n-inversion}
[u;t]_{-n} = [t^{-n} u;t]_n^{-1}, \qquad n >0.
\end{equation}
Similarly by combining \eqref{n-inversion} with the inversion formula of the floor function
\begin{equation}\label{floor-inversion}
1+ \left\lfloor\frac{\ell}{N}\right\rfloor = - \left\lfloor \frac{-\ell-1}{N}\right\rfloor, \qquad \ell \in \mathbb{Z},
\end{equation}
 we obtain
\begin{align}
\mathsf{N}_{\mu, \lambda}^{(N-k-1\vert N)} (q\kappa/u \vert q, \kappa)
&= \prod_{i,j=1}^\infty
\frac{\bigl[u^{-1}q^{j-i} t^{\frac{k}{N}}; t\bigr]_{1+\floor{\frac{\lambda_{i}^\vee -k-1 -\mu_j^\vee}{N}}}}
{\bigl[u^{-1}q^{j-i+1} t^{\frac{k}{N}} ; t\bigr]_ {1+\floor{\frac{\lambda_{i}^\vee -k -1-\mu_j^\vee}{N}}}} \nonumber\\
&= \prod_{i,j=1}^\infty
\frac{\bigl[u^{-1}q^{i-j+1} t^{\frac{k}{N}-n_{ij}} ; t\bigr]_ {n_{ij}}}
 { \bigl[u^{-1}q^{i-j} t^{\frac{k}{N}-n_{ij}} ; t\bigr]_ {n_{ij}}}.
 \label{inversion-sign}
\end{align}
Finally, the formula
\beq\label{q-shifted-inversion}
(-1)^n \bigl[u^{-1},t\bigr]_n = \bigl[u; t^{-1}\bigr]_n = \bigl[u t^{-n+1} ; t\bigr]_n,
\eeq
implies the desired relation.
Note that the right-hand side of \eqref{inversion-sign} is actually a finite product
after the cancellation of the infinite product. In general the sign factors from the denominator
and the numerator do not cancel and we need the sign factor. See Appendix~\ref{App-C}
for an explicit formula of the sign factor from the data of $(\lambda, \mu)$.
\end{proof}

The exchange symmetry of the orbifolded Nekrasov factor implies that
the matter contribution to the affine Laumon partition function is symmetric under $d_i \leftrightarrow \overline{d}_{i}$.
For matter contribution, one of the partitions in the Nekrasov factor is empty and
we have general formulas of finite product form (see \cite[equations~(F.29) and (F.30)]{AHKOSSY2}):
\begin{align}
\mathsf{N}_{\lambda, \varnothing}^{(k\vert N)} (u \vert q, \kappa)
&= \prod_{i \geq 1} \bigl[uq^{i-1}\kappa^{k}; \kappa^N\bigr]
_{\floor{\frac{\lambda_{i}^\vee + n-1 -k}{N}}},
\\
\mathsf{N}_{\varnothing, \lambda}^{(\ell\vert N)} (u \vert q, \kappa)
&= \prod_{i \geq 1}\bigl[uq^{-i}\kappa^{\ell -N\floor{\frac{\lambda_i^\vee + \ell }{N}}} ; \kappa^N\bigr]
_{\floor{\frac{\lambda_i^\vee + \ell }{N}}}.
\end{align}
We compute
\begin{align}
\mathsf{N}_{\varnothing, \lambda}^{(\ell\vert N)} (\frac{q\kappa}{d} \vert q, \kappa)
&= \prod_{i \geq 1}\bigl[d^{-1}q^{1-i}\kappa^{1+ \ell -N\floor{\frac{\lambda_i^\vee + \ell }{N}}} ; \kappa^N\bigr]
_{\floor{\frac{\lambda_i^\vee + \ell }{N}}}.
\nonumber\\
&= \prod_{i \geq 1} (-1)^{\floor{\frac{\lambda_i^\vee + \ell }{N}}}
\bigl[d q^{i-1}\kappa^{-1- \ell +N\floor{\frac{\lambda_i^\vee + \ell }{N}}} ; \kappa^{-N}\bigr]
_{\floor{\frac{\lambda_i^\vee + \ell }{N}}}
\nonumber\\
&= \prod_{i \geq 1} (-1)^{\floor{\frac{\lambda_i^\vee + \ell }{N}}}
\bigl[d q^{i-1}\kappa^{N -1- \ell} ; \kappa^{N}\bigr]
_{\floor{\frac{\lambda_i^\vee + \ell }{N}}}
\nonumber\\
&= \prod_{i \geq 1} (-1)^{\floor{\frac{\lambda_i^\vee + \ell }{N}}}
 \cdot \mathsf{N}_{\lambda, \varnothing}^{(N-1-\ell\vert N)} (d \vert q, \kappa),\label{d-inversion}
\end{align}
where we have used \eqref{q-shifted-inversion}.

In general, if we take the specialization
\begin{equation}\label{mass-specialization}
a_i = \frac{q \kappa b_{i-1}}{d_{i-1}}, \quad 0\equiv N, \qquad c_j = \frac{b_j}{\overline{d}_{j}},
\end{equation}
the anti-fundamental factor is
\begin{align}
\mathsf{N}_{\varnothing, \lambda^{(j)}}^{(j-i\vert N)} \left(b_{i-1,j}\frac{q\kappa}{d_{i-1}} \vert q, \kappa\right)
 \sim \mathsf{N}_{\lambda^{(j)}, \varnothing}^{(N+(i-1)-j\vert N)} (b_{j,i-1}d_{i-1} \vert q, \kappa)
= \mathsf{N}_{\lambda^{(j)}, \varnothing}^{(k-j\vert N)} (b_{jk}d_{k} \vert q, \kappa),
\end{align}
where $b_{ij} := b_i/b_j$. On the other hand, the fundamental factor is
\begin{equation}
\mathsf{N}_{\lambda^{(i)}, \varnothing}^{(j-i\vert N)} (b_{ij}\overline{d}_{j} \vert q, \kappa).
\end{equation}
Hence, up to sign, they are the same under the exchange of mass parameters.

\subsubsection{Vector multiplet}

Let us parametrize the lengths of the columns of an $N$-tuple of Young diagrams as follows:
\begin{equation}\label{column-length}
\bigl(\lambda^{(i)}\bigr)^\vee= \bigl(\ell_1^{(i)}, \ell_2^{(i)}, \dots\bigr), \qquad 1 \leq i \leq N.
\end{equation}
For $i \leq j$, the (inverse of) vector multiplet contribution is
\begin{equation}\label{vector-general}
\Nk^{(j-i \vert N)}_{\lambda^{(i)}, \lambda^{(j)}} \bigl(b_i/b_j \vert q, t^{-\frac{1}{N}}\bigr) =
\prod_{k, m=1}^\infty
\frac{\bigl[(b_i/b_j) q^{m -k} t^{1 -\frac{j-i}{N} + \floor{\frac{\ell^{(j)}_{k} + j-i - \ell^{(i)}_{m})}{N}}} ; t\bigr]_\infty}
{\bigl[(b_i/b_j) q^{m-k-1} t^{1 -\frac{j-i}{N} + \floor{\frac{\ell^{(j)}_{k}+ j-i - \ell^{(i)}_m)}{N}}} ; t\bigr]_\infty},
\end{equation}
where we have deleted the normalization factor which is independent of \smash{$\ell_k^{(i)}$}.
When $i > j$, $j-i <0$ in equation~\eqref{vector-general} should be replaced by $N+j-i$. But, since
\begin{equation}
 t^{1 -\frac{N+j-i}{N} + \floor{\frac{\ell^{(j)}_{k} +(n + j-i) - \ell^{(i)}_{m})}{N}}}
 =
 t^{1 -\frac{j-i}{N} + \floor{\frac{\ell^{(j)}_{k} + j-i - \ell^{(i)}_{m})}{N}}},
\end{equation}
we may use equation~\eqref{vector-general} for any $1 \leq i,j \leq N$.

Let us assume that the width of $\lambda^{(i)}$ is at most $L_i$ and set $L=L_1 + L_2 + \cdots + L_N$.
We introduce a disjoint decomposition of the index set $\{ 1,2, \dots, L\}$ by $S_1, S_2, \dots, S_N$, where
\begin{equation}
S_i = \{ L_1 + \cdots + L_{i-1} +1, \dots, L_1 + \cdots + L_{i} \}, \qquad \vert S_i \vert = L_i.
\end{equation}
When $I = L_1 + \cdots + L_{i-1} + m \in S_i$, $1 \leq m \leq L_i$, we denote $I=(i,m)$.
Then we define the variables
\begin{equation}\label{Jackson-variables}
z_{(i,m)}= b_i^{-1} q^{1-m} \kappa^{i-1} t^{\floor{\frac{\ell^{(i)}_m +i-1}{N}}}, \qquad \kappa=t^{-1/N}.
\end{equation}
We can order $L$ variables $z_I$ in the lexicographic manner.
In \cite{AHKOSSY2}, in order to write down the weight function $W_{m+n}(z)$ in a symmetric way,
we defined (see below equation~(4.18)):
\begin{equation}\label{n=2}
z_{N_1 +j} = Q^{-1} q^{1-j} t^{\floor{\frac{k_j-1}{2}}}, \qquad k_j = (\lambda^{(2)})^\vee_j.
\end{equation}
When $N=2$, the definition \eqref{Jackson-variables} implies
\begin{equation}
z_{(1,i)} = b_1^{-1} q^{1-i} t^{\floor{\frac{\ell^{(1)}_i}{2}}},
\qquad
z_{(2,j)} = b_2^{-1} q^{1-j} \kappa t^{\floor{\frac{\ell^{(2)}_j+1}{2}}}
= b_2^{-1} q^{1-j} \kappa^{-1} t^{\floor{\frac{\ell^{(2)}_j-1}{2}}}.
\end{equation}
After the scaling by $b_1$, this agrees with \eqref{n=2} with $Q=\kappa(b_2/b_1)$.

Now we note\footnote{$[\ell]_N$ and $[m]_N$ denote the residue class modulo $N$. Not to be confused with the $q$-integer.}
\begin{align}\label{floor-difference}
\left\lfloor \frac{\ell - m}{N}\right\rfloor
= \left\lfloor\frac{\ell}{N}\right\rfloor - \left\lfloor \frac{m}{N}\right\rfloor - \theta([\ell]_N, [m]_N),
\end{align}
where for a pair $(r,s)$ of the residue classes modulo $N$ we define
\begin{align}
\theta(r,s) =
\begin{cases}
1 & 0 \leq r < s \leq N-1, \\
0 & \text{otherwise}.
\end{cases}
\end{align}
By making use of the formula \eqref{floor-difference},
we can recast the (double product part of)\footnote{There is also a single product part coming from the boundaries of
semi-infinite regions.} vector multiplet contribution as follows:
\begin{equation}
\Nk^{(j-i \vert N)}_{\lambda^{(i)}, \lambda^{(j)}} \bigl(b_i/b_j \vert q, t^{-\frac{1}{N}}\bigr) =
\prod_{I =(i,m) \in S_i} \prod_{J =(j,k)\in S_j}
\frac{\bigl[(tz_J/z_I)t^{- \theta(\ell^{(j)}_k + j-1, \ell^{(i)}_m + i-1)} ; t\bigr]_\infty}
{\bigl[(t z_J/q z_I)t^{- \theta(\ell^{(j)}_k + j-1, \ell^{(i)}_m + i-1)} ; t\bigr]_\infty},
\end{equation}
where for $i=j$, we should remove the factor with $m=k$. Hence
\begin{equation}
\prod_{i,j=1}^N \Nk^{(j-i \vert N)}_{\lambda^{(i)}, \lambda^{(j)}} \bigl(b_i/b_j \vert q, t^{-\frac{1}{N}}\bigr) =
\prod_{I \neq J =1}^L \frac{\bigl[(tz_J/z_I)t^{- \theta(\ell^{(j)}_k + j-1, \ell^{(i)}_m + i-1)} ; t\bigr]_\infty}
{\bigl[(t z_J/q z_I)t^{- \theta(\ell^{(j)}_k + j-1, \ell^{(i)}_m + i-1)} ; t\bigr]_\infty}.
\end{equation}
The additional factor $t^{- \theta(\ell^{(j)}_k + j-1, \ell^{(i)}_m + i-1)}$ is a generalization of
$t^{\{(\ell_I)-1\}\cdot(\ell_J)}$ appearing in equation~(4.19) of \cite{AHKOSSY2},
where $(\ell_I):=[\ell_I]_2 $ is nothing but the parity of $\ell_I$.

For each pair $(i,j)$, $1 \leq i,j \leq N$, the boundary contribution of the mass truncation comes from
\begin{equation}
{\hbox{(I)}} \ 1 \leq m \leq L_i, \ L_j+1 \leq k < \infty, \qquad
{\hbox{and}} \qquad
{\hbox{(II)}} \ L_i +1 \leq m < \infty, \ 1 \leq k \leq L_j,
\end{equation}
where we can assume either $\ell_k^{(j)}=0$ or $\ell_m^{(i)}=0$.
With the same integration variables as above, we compute the boundary contribution as follows:
\begin{description}\itemsep=0pt
\item{Case (I):}
\begin{align}
& \prod_{m=1}^{L_i} \bigl[ (b_i/b_j) q^{m-L_j-1} t^{1- \frac{j-i}{N} + \floor{\frac{j-i-\ell_m^{(i)}}{N}}};t\bigr]_\infty\nonumber\\
&\qquad{} = \prod_{m=1}^{L_i} \bigl[ b_j^{-1} z_{(i,m)}^{-1} q^{-L_j} \kappa^{j-1} t^{1 - \theta(j-1, \ell_m^{(i)} +i-1)}. t\bigr]_\infty.
\end{align}
\item{Case (II):}
\begin{align}
& \prod_{k=1}^{L_j} \frac{1}{\bigl[ (b_i/b_j) q^{L_i-k} t^{1- \frac{j-i}{N} + \floor{\frac{\ell_k^{(j)} + j-i}{N}}};t\bigr]_\infty}\nonumber\\
&\qquad = \prod_{k=1}^{L_j} \frac{1}{ \bigl[ b_i z_{(j,k)}q^{L_i-1} \kappa^{1-i} t^{1 - \theta(\ell_k^{(j)} +j -1, i-1)}. t\bigr]_\infty}.
\end{align}
\end{description}
Taking the product of the inverses of all the boundary contributions, we obtain
\begin{equation}
\prod_{i=1}^N \prod_{m=1}^{L_i} \prod_{j=1}^N \frac{\bigl[z_{(i,m)} b_j q^{L_j-1} \kappa^{1-j} t^{{1- \theta(\ell_m^{(i)} + i-1, j-1)}};t\bigr]_\infty}
{\bigl[z_{(i,m)}^{-1} b_j^{-1} q^{-L_j} \kappa^{j-1} t^{{1- \theta(j-1, \ell_m^{(i)} + i-1)}};t\bigr]_\infty}.
\end{equation}

\subsubsection{Matter multiplet}

Substituting the specialization \eqref{mass-specialization},
we obtain the matter multiplet contributions as follows:
\begin{enumerate}\itemsep=0pt
\item[(1)] fundamental matter
\begin{align}
\Nk^{(j-i \vert N)}_{\lambda^{(i)}, \varnothing}
&=\prod_{m=1}^\infty \bigl[ b_{ij} \overline{d}_{j} q^{m-1} t^{1 - \frac{j-i}{N} + \floor{\frac{(j-1) - (\ell^{(i)}_m + i-1)}{N}}} ; t\bigr]_\infty \nonumber\\
&= \prod_{m=1}^\infty \bigl[b_j^{-1} \overline{d}_{j} z_{(i,m)}^{-1}
\kappa^{j-1} t^{1 - \theta(j-1, \ell_m^{(i)} + i -1)}; t\bigr]_\infty,
\end{align}
\item[(2)] anti-fundamental matter
\begin{align}\label{anti-fund-gln}
\Nk^{(j-i \vert N)}_{\varnothing, \lambda^{(j)}}
&= \prod_{k=1}^\infty \frac{1}{\bigl[b_{i-1, j} d_{i-1}^{-1} q^{1-k} t^{1 - \frac{j-i+1}{N} + \floor{\frac{\ell^{(j)}_k}{N}+j -i}}; t\bigr]_\infty} \nonumber\\
&= \prod_{k=1}^\infty \frac{1}{\bigl[b_{i-1} d_{i-1}^{-1} z_{(j,k)} \kappa^{2-i} t^{1 - \theta(\ell_k^{(j)} + j-1, i-1)} \bigr]_\infty},
\end{align}
\end{enumerate}
where the normalization factor is omitted in the same manner as \eqref{vector-general}.
\subsection{Mass parameter truncation and a basis of the cocycle}
\label{sec:Jackson-cocycle}

Let us impose the mass truncation condition $\overline{d}_k = q^{-m_k},~m_k \in \mathbb{Z}_{\geq 0}$.
Then by identifying $\ell_k$ in the previous section with $m_k$, we see that the fundamental matter contribution cancels half of
the boundary part of the vector multiplet contribution.
After the mass parameter truncation, the number of the columns of $\lambda^{(i)}$ is at most $m_i$, hence
we have $M := m_1 + \cdots + m_N$ columns in total.
It is convenient to introduce the index $I$, $1 \leq I \leq M$,
which is identified with the multi-index $(i,m)$: the $m$-th column of
the $i$-th Young diagram through a bijection
\begin{equation}\label{bijection}
\{1,2,\dots, M\} \to\{(1,1),\dots,(1,m_1), (2,1),\dots,(2,m_2), \dots, (N,1), \dots, (N,m_N)\}.
\end{equation}
Using this identification, we define the shifted residue $\sres{\ell_I}$ as the residue class of $\ell_m^{(i)} + i-1$ modulo $N$.
The shifted residue $\sres{\ell_I}$ tells the color of the end box in the corresponding column, see Figure~\ref{Fig:shifted residue}.
Then the affine Laumon partition function is obtained from the following weight function:\looseness=-1
\begin{align}
& W_M^{\glN}(z) = \prod_{I=1}^M \prod_{k=1}^N
\frac{\bigl[z_I b_k q^{m_k-1} \kappa^{1-k} t^{{1- \theta(\sres{\ell_I}, k-1)}} ;t\bigr]_\infty}
{\bigl[z_I b_{k-1} d_{k-1}^{-1} \kappa^{2-k} t^{1 - \theta(\sres{\ell_I}, k-1)} ;t\bigr]_\infty}\nonumber\\
&\hphantom{W_M^{\glN}(z) =}{}
\times \prod_{I \neq J =1}^M \frac
{\bigl[(t z_J/q z_I)t^{- \theta(\sres{\ell_J}, \sres{\ell_I})} ; t\bigr]_\infty}
{\bigl[(tz_J/z_I)t^{- \theta(\sres{\ell_J}, \sres{\ell_I})} ; t\bigr]_\infty},
\end{align}
with the cycle of the Jackson integral that is chosen according to the corresponding lattice~truncation.
Note that the weight function \smash{$W_M^{\glN}(z)$} depends on the fixed point data $\bigl\{ \lambda^{(i)} \bigr\}$
only through the shifted residues~$\sres{\ell_I}$.

\begin{figure}[!ht]
\ytableausetup{boxsize=1.5em}
$$
\bigl(\lambda^{(1)},\lambda^{(2)},\lambda^{(3)}\bigr)= \bigl(~\ytableaushort{{1}{1}{1}{1}{1},{2}{2}{2},{0}{0},{1},{2}}~,
~\ytableaushort{{2}{2}{2}{2},{0}{0},{1}{1},{2}}~,~\ytableaushort{{0}{0},{1}{1},{2}}~\bigr)
$$
\caption{There are $5+4+2=11$ columns (see the bijection \eqref{bijection}).
Corresponding shifted residues $\sres{\ell_I}_{I=1}^{11}$ are
given by the number in the end box of each column, namely $(2,0,2,1,1;2,1,2,2;2,1)$. }
\label{Fig:shifted residue}
\end{figure}

When $N=2$, $\theta(X,Y)$ is non-vanishing only when $X$ is even and $Y$ is odd.
Namely, we have $\theta(X,Y)= (1-(X))\cdot(Y)$. Hence,
\begin{align}
& W_M^{\widehat{\mathfrak{gl}}_2}(z)
= \prod_{I=1}^M \frac{\bigl[z_I q^{m-1} t ;t\bigr]_\infty}
{\bigl[z_I Q d_{0}^{-1} t ;t\bigr]_\infty}
\frac{\bigl[z_I Q q^{n-1} t^{1+\sres{\ell_I}} ;t\bigr]_\infty}
{\bigl[z_I d_{1}^{-1}t^{\sres{\ell_I}} ;t\bigr]_\infty}\nonumber\\
&\hphantom{W_M^{\widehat{\mathfrak{gl}}_2}(z)=}{}
\times
\prod_{I \neq J =1}^M \frac
{\bigl[(t z_J/q z_I)t^{(\sres{\ell_J}-1)\cdot \sres{\ell_I})} ; t\bigr]_\infty}
{\bigl[(tz_J/z_I)t^{(\sres{\ell_J}-1)\cdot \sres{\ell_I})} ; t\bigr]_\infty},
\end{align}
where we have substituted $m_1=m$, $m_2=n$ and $b_1=1$, $b_0=b_2 = Q/\kappa$.
We see that this is exactly equation~(4.19) in~\cite{AHKOSSY2} with~$d_0 \leftrightarrow d_4$.

For a matching with the formulas in \cite{MIto}, we define
\begin{equation}\label{Jackson-parameters}
\mathsf{a}_k := b_k^{-1} q^{1-m_k} \kappa^{k-1}, \qquad \mathsf{b}_k := b_{k} d_{k}^{-1} \kappa^{1-k},
\end{equation}
with $\mathsf{b}_N = t \mathsf{b}_0$. Note that $\mathsf{a}_k \mathsf{b}_k =q^{1-m_k}d_k^{-1}$.\footnote{Compare this with
 \cite[Theorem 4.1]{AHKOSSY2}. The original $q$-KZ equation is reproduced by considering the simultaneous
$t$-shift $\mathsf{a}_k \to t \mathsf{a}_k$, $\mathsf{b}_k \to t^{-1}\mathsf{b}_k$, which is equivalent to the $t$-shift of the Coulomb moduli
$b_k \to t^{-1} b_k$.}
We obtain
\begin{equation}
W_M^{\glN}(z) = \prod_{I=1}^M \prod_{k=1}^N
\frac{\bigl[\mathsf{a}_k^{-1} z_I t^{1- \theta(\sres{\ell_I}, k-1)} ;t\bigr]_\infty}
{\bigl[ \mathsf{b}_{k-1} z_I t^{1- \theta(\sres{\ell_I}, k-1)} ;t\bigr]_\infty}
 \cdot \prod_{I \neq J =1}^M \frac{[(t z_J/q z_I)t^{- \theta(\sres{\ell_J}, \sres{\ell_I})} ; t]_\infty}
{[(tz_J/z_I)t^{- \theta(\sres{\ell_J}, \sres{\ell_I})} ; t]_\infty},
\end{equation}
which should be compared with the weight function $\Phi_{M,N}(z)$ (see~\eqref{weight-MN}
and also equation~(4.27) of \cite{AHKOSSY2} in the case of $N=2$).

We may note that the sum in the Jackson integral~\eqref{Jackson-Selberg} is bilateral, namely it is taken over
$\{\nu_i \}\in \bbZ^M$. On the other hand, the affine Laumon partition function involves a sum over Young diagrams.
After the mass parameter truncation, the discrepancy is remedied by an appropriate choice of the cycle $\xi=(\xi_1, \dots, \xi_M)$.
In fact by taking the following cycle:\footnote{Recall that we have exchanged $q$ and $t$, compared with the original literature.}
\begin{equation}\label{cycle-choice}
\xi=\xi_{m_i}
=\bigl(\underbrace{\mathsf{a}_1, \mathsf{a}_1 q, \dots, \mathsf{a}_1 q^{m_1-1}}_{m_1}, \dots\dots,
\underbrace{\mathsf{a}_N, \mathsf{a}_N q, \dots, \mathsf{a}_N q^{m_N-1}}_{m_N}\bigr),
\end{equation}
the lattice summation over $\bigl\{z_i=\xi_i t^{\nu_i}\bigr\}$ is truncated to a cone \cite[Section 3.2]{ItoNoumi}:
\begin{align}
&0\leq \nu_1\leq \nu_2 \leq \dots \leq \nu_{n-1} \leq \nu_{m_1}, \nonumber\\
&0\leq \nu_{m_1+1} \leq \nu_{m_1+2} \leq \dots \leq \nu_{n-1} \leq \nu_{m_1+m_2}, \nonumber\\
& \vdots \qquad \vdots \qquad \vdots \nonumber\\
&0 \leq \nu_{M- m_N+1} \leq \nu_{M -m_N +2}\leq \dots \leq \nu_{M-1} \leq \nu_{M}.
\end{align}

In order to identify the cocycle factor $\phi(z)$ in the Jackson integral \eqref{Jackson-Selberg},
let us introduce a~disjoint decomposition of $\{ 1,2, \dots, M \} = R_0 \sqcup \cdots \sqcup R_{N-1}$
by $R_k := \{ I \mid \sres{\ell_I} =k \}$.
Then we can decompose the weight function as
\begin{equation}
W_M^{\glN}(z) = W_L^{\glN, (0)}(z) \cdot P_{\{R_0 \sqcup \cdots \sqcup R_{N-1}\}}(z),
\end{equation}
where
\begin{equation}\label{weight-function}
W_M^{\glN, (0)}(z) = \prod_{I=1}^M \prod_{k=1}^N
\frac{\bigl[t \mathsf{a}_k^{-1} z_I ;t\bigr]_\infty}
{\bigl[ \mathsf{b}_{k-1} z_I ;t\bigr]_\infty} \cdot
\prod_{I \neq J =1}^M \frac{[(t z_J/q z_I) ; t]_\infty}{[(tz_J/z_I); t]_\infty},
\end{equation}
and
\begin{align}
P_{\{R_0 \sqcup \cdots \sqcup R_{N-1}\}}(z)
\sim{}& \left(\prod_{k=2}^N \prod_{\ell < k -1} \prod_{J \in R_\ell} (1- \mathsf{a}_k^{-1} z_J)\right)
\left(\prod_{k=1}^{N-1} \prod_{k \leq \ell } \prod_{J \in R_\ell} (1- \mathsf{b}_{k} z_J) \right) \nonumber\\
& {} \times \left(\prod_{0 \leq k < \ell \leq N-1} \prod_{I \in R_k} \prod_{J \in R_\ell} \frac{z_J - q^{-1}z_I}{z_J - z_I}\right).
\end{align}
By \cite[Lemma 3.1]{AHKOSSY2}, the second factor of $W_M^{\glN, (0)}(z)$ is
\begin{equation}
\prod_{I \neq J =1}^M \frac{[(t z_J/q z_I) ; t]_\infty}{[(tz_J/z_I); t]_\infty}
= C(z) \Delta(z) \prod_{I=1}^M z_I^{-\tau(M-1)} \prod_{1 \leq I<J \leq M} z_I^{2\tau -1} \frac{[(t z_J/q z_I) ; t]_\infty}{[(qz_J/z_I); t]_\infty},
\end{equation}
where $\tau = \log_t q$ and $C(z)$ is a pseudo-constant that is invariant under $z_I \to tz_I$ for each variable~$z_I$.

The last factor of $P_{\{R_0 \cup \cdots \cup R_{N-1}\}}(z) $ is exactly such that we can apply \cite[Proposition~G.1]{AHKOSSY2}.
To apply the proposition, we recast the remaining factors as follows:
\begin{equation}\label{cocycle-base}
\prod_{\ell=0}^{N-1} \prod_{J \in R_{\ell}} f_\ell(z_J), \qquad
f_{\ell}(z) := \prod_{\ell+1 < k} \bigl(1- \mathsf{a}_k^{-1} z\bigr) \cdot \prod_{k \leq \ell} (1- \mathsf{b}_{k} z).
\end{equation}
For example, when $N=3$,
\begin{align*}
&f_0(z)= \bigl(1- \mathsf{a}_2^{-1}z\bigr)\bigl(1- \mathsf{a}_3^{-1}z\bigr), \qquad
f_1(z)= \bigl(1- \mathsf{a}_3^{-1}z\bigr)(1- \mathsf{b}_1 z), \\
& f_2(z)= (1- \mathsf{b}_1 z)(1- \mathsf{b}_2 z).
\end{align*}

Now let consider all the partitions $R_0 \sqcup \cdots \sqcup R_{N-1}$ of $\{1, \dots, M\}$
with fixed $|R_k| = r_k $. Then \cite[Proposition G.1]{AHKOSSY2} tells
\begin{align}
&\phi_{(r_0,r_1, \dots, r_{N-1})}(z)
:=\sum_{R_0 \sqcup \cdots \sqcup R_{N-1}} P_{\{R_0 \sqcup \cdots \sqcup R_{N-1}\}}(z) \nonumber\\
&\qquad = \prod_{k=0}^{N-1} \frac{1}{[r_k]_{q^{-1}}!} \cdot \Delta(z)^{-1} \nonumber\\
& \qquad\quad \times \mathcal{A} \left( \prod_{i_0=1}^{r_0} f_0(z_{i_0}) \prod_{i_1=1}^{r_1} f_1(z_{r_0 + i_1}) \cdots
\prod_{i_{N-1}=1}^{r_{N-1}} f_{N-1}(z_{r_0 + \cdots + r_{N-2}+ i_{N-1}})\Delta(q,z) \right).\label{symmetrization}
\end{align}

In summary, we have the following.\footnote{Precisely speaking, the Pochhammer type factors are used in the Jackson integral,
while we employ the sinh type factors in the definition of the affine Laumon function. They agree up to some monomial factor
(see Appendix~\ref{App-C}).}
\begin{Proposition}\label{Laumon= Jackson}
By the identification of parameters \eqref{Jackson-parameters},
the affine Laumon partition function \eqref{glnAL} with the mass parameter truncation
coincides with the Jackson integral \eqref{Jackson-Selberg} with the choice of the integration cycle \eqref{cycle-choice},
where the cocycle functions in \eqref{Jackson-Selberg} are given by \eqref{symmetrization}.
\end{Proposition}

The functions $\phi_{(r_0,r_1, \dots, r_{N-1})}(z)$ are supposed to give a basis of the cocycle factors.
In fact, in~\cite{ItoNoumi} the set $Z_{N,M}= \{ \mu = (\mu_1, \mu_2, \dots, \mu_N) \in \mathbb{Z}^N_{\geq 0} \vert
\mu_1 + \mu_2 + \cdots + \mu_N =M \}$ is introduced\footnote{We changed $s \to N$ and $n \to M$ from~\cite{ItoNoumi}.}
and the functions $E_\lambda(z)$ labeled by $\lambda \in Z_{N,M}$ are considered.
We expect the functions $\phi_{(r_0,r_1, \dots, r_{N-1})}(z)$
gives a basis of solutions to $t$-difference equation of rank
\[
r:= \vert Z_{N,M} \vert = \binom{N+M-1}{M}= \binom{N+M-1}{N-1}.
\]
Note that the rank agrees with the size of $R_M$ block of the $R$-matrix computed in Section \ref{sec:mass-truncation}.
We can check the functions $\phi_{(r_0,r_1, \dots, r_{N-1})}(z)$ coincide with the basis of the cocycle functions
for \smash{$U_q\bigl(\widehat{\mathfrak{sl}}_2\bigr)$} $q$-KZ equation by Matsuo and Varchenko \cite{AMatsuo2,VarchenkoCMP}.

\begin{Example}
When $N=3$, up to the normalization factor, we have
\begin{align}
\phi_{(r_0,r_1, r_2)}(z)
&= \Delta(z)^{-1}
\mathcal{A} \left( \prod_{i=1}^{r_0} (1- \mathsf{a}_2^{-1}z_i)
 \prod_{i=1}^{r_0+r_1} (1- \mathsf{a}_3^{-1}z_i) \right. \nonumber\\
 & \left. \times \prod_{i=r_0+ 1}^{M} (1- \mathsf{b}_1 z_i) \prod_{i=r_0+r_1 + 1}^{M}(1- \mathsf{b}_2 z_i) \Delta(q,z) \right).\label{N=3;cocycle}
\end{align}
The cocycle function $\phi_{(r_0,r_1, r_2)}(z)$ corresponds to $z_1^{r_1} z_2^{r_2} = x_1^{r_1+r_2} x_2^{r_2}$ term
(modulo the power of $\Lambda=x_1x_2x_3$) in the expansion of the partition function (see Table~\ref{M=3}
in Appendix~\ref{App:Combi}).
\end{Example}

In \cite{AHKOSSY2}, we rely on the Jackson integral representation of the affine Laumon partition function
to prove that it is a solution to the non-stationary difference equation~\eqref{Shakirov-eq}.
As we have shown in this section the Jackson integral representation of the affine Laumon partition function is also valid for~\smash{$\glN$}.
Unfortunately, it only tells us the relation to the $q$-KZ equation of \smash{$U_q\bigl(\widehat{\mathfrak{sl}}_2\bigr)$}.
Namely, we can see the \smash{$\glN$} affine Laumon partition function corresponds to $N+2$ point correlation functions for
the \smash{$U_q\bigl(\widehat{\mathfrak{sl}}_2\bigr)$} $q$-KZ equation, where the shift operator acts on the mass parameters~$d_i$,
not on the expansion parameters~$x_i$.
In Section~\ref{sec:mass-truncation}, we have seen that after imposing the mass truncation condition
the non-stationary difference equation for \smash{$\glN$} is related to the $R$-matrix of \smash{$U_q\bigl(A_{N-1}^{(1)}\bigr)$}.
It is natural to expect a duality of $q$-KZ equations between \smash{$U_q\bigl(\widehat{\mathfrak{sl}}_2\bigr)$} and
\smash{$U_q\bigl(\widehat{\mathfrak{sl}}_N\bigr)$}. In fact, when $N=2$, we have a dual pair of the \smash{$U_q\bigl(\widehat{\mathfrak{sl}}_2\bigr)$}
KZ equations \cite{AHKOSSY2,MIto}. Such a duality may be crucial for a proof of our conjecture.

\section{Four-dimensional limit and Fuji--Suzuki--Tsuda system}
\label{sec:4d-limit}

The following computation uses almost the same method as that in~\cite[Appendix C]{AHKOSSY1}.

We start with the equation satisfied by the five-dimensional \smash{$\widehat{\mathfrak{sl}}_N$}
affine Laumon function $\psi_{5d}$ written in the normal ordered form
\begin{align}\label{eq:6.1}
&(A_1 -A_2) \cdot \psi_{5d}=0, \\
&A_1= {:} \prod_{a=1}^N \frac{\bigl(q^{\vartheta'_a+{\tilde m}_a} x_a;q\bigr)_\infty}{(x_a;q)_\infty} q^{\gamma_a \vartheta_a}{:} ,\qquad
A_2= {:} \prod_{a=1}^N \frac{\bigl(q^{{\tilde m}_a-m_a} x_a;q\bigr)_\infty}{\bigl(q^{-\vartheta'_a-m_a} x_a;q\bigr)_\infty} q^{-\vartheta_a \vartheta'_a}{:} .
\end{align}
To take the four-dimensional limit, we have set the parameters as
\begin{equation}
d_a=q^{m_a}, \quad
\overline{d}_{a} =q^{\tilde{m}_a}, \quad
\kappa \frac{b_{a+1}}{b_a} =q^{\gamma_a}, \quad
\vartheta_a=x_a \frac{\partial}{\partial x_a}, \quad
\vartheta'_a=\vartheta_a-\vartheta_{a-1},
\end{equation}
where $a \in \bbZ/ N \bbZ$.

For $q=e^h$, $h \to 0$, we have
\begin{equation}\label{4dexpansion}
(q^a x;q)_b=(1-x)^{b} \left\{1+\frac{h x}{1-x}\left(a b +{ \binom{b}{2} }\right)+O\bigl(h^2\bigr)\right\}.
\end{equation}
Namely, by taking the limit of the $q$-binomial theorem (see \cite[equation~(5.2)]{AHKOSSY2})
\begin{align}
&\frac{(q^\alpha x;q)_\infty}{(x;q)_\infty}= \sum_{k=1}^\infty \frac{(q^\alpha;q)_k}{(q;q)_k} x^k
=(1-x)^{-\alpha}\left\{1+\frac{h}{2}\alpha(\alpha-1)\frac{x}{1-x}+O\bigl(h^2\bigr)\right\}, \\
&\frac{(-q^\alpha x;q)_\infty}{(-q^\beta x;q)_\infty}
=(1+x)^{\beta-\alpha}\left\{1+\frac{h}{2}(\alpha-\beta)(\alpha+\beta-1)\frac{x}{1-x}+O\bigl(h^2\bigr)\right\}.
\end{align}

By using \eqref{4dexpansion}, the operators $A_1$, $A_2$ are expanded as
\begin{align}
&A_1={:} \prod_{a=1}^N (1-x_a)^{-\vartheta'_a-{\tilde m}_a} \left\{1+h \sum _{a=1}^N K_{1,a}+O\bigl(h^2\bigr)\right\}{:} ,\\
&A_2={:} \prod_{a=1}^N (1-x_a)^{-\vartheta'_a-{\tilde m}_a} \left\{1+h \sum_{a=1}^N K_{2,a} +O\bigl(h^2\bigr)\right\}{:} ,
\end{align}
where
\begin{align}
&K_{1,a}=\gamma_a \vartheta_a+\frac{x_a}{1-x_a} \binom{\vartheta' _a+{\tilde m}_a}{2}, \\
& K_{2,a}=
-\vartheta_a\vartheta'_a-\frac{x_a}{1-x_a} \left(({\tilde m}_a-m_a) (-\vartheta' _a-{\tilde m}_a)+\binom{-\vartheta'_a-{\tilde m}_a}{2} \right).
\end{align}
Then, since $ \binom{a}{2}+\binom{-a}{2}=a^2 $, we have
\begin{align}\label{eq:K12-lim}
&K_1-K_2=\vartheta_a(\vartheta'_a+\gamma_a)+\frac{x_a}{1-x_a}(\vartheta'_a+m_a)(\vartheta'_a+\tilde{m}_a).
\end{align}

Define an operator $R$ acting only on $x$-variables (not on $\vartheta$'s) as
\begin{align}
& R\colon \ x_a \mapsto x_a \frac{U_{a+1}}{U_{a}}, \qquad a \in \bbZ/(N \bbZ),
\\
&U_a =\sum_{i=0}^{N-1}\prod_{j=0}^{i-1}x_{a+j}=\underbrace{x_a+x_ax_{a+1}+\cdots+\Lambda}_{N}, \\
&\Lambda=x_a x_{a+1}\cdots x_n x_1 \cdots x_{a-1}=x_1x_2\cdots x_N.
\end{align}
The following relation is useful.
\begin{align}
U_a-x_a U_{a+1}&=(x_a+x_ax_{a+1}+\cdots+\Lambda)
-x_a(x_{a+1}+x_{a+1}x_{a+2}+\cdots+\Lambda) \nonumber \\
&=x_a(1-\Lambda).
\end{align}

\begin{Lemma}\label{Lemma:CR}
For any function $F(x,\vartheta)$, we have the following operator identity:
\begin{equation}
C^{-1} R\colon \ \prod_{a=1}^N (1-x_a)^{-\vartheta'_a} F(x,\vartheta): = {:} F\bigl(C^{-1} R(x),\vartheta\bigr){:},
\end{equation}
where $C$ acts only on $x$-variables (not on $\vartheta$'s) as $x_a \mapsto x_{a+1}$, $ a \in \bbZ/(N \bbZ)$.
\end{Lemma}

\begin{proof}
The action of the left-hand side (LHS) on a monomial $x^{\nu}=\prod_{a=1}^N {x_{a}}^{\nu_a}$ is computed as
\begin{align}
\hbox{(LHS)} x^\nu&=C^{-1} R \left(\prod_{a=1}^N (1-x_a)^{-\nu_a+\nu_{a-1}} F(x,\nu)x^\nu \right) \nonumber \\
&=C^{-1} R \left(\prod_{a=1}^N \left(\frac{1-x_{a+1}}{1-x_a}\right)^{\nu_a} F(x,\nu)x^\nu \right) \nonumber\\
& =C^{-1} \prod_{a=1}^N R \left(\frac{1-x_{a+1}}{1-x_a}x_a \right)^{\nu_a} F(R(x),\nu) \nonumber \\
&=C^{-1} \prod_{a=1}^N \left(\frac{U_{a+1}-U_{a+2}x_{a+1}}{U_a-U_{a+1}x_a}x_a \right)^{\nu_a} F(R(x),\nu) \nonumber \\
&=C^{-1}\prod_{a=1}^N (x_{a+1})^{\nu_a} F(R(x),\nu)
 =F\bigl(\bigl(C^{-1} R(x)\bigr),\nu\bigr) \prod_{a=1}^N {x_{a}}^{\nu_a}.
\end{align}
The last expression is nothing but the desired one $\hbox{(RHS)} x^{\nu}$.
\end{proof}

\begin{Theorem}
If Conjecture~{\rm \ref{con:1.6}} is true,
the four-dimensional limit $\psi_{4d}=\psi_{5d}\big{|}_{h \to 0}$ of the partition function satisfies the equation
\begin{equation}\label{FSTeq}
\sum_{a=1}^N \left\{ \vartheta_a(\vartheta'_a+\gamma_a)+\frac{U_{a}}{1-\Lambda} (m_a+\vartheta'_a)(\tilde{m}_a+\vartheta'_a)\right\}\psi_{4d}=0.
\end{equation}
Namely, \eqref{eq:6.1} implies \eqref{FSTeq}.
\end{Theorem}

\begin{proof}
By using \eqref{eq:K12-lim} and the Lemma \ref{Lemma:CR}, apply
$C^{-1}R$ from the left to the $h$-expansion of the equation$(A_1-A_2) \cdot \psi_{5d}=0$.
Then the equation \eqref{FSTeq} is obtained as $O(h)$ term.
\end{proof}

\begin{Remark} In \cite{Yamada:2010rr}, 
in the context of four-dimensional gauge theories, the equation \eqref{FSTeq}
has been obtained as a quantization of the differential Fuji--Suzuki--Tsuda (FST) system.

The FST system (called $P_{\rm VI}$-chain in \cite{Tsu1})
was first considered by Tsuda as a similarity reduction of his UC-hierarchy which is a certain generalization of the KP hierarchy.
Independently, in the context of the Drinfeld--Sokolov hierarchy, it was obtained by Fuji--Suzuki~\cite{FS} $(N=3)$
and Suzuki \cite{Suz} $(N\geq 3)$.
The FST system relevant here is an isomonodromic deformation of the $N \times N$
Fuchsian equation on ${\mathbb P}^1$ with four regular singularities at $z=0,1,t,\infty$
with the following spectral type
\[
\bigl(\bigl(1^N\bigr), \underbrace{(N-1,1),\dots,(N-1,1)}_{k},\bigl(1^N\bigr)\bigr), \qquad k=2.
\]
Its multi-time extensions $(k\geq 3)$ has also been studied in \cite{Tsu2}.
The case $N=2$, $k=2$ is~$P_{\rm VI}$ and the cases $N=2, k\geq 3$ are the Garnier system.
The $q$-difference version of the FST system is also known.
Interestingly, in $q$-case, there exists a duality between the system of type $(N,k)=(m,2)$ and $(N, k)=(2,m)$ (see, e.g.,~\cite{Nagao-Yamada}).
\end{Remark}

Recently, some constructions of the $q$-FST system through the cluster algebra were obtained
 (see, e.g., \cite{Okubo-Suzuki,Suzuki-Okubo}).
Since such formulation will also give a natural way of the quantization (see, e.g., \cite{Marshakov:2019vnz}),
the relation to our construction from the affine Laumon space is an interesting problem.
The relation to the tetrahedron $R$-matrix uncovered in Section~\ref{sec:mass-truncation} also
gives a link to the cluster algebra. In \cite{Gavrylenko:2020eov}, it was argued that
the conjugation by the tetrahedron $R$-matrix can be described by a sequence of mutations
(see also~\cite{Inoue:2024swb} for the quantum case).

In \cite{Gavrylenko:2018ckn}, it was shown that the $W_N$ conformal block with a fully degenerate field
gives a solution to the FST system. This opens a problem of a higher-rank generalization of
the relation of the Virasolo conformal block with a degenerate field and the conformal block
of the ${\rm SL}(2)$ current algebra which is described by the KZ equation.
We find it intriguing in what sense the affine Laumon partition function gives a conformal block of
the deformed~$W$ algebra.

\appendix

\section[Symmetric form of widehat\{gl\}\_2 Hamiltonian]{Symmetric form of $\boldsymbol{\widehat{\mathfrak{gl}}_2}$ Hamiltonian}
\label{App-A}
The $\widehat{\mathfrak{gl}}_2$-non-stationary difference equation considered in \cite{AHKOSSY1,Shakirov:2021krl} is
\begin{equation}\label{Shakirov;Appendix}
\SS T_{qtQ,x}^{-1}T_{t,\Lambda}^{-1} \cdot \Psi(\Lambda,x)=\Psi(\Lambda,x).
\end{equation}
The Hamiltonian is
\begin{align}
&\SS = \frac{1}{\varphi(qx)\varphi(\Lambda/x)} \cdot \Bor \cdot
\frac{\varphi(\Lambda)\varphi\bigl(q^{-1} d_1d_2d_3d_4\Lambda\bigr)}{\varphi(-d_1x)\varphi(-d_2x)\varphi(-d_3 \Lambda/x)\varphi(-d_4 \Lambda/x)}
\nonumber \\
&\hphantom{\SS =}{}\cdot \Bor\cdot \frac{1}{\varphi\bigl(q^{-1}d_1d_2x\bigr)\varphi(d_3d_4 \Lambda/x)},
\end{align}
where $\varphi(z) := (z;q)_\infty$ and $\Bor$ is the $q$-Borel transformation on a formal Laurent series in~$x$:
\begin{equation}\label{q-Borel}
\Bor \left( \sum_{n} c_n x^n \right) = \sum_{n} q^{\frac{1}{2}n(n+1)} c_n x^n.
\end{equation}
In order to generalize the non-stationary difference equation \eqref{Shakirov;Appendix} to higher rank, it is instructive
to recast $\SS$ into more ``symmetric'' form.
In terms of homogeneous coordinates $x_1:=x$ and $x_2:= q^{-1} \Lambda /x$,
the total Hamiltonian becomes
\begin{align}
&\SS T_{qtQ,x}^{-1}T_{t,\Lambda}^{-1} = \frac{1}{\varphi(qx_1)\varphi(qx_2)} \cdot \Bor \cdot
\frac{\varphi(q x_1 x_2)\varphi(d_1d_2d_3d_4 x_1 x_2)}{\varphi(-d_1x_1)\varphi(-d_2x_1)\varphi(-q d_3 x_2)\varphi(-q d_4 x_2)}
\nonumber \\
&\hphantom{\SS T_{qtQ,x}^{-1}T_{t,\Lambda}^{-1} =}{}\cdot \Bor\cdot \frac{1}{\varphi(q^{-1}d_1d_2x_1)\varphi(qd_3d_4 x_2)} T_{qtQ,x_1}^{-1}T_{qQ,x_2}.
\end{align}
Since $x_1^{n} x_2^{m} \sim x^{n-m}$ the $q$-Borel transformation in coordinates $(x_1,x_2)$ is
\begin{equation}\label{relative}
 \Bor = q^{\frac{1}{2} (\vartheta_1 - \vartheta_2)(\vartheta_1 - \vartheta_2 +1)} = q^{\frac{1}{2} (\vartheta_1 - \vartheta_2)^2}
 p_1^{\frac{1}{2}} p_2^{-\frac{1}{2}},
\end{equation}
where $\vartheta_i := x_i \frac{\partial}{\partial x_i}$ and $p_i = q^{\vartheta_i}= T_{q, x_i}$.
Hence, we have
\begin{align}
&\SS T_{qtQ,x}^{-1}T_{t,\Lambda}^{-1} = \frac{1}{\varphi(qx_1)\varphi(qx_2)} \cdot q^{\frac{1}{2} (\vartheta_1 - \vartheta_2)^2} \nonumber\\
&\hphantom{\SS T_{qtQ,x}^{-1}T_{t,\Lambda}^{-1} =}{} \cdot
\frac{\varphi(q x_1 x_2)\varphi(d_1d_2d_3d_4 x_1 x_2)}{\varphi\bigl(-q^{\frac{1}{2}} d_1x_1\bigr)\varphi\bigl(-q^{\frac{1}{2}} d_2x_1\bigr)
\varphi(-q^{\frac{1}{2}} d_3 x_2)\varphi\bigl(-q^{\frac{1}{2}} d_4 x_2\bigr)}
\nonumber \\
&\hphantom{\SS T_{qtQ,x}^{-1}T_{t,\Lambda}^{-1} =}{}
\cdot q^{\frac{1}{2} (\vartheta_1 - \vartheta_2)^2} \cdot \frac{1}{\varphi(d_1d_2x_1)\varphi(d_3d_4 x_2)} T_{tQ,x_1}^{-1}T_{Q,x_2}.
\end{align}
The shift operator \smash{$ p_1^{\frac{1}{2}} p_2^{-\frac{1}{2}}$} in \eqref{relative} is combined with the original shift operator \smash{$T_{qtQ,x_1}^{-1}T_{qQ,x_2}$}.
Note that the factors in the numerator commute with \smash{$q^{\frac{1}{2} (\vartheta_1 - \vartheta_2)^2}$}.

Parametrizing \smash{$Q= \kappa \frac{b_2}{b_1}$} with \smash{$\kappa:= t^{-\frac{1}{2}}$}, we finally obtain
\begin{align}
&\SS T_{qtQ,x}^{-1}T_{t,\Lambda}^{-1} = \frac{1}{\varphi(qx_1)\varphi(qx_2)} \cdot q^{\frac{1}{2} (\vartheta_1 - \vartheta_2)^2} \nonumber\\
&\hphantom{\SS T_{qtQ,x}^{-1}T_{t,\Lambda}^{-1} =}{}\cdot
\frac{\varphi(q x_1 x_2)\varphi(d_1d_2d_3d_4 x_1 x_2)}{\varphi(-q^{\frac{1}{2}} d_1x_1)\varphi\bigl(-q^{\frac{1}{2}} d_2x_1\bigr)
\varphi\bigl(-q^{\frac{1}{2}} d_3 x_2\bigr)\varphi\bigl(-q^{\frac{1}{2}} d_4 x_2\bigr)}
\nonumber \\
&\hphantom{\SS T_{qtQ,x}^{-1}T_{t,\Lambda}^{-1} =}{}
\cdot q^{\frac{1}{2} (\vartheta_1 - \vartheta_2)^2} \cdot \frac{1}{\varphi(d_1d_2x_1)\varphi(d_3d_4 x_2)}
T_{\frac{\kappa b_1}{b_2},x_1}T_{\frac{\kappa b_2}{b_1},x_2}.\label{gl2}
\end{align}

The rescaling $x_i \to - q^{\frac{1}{2}} x_i$ and the exchange of mass parameters $d_2 \leftrightarrow d_3$ implies
a complete matching of the Hamiltonian \eqref{eq:glN Hamiltonian} with $N=2$ and \eqref{gl2}.
When $N=2$, $\Delta = (\vartheta_1 - \vartheta_2)^2$ and the twist operation on $\varphi(-\check{x}_0)^{-1}$ and $\varphi(-\hat{x}_0)^{-1}$ is trivial.
Hence the Hamiltonian \eqref{eq:glN Hamiltonian} reduces~to
\begin{equation}\label{gl2-Hamiltonian}
\mathcal{H}^{\widehat{\mathfrak{gl}}_2}(x_i; b_i, d_i, q,t) = q^{\frac{1}{2} (\vartheta_1 - \vartheta_2)^2} \cdot \mathcal{A}_L \cdot \mathcal{A}_C \cdot \mathcal{A}_R
\cdot q^{\frac{1}{2} (\vartheta_1 - \vartheta_2)^2} \cdot T_{\frac{\kappa b_1}{b_{2}}, x_1} T_{\frac{\kappa b_2}{b_{1}}, x_2},
\end{equation}
where
\begin{align*}
&\mathcal{A}_L= \frac{1}{\varphi(-\check{x}_{2})} \frac{1}{\varphi(-\check{x}_{1})} \varphi(\Lambda),
\\
&\mathcal{A}_R= \varphi\bigl(q^{-1} d_1d_2d_3d_4 \Lambda \bigr) \frac{1}{\varphi(-d_1d_3\hat{x}_1)} \frac{1}{\varphi(-d_2d_4\hat{x}_2)},
\\
&\mathcal{A}_C= \frac{1}{\varphi(d_1 x_1)\varphi(d_3 x_1 )\varphi(d_2 x_2)\varphi(d_4 x_2)},
\end{align*}
and we have identified $x_0$ with $x_2$. In order to compare the Hamiltonian \eqref{gl2-Hamiltonian}
with the symmetric form of the ${\widehat{\mathfrak{gl}}_2}$ Hamiltonian \eqref{gl2},
we have to commute $q^{\frac{1}{2} (\vartheta_1 - \vartheta_2)^2} $ with $\mathcal{A}_L$ or $\mathcal{A}_R$ by
using the formula \eqref{adjoint-onpower} in the next subsection. This commutation removes the hat and the check on~$x_i$:
$\check{x}_i \to x_i$ and $\hat{x}_i \to x_i$
and also it scales $x_i$ by $q^{\pm \frac{1}{2}}$. Namely,
\begin{align}
\mathcal{A}_L \longrightarrow \widetilde{\mathcal{A}}_L&= \frac{1}{\varphi\bigl(-q^{\frac{1}{2}} x_{2}\bigr)} \frac{1}{\varphi\bigl(- q^{\frac{1}{2}} x_{1}\bigr)} \varphi(\Lambda),
\\
\mathcal{A}_R \longrightarrow \widetilde{\mathcal{A}}_R&= \varphi(q^{-1} d_1d_2d_3d_4 \Lambda )
 \frac{1}{\varphi\bigl(-d_1d_3 q^{-\frac{1}{2}} x_1\bigr)} \frac{1}{\varphi\bigl(-d_2d_4 q^{-\frac{1}{2}} x_2\bigr)},
\end{align}
Then as claimed above by the rescaling $x_i \to - q^{\frac{1}{2}} x_i$ and the exchange of mass parameters, we see the agreement of
the Hamiltonians \eqref{gl2-Hamiltonian} and \eqref{gl2}.

\subsection[q-Borel transformation for multi-variables]{$\boldsymbol{q}$-Borel transformation for multi-variables}

In the higher-rank generalization of~$\SS$, there appear the operators of the form $q^{\mathcal{L}}$, where $\mathcal{L}$ is a second order polynomial
in the Euler derivative $\vartheta_i$. For later convenience, let us work out the commutation relations with coordinate variables $x_i$.
Since $q^{\frac{1}{2}\vartheta_i^2} x_i \cdot x_i^{n} = q^{\frac{1}{2}(n+1)^2} x_i^{n+1}
= q^{n + \frac{1}{2}} x_i q^{\frac{1}{2}\vartheta_i^2} \cdot x_i^n = q^{\frac{1}{2}} x_i p_i q^{\frac{1}{2}\vartheta_i^2} \cdot x_i^n$,
we see
\begin{equation}\label{xx-com}
{\mathrm{Ad}}\bigl(q^{\pm \frac{1}{2}\vartheta_i^2}\bigr) \cdot x_i = q^{\pm \frac{1}{2}} x_i p_i^{\pm 1}.
\end{equation}
We also have
\begin{equation}\label{xy-com}
{\mathrm{Ad}}\bigl(q^{\pm \vartheta_i \vartheta_j}\bigr) \cdot x_i = x_i p_j^{\pm 1}, \qquad i \neq j.
\end{equation}
Let us introduce the relative $q$-Borel transformation $\widetilde{\Bor}_{xy}$ by (see equation~\eqref{relative})
\begin{equation}
\widetilde{\Bor}_{xy} := q^{\frac{1}{2}(\vartheta_x - \vartheta_y)(\vartheta_x - \vartheta_y +1)}
= q^{\frac{1}{2}(\vartheta_x^2 + \vartheta_y^2)} q^{-\vartheta_x \vartheta_y} p_x^{\frac{1}{2}} p_y^{-\frac{1}{2}}.
\end{equation}
Namely,
\begin{equation}
\widetilde{\Bor}_{xy} \cdot x^n y^m = q^{\frac{1}{2}(n-m)(n-m+1)} x^n y^m.
\end{equation}
From \eqref{xx-com} and \eqref{xy-com}, we obtain the commutation relation
\begin{equation}
{\mathrm{Ad}}(\widetilde{\Bor}_{xy}) \cdot x = p_x p_y^{-1} x,
\qquad
{\mathrm{Ad}}(\widetilde{\Bor}_{xy}) \cdot y = y p_x^{-1} p_y.
\end{equation}
We note that for $x'= p_x p_y^{-1} x$ and $y' = y p_x^{-1} p_y$, we have $x'y' = y' x' = xy$.
On the power of $x_i$, we have
\begin{equation}\label{adjoint-onpower}
{\mathrm{Ad}}\bigl(q^{\pm \frac{1}{2}(\vartheta_i - \vartheta_j)^2}\bigr) \cdot x_i^n = q^{\pm \frac{1}{2}n^2} x_i^n p_{ij}^{\pm n}
= q^{\pm\frac{n}{2}}\bigl(x_i p_{ij}^{\pm 1}\bigr)^n, \qquad p_{ij} = p_i/p_j, \qquad i \neq j.
\end{equation}
This formula is useful in the computation of the normal ordered Hamiltonian.


\section{Two types of the affine Laumon partition function
}\label{App-C}



There are two types of the Nekrasov factor, which we call Pochhammer type and hyperbolic-sine (sinh) type.
From the view point of the index theorem for the instanton moduli space,
they come from the Dolbeault operator and the Dirac operator, respectively.
Consequently we have the affine Laumon partition function of Pochhammer type and of sinh type.
When the moduli space is hyperK\"ahler, the index of the Dolbeault operator
and the Dirac operator coincide, since the discrepancy is measured by
the first Chern class. However,
the affine Laumon space is not hyperK\"ahler, because of the asymmetry of the chain-saw quiver \cite{FFNR,FR}
and two types of the partition function are different in general.
In this appendix, we will show that the $q$-Borel transformation transforms
the affine Laumon partition function of sinh type into of Pochhammer type.

\subsection[q-Borel transformation from sinh type into Pochhammer type]{$\boldsymbol{q}$-Borel transformation from sinh type into Pochhammer type}

Let $\lambda$ be a Young diagram,
i.e.,
a partition $\lambda = (\lambda_1,\lambda_2,\cdots)$,
which is a sequence of nonnegative integers such that
$\lambda_{i} \geq \lambda_{i+1}$ and
$|\lambda| = \sum_i \lambda_i < \infty$.
$\lambda^\vee $ denotes its conjugate (dual) diagram.
We define
\be
{\fweight\lambda k}:= \sum_{n \in \mathbb{Z}}
\lambda_{k+nN}, \qquad k \in \mathbb{Z}/N\mathbb{Z},
\ee
where we set $\lambda_i = 0$ for $i \leq 0$.
Throughout Appendix~\ref{App-C}, we use the notation
\begin{equation}
\mathsf{v} := (q{\tkappa})^{\ha}.
\end{equation}
In this appendix, the notation $\equiv$ always means the congruence of integers modulo $N$.

For a pair of Young diagrams $\lambda$ and $\mu$, we define
the orbifolded Nekrasov factor of Pochhammer type as
\begin{align}
& { \fNekN kN{\lambda}{\mu}{\mathsf{v}Q}q{\tkappa} }
:=
\prod_{\Frac{(i,j)\in\mu}{\mu^\vee_j-i+\ha \equiv -k-\ha}}
\bigl(1-Qq^{\lambda_i-j+\ha} {\tkappafactor{-\mu^\vee_j+i-\ha} }\bigr)\nonumber\\
&\hphantom{{ \fNekN kN{\lambda}{\mu}{\mathsf{v}Q}q{\tkappa} } :=}{}
\cdot
\prod_{\Frac{(i,j)\in\lambda}{\lambda^\vee_j-i+\ha \equiv k+\ha}}
\bigl(1-Qq^{-\mu_i+j-\ha} {\tkappafactor{\lambda^\vee_j-i+\ha} }\bigr).
\label{eq:ONekIpDef}
\end{align}
The orbifolded Nekrasov factor of sinh type
(\ref{inf-prod-form}) is written as
\begin{align}
& { \fNekNs kN{\lambda}{\mu}{\mathsf{v}Q}q{\tkappa} }
=
\prod_{\Frac{(i,j)\in\mu}{\mu^\vee_j-i+\ha \equiv -k-\ha}}
\bigl[Qq^{\lambda_i-j+\ha} {\tkappafactor{-\mu^\vee_j+i-\ha} }\bigr]\nonumber\\
&\hphantom{{ \fNekNs kN{\lambda}{\mu}{\mathsf{v}Q}q{\tkappa} }=}{}
\cdot
\prod_{\Frac{(i,j)\in\lambda}{\lambda^\vee_j-i+\ha \equiv k+\ha}}
\bigl[Qq^{-\mu_i+j-\ha} {\tkappafactor{\lambda^\vee_j-i+\ha} }\bigr],
\end{align}
with $[x]:=x^{-\ha}-x^{\ha}$ (see \cite[Appendix~F]{AHKOSSY2}).
It satisfies\footnote{Such a simple symmetry is specific to the Nekrasov factor of sinh type.
But both
${ \fNekNs kN{\lambda}{\mu}{\mathsf{v}Q}q{\tkappa} }$ and
${ \fNekN kN{\lambda}{\mu}{\mathsf{v}Q}q{\tkappa} }$
satisfy the inversion formula of type
 \[
{ \fNekN kN{\lambda}{\mu}{\mathsf{v}Q}q{\tkappa} }
=
N^{{\rm Poch }(-k-1\vert N)}_{\mu\lambda}\bigl(\mathsf{v}^{-1}Q\vert q^{-1},\tkappa^{-1}\bigr),
\]
which follows directly from the definition.}{\samepage
 \be
{ \fNekNs kN{\lambda}{\mu}{\mathsf{v}Q}q{\tkappa} }
=
(-1)^n
\mathsf{N}^{(k\vert kN)}_{\lambda\mu}\bigl(\mathsf{v}^{-1}Q^{-1}\vert q^{-1},\tkappa^{-1}\bigr),
\ee
where $n= {\fweight\mu {-k}}+ {\fweight\lambda {1+k}}$, as we will show in (\ref{eq:gLemma}).}

For $N$-tuple of Young diagrams
$\lambda^{(i)}$ with $i \in \mathbb{Z}/N\mathbb{Z}$,
let $\vlambda:=\bigl(\lambda^{(1)}, \lambda^{(2)},\dots, \lambda^{(N)}\bigr)$.
For $N$ variables
$x_i\in\bC$ with $i \in \mathbb{Z}/N\mathbb{Z}$,
let $\vx:=\left(x_1,x_2,\dots, x_N\right)$.
For $3N^2$ variables
${\Kahlerm \ua ij}, {\Kahlerm \vb ij}, {\Kahlerm \wc ij}\in\bC$ with $i,j \in \mathbb{Z}/N\mathbb{Z}$,
let
\begin{align}
&{ \NekZs{\vlambda}{\vx} }
:=
\prod_{i,j=1}^N
\frac{
{\fNekNs {j-i}N{\varnothing}{\lambda^{(j)}} { {\mathsf{v}\Kahlerm \ua ij} }q{\tkappa} }
{\fNekNs {j-i}N{\lambda^{(i)}}{\varnothing} { {\mathsf{v}\Kahlerm \wc ij} }q{\tkappa} }
}{
{\fNekNs {j-i}N{\lambda^{(i)}}{\lambda^{(j)}} { {\mathsf{v}\Kahlerm \vb ij} }q{\tkappa} }
}
x_i^{ {\fweight{\lambda^{(j)}} {1+i-j}} },
\label{eq:defZlambda}
\\
&{ \NekZps{\vlambda}{\vx} }
:=
\prod_{i,j=1}^N
\frac{1}{
{\fNekNs {j-i}N{\lambda^{(i)}}{\lambda^{(j)}} { {\mathsf{v}\Kahlerm \vb ij} }q{\tkappa} }
}
x_i^{ {\fweight{\lambda^{(j)}} {1+i-j}} }.
\label{eq:defZpurelambda}
\end{align}
We define $\NekZ{\vlambda}{\vx}$ and $\NekZp{\vlambda}{\vx}$ by replacing the orbifolded Nekrasov factor
of sinh type with that of Pochhammer type.

%
As we will show in (\ref{eq:gggLemma}),
the denominators and the numerators of
(\ref{eq:defZlambda}) and
(\ref{eq:defZpurelambda})
have even numbers of factors or brackets $[\quad]$.
Therefore, even if we change the sign of $[x]$,
${ \NekZs{\vlambda}{\vx} }$
and
${ \NekZps{\vlambda}{\vx} }$
are unchanged.
%


Let
$\vt_{x_i}:=x_i \frac \partial{\partial x_i}$
and
\be
\Delta
:=
\ha\sum_{i=1}^N
({\vt_{x_{i-1}} }-{\vt_{x_i} })^2
=\sum_{i=1}^N
\bigl( \vt_{x_i}^2 -\vt_{x_{i-1}}\vt_{x_i} \bigr).
\ee
Since $\vt_{x} x = x (1+\vt_{x} )$,
$q^{\vt_{x}^2} x^n \cdot 1 = x^n q^{(n+\vt_{x})^2} \cdot 1 = x^n q^{n^2} \cdot 1 $.
Thus, for any $c\in\bC$, ${ \NekZs{\vlambda}{\vx} }$ satisfies
\begin{align}
&q^{ {\frac c2}
\vt_{x_i} }
{ \NekZs{\vlambda}{\vx} }
=
{ \NekZs{\vlambda}{\vx} }
q^{ {\frac c2}
\sum_{j=1}^{N}
{\fweight{\lambda^{(j)}}{1+i-j} }
},
\label{eq:ZZshift}
\\
&q^{ {\frac c2}\Delta
}
{ \NekZs{\vlambda}{\vx} }
=
{ \NekZs{\vlambda}{\vx} }
\prod_{i=1}^N
q^{ {\frac c4} 
\left(\sum_{j=1}^{N}\left(
{\fweight{\lambda^{(j)}}{i-j} }-{\fweight{\lambda^{(j)}}{1+i-j} }
\right)\right)^2}.
\label{eq:ZZqBorel}
\end{align}
The same relations are also valid for
${ \NekZps{\vlambda}{\vx} }$, ${ \NekZ{\vlambda}{\vx} }$ and ${ \NekZp{\vlambda}{\vx} }$.

\begin{Proposition}\label{Prop-C}
When
\be
{\mathsf{v}\Kahlerm \ua ij} = \frac {\ua_i}{\vb_j}, \quad
{\mathsf{v}\Kahlerm \vb ij} = \frac{\vb_i}{\vb_j}, \quad
{\mathsf{v}\Kahlerm \wc ij} = \frac{\vb_i}{\wc_j}, 
\label{eq:Quvw}
\ee
with $3N$ variables $\ua_i$, $\vb_i$ and $\wc_i$,
we have
\begin{align}
&{ \NekZ{\vlambda}{\vx} }
=
q^{\ha\Delta
}
\prod_{i=1}^N
\left(
\frac{\ua_i}{\vb_i} \frac{\vb_{i-1}}{\wc_{i-1}}\right)^{\ha\vt_{x_{i-1}}}
\cdot
{ \NekZs{\vlambda}{\vx} },
\label{eq:ZZ}
\\
&{ \NekZp{\vlambda}{\vx} }
=
q^{\ha \Delta
}
\prod_{i=1}^N
\left(\frac{\vb_{i-1}}{\vb_i} q{\tkappa} \right)^{\ha\vt_{x_{i-1}}}
\cdot
{ \NekZps{\vlambda}{\vx} }.
\label{eq:ZZpure}
\end{align}
\end{Proposition}

Remark that, in the case of $N=2$,
$\Delta=\vt_{x_1}^2 -2\vt_{x_{1}}\vt_{x_2} +\vt_{x_2}^2$.
Remark also that,
since
smash{$\sum_{i=1}^{N}{\fweight{\lambda^{(j)}} {i-j}} = {\fweight{\lambda^{(j)}} {}}$},\footnote{
The inverse square of the last factor 
of (\ref{eq:ZZpureS})
\[
q^{-\ha
 (\sum_{j=1}^{N} (
{\fweight{\lambda^{(j)}}{i-j} }-{\fweight{\lambda^{(j)}}{1+i-j} }
 ) )^2}
\vb_{i}^{\sum_{j=1}^{N} (
{\fweight{\lambda^{(j)}}{i-j} }-{\fweight{\lambda^{(j)}}{1+i-j} }
 ) }
 {\tkappafactor{ - { \fweight{\lambda^{(i)}} { }} } }
\]
up to $q^{ {\fweight{\lambda^{(i)}}{ }} }$
is the same as
the prefactor $s_i^{-m_i} q^{-m_i^2/2} \kappa^{-|\lambda^{(i)}|}$
of~(12) in~\cite{rf:Shiraishi}.
Here
$s_i := 1/\vb_i$ and
$m_i := \sum_{j=1}^{N}\bigl(
{\fweight{\lambda^{(j)}}{i-j} }-{\fweight{\lambda^{(j)}}{1+i-j} }
\bigr) $.
}
\begin{align}
\label{eq:ZZpureS}
&
{ \NekZp{\vlambda}{\vx} }
\cr
& \quad =
{ \NekZps{\vlambda}{\vx} }
\prod_{i=1}^N
q^{\frac 14
 (\sum_{j=1}^{N} (
{\fweight{\lambda^{(j)}}{i-j} }-{\fweight{\lambda^{(j)}}{1+i-j} }
 ) )^2}
\vb_{i}^{-\ha\sum_{j=1}^{N} (
{\fweight{\lambda^{(j)}}{i-j} }-{\fweight{\lambda^{(j)}}{1+i-j} }
 ) }
 (q{\tkappa} )^{\ha{\fweight{\lambda^{(i)}}{}} }.
\end{align}

For any function $f$ of $\ua_i$, $\vb_j$, $\wc_k$, $q$, $t$, let
\begin{align}
&
\overline{f(\ua_1,\ua_2,\dots, \vb_1,\vb_2,\dots, \wc_1,\wc_2,\dots,q,{\tkappa}) }
\cr
&\qquad :=
 f\bigl(\ua_1^{-1},\ua_2^{-1},\dots, \vb_1^{-1},\vb_2^{-1},\dots, \wc_1^{-1},\wc_2^{-1},\dots,q^{-1},{ \tkappa^{-1} }\bigr).
\end{align}
Since,
$\overline{ \NekZs{\vlambda}{\vx} }={ \NekZs{\vlambda}{\vx} }$ and
$\overline{ \NekZps{\vlambda}{\vx} }={ \NekZps{\vlambda}{\vx} }$,
we have
\begin{align}
&\overline{ \NekZ{\vlambda}{\vx} }
=
q^{-\ha\Delta
}
\prod_{i=1}^N
\left(
\frac{\ua_i}{\vb_i} \frac{\vb_{i-1}}{\wc_{i-1}}\right)^{-\ha\vt_{x_{i-1}}}
\cdot
{ \NekZs{\vlambda}{\vx} },
\\
&\overline{ \NekZp{\vlambda}{\vx} }
=
q^{-\ha \Delta
}
\prod_{i=1}^N
\left(\frac{\vb_{i-1}}{\vb_i} q{\tkappa} \right)^{-\ha\vt_{x_{i-1}}}
\cdot
{ \NekZps{\vlambda}{\vx} }.
\end{align}

Let $[x]_c:=x^{\frac c2}[x]$ with $[x]=x^{-\ha}-x^{\ha}$.
From \smash{${ \NekZs{\vlambda}{\vx} }$} and \smash{${ \NekZps{\vlambda}{\vx} }$},
we define \smash{${ \NekZsc{c}{\vlambda}{\vx} }$} and \smash{${ \NekZpsc{c}{\vlambda}{\vx} }$}
by replacing $[x]$ in the orbifolded Nekrasov factor of sinh type with~$[x]_c$.
For example,
\[
{ \NekZsc{0}{\vlambda}{\vx} }= {\NekZs{\vlambda}{\vx} } ,\qquad
 { \NekZsc{1}{\vlambda}{\vx} }= {\NekZ{\vlambda}{\vx} } \qquad \text{and}\qquad
 { \NekZsc{-1}{\vlambda}{\vx} }= \overline{\NekZ{\vlambda}{\vx} }.
 \]
In view of \eqref{eq:ZZshift} and \eqref{eq:ZZqBorel},
it should be clear that even if we
replace
\smash{${ \NekZ{\vlambda}{\vx} }$ and ${ \NekZp{\vlambda}{\vx} }$}
with
\smash{${ \NekZsc{c}{\vlambda}{\vx} }$} and \smash{${ \NekZpsc{c}{\vlambda}{\vx} }$}, respectively,
Proposition~\ref{Prop-C} is true,
if $\ha\Delta$ and $\ha\vartheta_{x_{i-1}}$ are replaced with
${\frac c2}\Delta$ and ${\frac c2}\vartheta_{x_{i-1}}$,
respectively.

\subsection{Proof of the Proposition~\ref{Prop-C}}

Before embarking on a proof of Proposition~\ref{Prop-C}, it is convenient to introduce a few notations.
Since $(1-x) = x^{\frac{1}{2}}[x]$, we have
\be\label{Nekrasovfactor-rel}
{ \fNekN kN{\lambda}{\mu}{\mathsf{v}Q}q{\tkappa} }
=
{ \fNekNs kN{\lambda}{\mu}{\mathsf{v}Q}q{\tkappa} }
\bigl(
{\ffactor kN {\lambda}{\mu} }
{\gfactor kN {\lambda}{\mu} Q }
\bigr)^\ha,
\ee
where 
\begin{align}
&{\ffactor k N \lambda\mu}
:=
\prod_{\Frac{(i,j)\in\mu}{\mu^\vee_j-i+\ha \equiv N-k-\ha}}
q^{\lambda_i-j+\ha} {\tkappafactor{-\mu^\vee_j+i-\ha} }
\cdot
\prod_{\Frac{(i,j)\in\lambda}{\lambda^\vee_j-i+\ha \equiv k+\ha}}
q^{-\mu_i+j-\ha} {\tkappafactor{\lambda^\vee_j-i+\ha} },
\\
&{\gfactor k N \lambda\mu Q}
:=
\prod_{\Frac{(i,j)\in\mu}{\mu^\vee_j-i+\ha \equiv N-k-\ha}}
Q
\cdot
\prod_{\Frac{(i,j)\in\lambda}{\lambda^\vee_j-i+\ha \equiv k+\ha}}
Q.
\label{eq:gDef}
\end{align}
Let us denote
\begin{align}
&f_\vlambda
:=
\prod_{i,j=1}^N
\frac{
{\ffactor {j-i}N \varnothing{\lambda^{(j)}} }
{\ffactor {j-i}N {\lambda^{(i)}}\varnothing }
}{
{\ffactor {j-i}N {\lambda^{(i)}}{\lambda^{(j)}} }
},
\qquad
g_\vlambda
:=
\prod_{i,j=1}^N
\frac{
{g^{(j-i|N)}_{\varnothing\lambda^{(j)}} \bigl(Q_{i,j}^a\bigr)}
{g^{(j-i|N)}_{\lambda^{(i)}\varnothing} \bigl(Q_{i,j}^c\bigr)}
}{
{g^{(j-i|N)}_{\lambda^{(i)} \lambda^{(j)}} \bigl(Q_{i,j}^b\bigr)}
},
\\
&f^{{\rm num}}_\vlambda
:=
\prod_{i,j=1}^N
{\ffactor {j-i} N \varnothing{\lambda^{(j)}} }
{\ffactor {j-i} N {\lambda^{(i)}}\varnothing }
=
\prod_{i,j=1}^N
{\ffactor {j-i} N \varnothing{\lambda^{(j)}} }
{\ffactor {i-j-1} N {\lambda^{(j)}}\varnothing },
\\
&g^{{\rm num}}_\vlambda
:=
\prod_{i,j=1}^N
{\gfactor {j-i} N \varnothing{\lambda^{(j)}} { {\Kahlerm \ua ij} } }
{\gfactor {j-i} N {\lambda^{(i)}}\varnothing { {\Kahlerm \wc ij} } }
=
\prod_{i,j=1}^N
{\gfactor {j-i} N \varnothing{\lambda^{(j)}} { {\Kahlerm \ua ij} } }
{\gfactor {i-j-1} N {\lambda^{(j)}}\varnothing { {\Kahlerm \wc j{i-1}} } },
\\
&f^{{\rm pure}}_\vlambda
:=
\frac{f_\vlambda}{f^{{\rm num}}_\vlambda}
=
\prod_{i,j=1}^N
\frac{
1}{
{\ffactor {j-i}N {\lambda^{(i)}}{\lambda^{(j)}} }
},
\qquad 
g^{{\rm pure}}_\vlambda
:=
\frac{g_\vlambda}{g^{{\rm num}}_\vlambda}
=
\prod_{i,j=1}^N
\frac{
1}{
{g^{(j-i|N)}_{\lambda^{(i)} \lambda^{(j)}} \bigl(Q_{i,j}^b\bigr)}
}.
\label{eq:gpvlambdaDef}
\end{align}
From \eqref{Nekrasovfactor-rel},
we have
\begin{align}
&{ \NekZ{\vlambda}{\vx} }
=
\left(f_\vlambda g_\vlambda\right)^\ha
{ \NekZs{\vlambda}{\vx} },
\label{eq:ZZfg}
\\
&{ \NekZp{\vlambda}{\vx} }
=
\bigl(f^{{\rm pure}}_\vlambda g^{{\rm pure}}_\vlambda\bigr)^\ha
{ \NekZps{\vlambda}{\vx} }.
\label{eq:ZpZpfg}
\end{align}
Hence to prove Proposition~\ref{Prop-C}, it is enough to evaluate
$f_\vlambda$, $g_\vlambda$, $f^{{\rm num}}_\vlambda$ and $g^{{\rm num}}_\vlambda$.
This is achieved by the following four steps.

\subsubsection{Step 1: Good combination}

We have
\be
{\ffactor {-k} N \mu\lambda}
=
\hskip-15pt
\prod_{\Frac{(i,j)\in\lambda}{\lambda^\vee_j-i+\ha \equiv k-\ha}}
\hskip-15pt
q^{\mu_i-j+\ha} {\tkappafactor{-\lambda^\vee_j+i-\ha} }
\cdot
\hskip-15pt
\prod_{\Frac{(i,j)\in\mu}{\mu^\vee_j-i+\ha \equiv N-k+\ha}}
\hskip-20pt
q^{-\lambda_i+j-\ha} {\tkappafactor{\mu^\vee_j-i+\ha} },
\ee
and we can eliminate $\kappa$ from $f$-factors by taking the following combinations:
\begin{align}
&\frac{
{\ffactor k N \varnothing\mu}{\ffactor k N \lambda\varnothing}
}{
{\ffactor k N \lambda\mu}
}
=
\prod_{\Frac{(i,j)\in\mu}{\mu^\vee_j-i+\ha \equiv N-k-\ha}}
q^{-\lambda_i}
\cdot
\prod_{\Frac{(i,j)\in\lambda}{\lambda^\vee_j-i+\ha \equiv k+\ha}}
q^{\mu_i} ,
\label{eq:fffkqq}
\\
&\frac{
{\ffactor {-k} N \varnothing\lambda}{\ffactor {-k} N \mu\varnothing}
}{
{\ffactor {-k} N \mu\lambda}
}
=
\prod_{\Frac{(i,j)\in\lambda}{\lambda^\vee_j-i+\ha \equiv k-\ha}}
q^{-\mu_i}
\cdot
\prod_{\Frac{(i,j)\in\mu}{\mu^\vee_j-i+\ha \equiv N-k+\ha}}
q^{\lambda_i}.
\end{align}
We also have
\begin{align}
&{\ffactor k N \lambda\mu}
{\ffactor {-1-k} N \mu\lambda}
=1,
\label{eq:ffg}
\\
&{\gfactor k N \lambda\mu Q}
{\gfactor {-1-k} N \mu\lambda {Q'}}
=
\prod_{\Frac{(i,j)\in\mu}{\mu^\vee_j-i+\ha \equiv N-k-\ha}}
QQ'
\cdot
\prod_{\Frac{(i,j)\in\lambda}{\lambda^\vee_j-i+\ha \equiv k+\ha}}
QQ'
=
{\gfactor k N \lambda\mu {QQ'} }. \label{eq:ggg}
\end{align}
Note that, by (\ref{eq:fffkqq}),
\be
f_\vlambda
=
\prod_{i,j=1}^N
\hskip-10pt
\prod_{\Frac{(a,b)\in\lambda^{(j)} }{ {\lambda^{(j)} }^\vee_b-a+\ha \equiv i-j-\ha}}
\hskip-26pt
q^{-\lambda^{(i)}_a}
\hskip3pt
\cdot
\prod_{\Frac{(a,b)\in\lambda^{(j)} }{ {\lambda^{(j)} }^\vee_b-a+\ha \equiv i-j+\ha} }
\hskip-26pt
q^{\lambda^{(i)}_a} .~~
\label{eq:fffkqqLemma}
\ee

\subsubsection{Step 2: Box product}

Since $\lambda^\vee_j-i$ is the leg length of the square $(i,j)$
in the Young diagram of a partition $\lambda$, we have
\be
\sum_{\Frac{(i,j)\in\lambda}{\lambda^\vee_j-i+1 \equiv k}}
1
= \sum_{n=0}^\infty\lambda_{k+nN}
= {\fweight\lambda k}.
\ee
Hence, the power of $Q$'s in (\ref{eq:gDef})
is
$ {\fweight\mu {-k}}+ {\fweight\lambda {1+k}}$,
and
\be
{\gfactor k N \lambda\mu Q}
=Q^{{\fweight\mu {-k}}+ {\fweight\lambda {1+k}}}.
\label{eq:gLemma}
\ee

\begin{Lemma}
For $1\leq k \leq N-1$,
\begin{align}
&\frac{
^{\ffactor k N \varnothing\mu}{\ffactor k N \lambda\varnothing}
{\ffactor{-k} N \varnothing\lambda}{\ffactor{-k} N \mu\varnothing}
}{
{\ffactor k N \lambda\mu}
{\ffactor{-k} N \mu\lambda}
}
\nonumber\\
&\qquad=
\prod_{\Frac{i,j=1}{j-i \equiv k}}^N
q^{\ha
\left({\fweight\mu{i-1} }-{\fweight\mu{i} }\right)
\left({\fweight\lambda{j-1} }-{\fweight\lambda{j} }\right)
}
\prod_{\Frac{i,j=1}{j-i \equiv -k}}^N
q^{\ha\left({\fweight\lambda{i-1} }-{\fweight\lambda{i} }\right)
\left({\fweight\mu{j-1} }-{\fweight\mu{j} }\right)
},
\label{eq:ffactorN}
\\
&\frac{
{\ffactor 0 N \varnothing\lambda}{\ffactor 0 N \lambda\varnothing}
}{
{\ffactor 0 N \lambda\lambda}
}
=
q^{\ha\sum_{i=1}^N
\left({\fweight\lambda{i-1} }-{\fweight\lambda{i} }\right)^2}.
\label{eq:ffactorZero}
\end{align}
For $0\leq k \leq N-1$,
\begin{align}
\frac{
{\gfactor k N \varnothing\mu {Q^{\ua}} }{\gfactor k N \lambda\varnothing {Q^{\wc}} }
}{
{\gfactor k N \lambda\mu {Q^{\vb}} }
}
&=
\left(\frac{Q^{\ua}}{Q^{\vb}}\right)^{\fweight{\mu} {-k}}
\left(\frac{Q^{\wc}}{Q^{\vb}}\right)^{\fweight{\lambda} {1+k}}
\label{eq:gfactorQ}
\end{align}
and\footnote{For the Nekrasov partition function without surface defect,
the $f$-factors corresponding to (\ref{eq:ffactorN}) and (\ref{eq:ffactorZero}) are
equal to $1$ and the $g$-factors corresponding to (\ref{eq:gfactorQ}) and (\ref{eq:ggfactor}) are given by
replacing ${\fweight{\lambda}k }$ with ${\fweight{\lambda}{}}$.
Thus the Nekrasov partition function without surface defect
satisfies (\ref{eq:ZZ}) without $q$-Borel transformation $q^{\ha\Delta}$
and by replacing ${\fweight{\lambda}k }$ with ${\fweight{\lambda}{}}$.}
\begin{align}
{\ffactor k N \lambda\mu}
{\ffactor {-1-k} N \mu\lambda}
= 1,
\qquad
{\gfactor k N \lambda\mu Q}
{\gfactor {-1-k} N \mu\lambda {Q'}}
=
(QQ')^{{\fweight\mu {-k}}+ {\fweight\lambda {1+k}}}.
\label{eq:ggfactor}
\end{align}
\end{Lemma}

\begin{proof}
By (\ref{eq:gLemma}),
\be
\frac{
{\gfactor k N \varnothing\mu {Q^{\ua}} }{\gfactor k N \lambda\varnothing {Q^{\wc}} }
}{
{\gfactor k N \lambda\mu {Q^{\vb}} }
}
=
\frac{
\left(Q^{\ua}\right)^{\fweight{\mu} {-k}}
\left(Q^{\wc}\right)^{\fweight{\lambda} {1+k}}
}{
\left(Q^{\vb}\right)^{{\fweight\mu {-k}}+ {\fweight\lambda {1+k}}}
},
\ee
which gives (\ref{eq:gfactorQ}).
From (\ref{eq:ffg}), (\ref{eq:ggg}) and (\ref{eq:gLemma}), we get
(\ref{eq:ggfactor}).

For any $j,s\in\bZ_{>0}$,
$\lambda^\vee_j=s$
if and only if 
$1+\lambda_{s+1}\leq j \leq \lambda_s$.
Thus, for $0\leq k \leq N-1$,
\begin{align}
\prod_{\Frac{(i,j)\in\lambda}{\lambda^\vee_j-i+\ha \equiv k+\ha}}
q^{\mu_i}
&=
\prod_{i\geq 1}\prod_{n\geq 0}
\prod_{\Frac{j\geq 1}{\lambda^\vee_j=i+k+nN}}
q^{\mu_i}
\nonumber\\
&=
\prod_{i\geq 1}\prod_{n\geq 0}
\prod_{\Frac{s\geq 1}{s=i+k+nN}}
q^{\mu_i(\lambda_s-\lambda_{s+1})}
=
\prod_{n\geq 0}
\prod_{\Frac{i,j\in\bZ}{j-i=k+nN}}
q^{\mu_{i-1}(\lambda_{j-1}-\lambda_{j})},
\label{eq:mu}
\end{align}
where we have used $\lambda^\vee_j=s\Leftrightarrow 1+\lambda_{s+1}\leq j \leq \lambda_s$ for the second equality.
By replacing
$q$, $k$, $\lambda$ and $\mu$ in (\ref{eq:mu})
with
$1/q$, $N-1-k$, $\mu $ and $\lambda$, respectively, we have
\begin{align}
\prod_{\Frac{(i,j)\in\mu}{\mu^\vee_j-i+\ha \equiv N-k-\ha}}
q^{-\lambda_i}
&=
\prod_{n\geq 0}
\prod_{\Frac{i,j\in\bZ}{j-i+1=N-k+nN}}
q^{-\lambda_{i-1}(\mu_{j-1}-\mu_{j})}\nonumber\\
&
=
\prod_{n\geq 1}
\prod_{\Frac{i,j\in\bZ}{j-i=-k+nN}}
q^{-\lambda_{i}(\mu_{j-1}-\mu_{j})}.
\label{eq:mlambda}
\end{align}
Similarly, by replacing
$q$, $k$ in (\ref{eq:mu}) and (\ref{eq:mlambda})
with
$1/q$, $k-1$, respectively, we obtain the formulas
for $1\leq k \leq N$.
Combining them, we have
\begin{align}
&\prod_{\Frac{(i,j)\in\lambda}{\lambda^\vee_j-i+\ha \equiv k+\ha}}
q^{\mu_i}
\cdot
\prod_{\Frac{(i,j)\in\lambda}{\lambda^\vee_j-i+\ha \equiv k-\ha}}
q^{-\mu_i}
=
\prod_{n\geq 0}
\prod_{\Frac{i,j\in\bZ}{j-i=k+nN}}
q^{(\mu_{i-1}-\mu_{i})(\lambda_{j-1}-\lambda_{j})},
\nonumber\\
&\prod_{\Frac{(i,j)\in\mu}{\mu^\vee_j-i+\ha \equiv N-k-\ha}}
q^{-\lambda_i}
\cdot
\prod_{\Frac{(i,j)\in\mu}{\mu^\vee_j-i+\ha \equiv N-k+\ha}}
\hskip-15pt
q^{\lambda_i}
=
\prod_{n\geq 1}
\prod_{\Frac{i,j\in\bZ}{j-i=-k+nN}}
q^{(\lambda_{i-1}-\lambda_{i})(\mu_{j-1}-\mu_{j})}. \nonumber
\end{align}
Then, for $1\leq k \leq N-1$,
\begin{align*}
\frac{
{\ffactor k N \varnothing\mu}{\ffactor k N \lambda\varnothing}
{\ffactor{-k} N \varnothing\lambda}{\ffactor{-k} N \mu\varnothing}
}{
{\ffactor k N \lambda\mu}
{\ffactor{-k} N \mu\lambda}
}
&=
\prod_{\Frac{(i,j)\in\mu}{\mu^\vee_j-i+\ha \equiv N-k-\ha}}
q^{-\lambda_i}
\cdot
\prod_{\Frac{(i,j)\in\lambda}{\lambda^\vee_j-i+\ha \equiv k+\ha}}
q^{\mu_i}\\
&\quad{}
\cdot
\prod_{\Frac{(i,j)\in\lambda}{\lambda^\vee_j-i+\ha \equiv k-\ha}}
q^{-\mu_i}
\cdot
\prod_{\Frac{(i,j)\in\mu}{\mu^\vee_j-i+\ha \equiv N-k+\ha}}
q^{\lambda_i}
\nonumber\\
&=
\prod_{n\in\bZ}
\prod_{\Frac{i,j\in\bZ}{j-i=k+nN}}
q^{(\mu_{i-1}-\mu_{i})(\lambda_{j-1}-\lambda_{j})}
\nonumber\\
&=
\prod_{\Frac{i,j=1}{j-i\equiv k}}^N
\prod_{n,m\geq 0}
q^{(\mu_{i-1+nN}-\mu_{i+nN})(\lambda_{j-1+mN}-\lambda_{j+mN})},
\end{align*}
which reduces to (\ref{eq:ffactorN}).
In
\smash{$\frac{
{\ffactor k N \varnothing\mu}{\ffactor k N \lambda\varnothing}
}{
{\ffactor k N \lambda\mu}
}
$}
and
\smash{$\frac{
{\ffactor{-k} N \varnothing\lambda}{\ffactor{-k} N \mu\varnothing}
}{
{\ffactor{-k} N \mu\lambda}
}
$},
$k$ should be
$0\leq k \leq N-1$ and $1\leq k \leq N$, respectively.
But
\[
\frac{
{\ffactor 0 N \varnothing\lambda}{\ffactor 0 N \lambda\varnothing}
}{
{\ffactor 0 N \lambda\lambda}
}
=
\frac{
{\ffactor{-N} N \varnothing\lambda}{\ffactor{-N} N \lambda\varnothing}
}{
{\ffactor{-N} N \lambda\lambda}
}
=
\left(
\frac{
{\ffactor 0 N \varnothing\lambda}{\ffactor 0 N \lambda\varnothing}
}{
{\ffactor 0 N \lambda\lambda}
}
\frac{
{\ffactor{-N} N \varnothing\lambda}{\ffactor{-N} N \lambda\varnothing}
}{
{\ffactor{-N} N \lambda\lambda}
}\right)^\ha,
\]
which gives (\ref{eq:ffactorZero}).
\end{proof}

\subsubsection[Step 3: N\^{}2 product]{Step 3: $\boldsymbol{N^2}$ product}

We have the following lemma.

\begin{Lemma}\label{L-3}
\begin{align}
&f_\vlambda
=
\prod_{i=1}^N
q^{\frac 12
\left(\sum_{j=1}^{N}\left(
{\fweight{\lambda^{(j)}}{i-j} }-{\fweight{\lambda^{(j)}}{1+i-j} }
\right)\right)^2} ,
\label{eq:fffLemma}
\qquad
f^{{\rm num}}_\vlambda
= 1, \\
&g_\vlambda
=
\prod_{i,j=1}^N
\left(
\frac{ {\Kahlerm \ua ij} {\Kahlerm \wc j{i-1}} }{ {\Kahlerm \vb ij} {\Kahlerm \vb j{i-1}} }
\right)^{{\fweight{\lambda^{(j)}} {i-j}}}
,
\qquad
g^{{\rm num}}_\vlambda
=
\prod_{i,j=1}^N
\left( {\Kahlerm \ua ij} {\Kahlerm \wc j{i-1}} \right)^{ {\fweight{\lambda^{(j)}} {i-j}} }.
\label{eq:gggLemma}
\end{align}
\end{Lemma}
\begin{proof}
(\ref{eq:gggLemma}) follows from (\ref{eq:gfactorQ}) and (\ref{eq:ggfactor}).
By using (\ref{eq:ffactorN}) and (\ref{eq:ffactorZero}),
we obtain
\[
f_\vlambda
=
\prod_{i,j=1}^N
\prod_{\Frac{a,b=1}{b-a \equiv j-i}}^N
q^{\frac 12
\left({\fweight{\lambda^{(j)}}{a-1} }-{\fweight{\lambda^{(j)}}{a} }\right)
\left({\fweight{\lambda^{(i)}}{b-1} }-{\fweight{\lambda^{(i)}}{b} }\right)
}.
\]
For any variables $y_a^i$ such that $y_a^{i+N}=y_{a+N}^i=y_a^i$, we have
\be\label{eq:square-sum}
\sum_{\Frac{ a,b,i,j=0 }{ j-i \equiv \pm(b-a) }}^{N-1}
y_a^i y_b^j
=
\sum_{a=0}^{N-1}
\left(\sum_{i=0}^{N-1}
y_{a\pm i}^i \right)^2.
\ee
Therefore, when
$y_a^i= \bigl|\lambda^{(i)}\bigr|_{a-1} -\bigl|\lambda^{(i)}_{a}$,
\begin{align*}
\sum_{\Frac{ a,b,i,j=0 }{ j-i \equiv \pm(b-a) }}^{N-1}\!\!
 \bigl(\bigl|\lambda^{(i)}\bigr|_{a-1} -\bigl|\lambda^{(i)}\bigr|_{a} \bigr)
\bigl(\bigl|\lambda^{(j)}\bigr|_{b-1} -\bigl|\lambda^{(j)}\bigr|_{b} \bigr)
=
\sum_{i=1}^{N}\!
\left(\sum_{j=1}^{N}\bigl(
\bigl|\lambda^{(j)}\bigr|_{i\pm j} -\bigl|\lambda^{(j)}\bigr|_{i\pm j+1}
\bigr)\right)^2.\!\!\!\!\! \tag*{\qed}
\end{align*} \renewcommand{\qed}{}
\end{proof}

Here is a remark on (\ref{eq:square-sum}) in the case of $N=2$.
It should read, if $y_{a+1}^i = -y_a^i$,
\[
\sum_{\Frac{ a,b,i,j=0 }{ j-i \equiv \pm(b-a) }}^{1}
y_a^i y_b^j
=
\sum_{a=0}^{1}
\left(\sum_{i=0}^{1}
y_{a\pm i}^i \right)^2
=
2
\left(\sum_{i=0}^{1}
y_{i}^i \right)^2.
\]

\subsubsection{Final step}

By using relations
(\ref{eq:ZZfg}),
(\ref{eq:ZpZpfg}),
(\ref{eq:gpvlambdaDef})
and Lemma \ref{L-3},
we obtain
\begin{align}
{ \NekZ{\vlambda}{\vx} }
&=
{ \NekZs{\vlambda}{\vx} }
q^{\frac 14 \sum_{i=1}^{N}
\left(\sum_{j=1}^{N}\left(
{\fweight{\lambda^{(j)}}{i-j} }-{\fweight{\lambda^{(j)}}{1+i-j} }
\right)\right)^2}
\prod_{i,j=1}^N
\left(
\frac{ {\Kahlerm \ua ij} {\Kahlerm \wc j{i-1}} }{ {\Kahlerm \vb ij} {\Kahlerm \vb j{i-1}} }
\right)^{\ha{\fweight{\lambda^{(j)}} {i-j}}}
\nonumber\\
&=
q^{\ha\Delta
}
{ \NekZs{\vlambda}{\vx} }\prod_{i,j=1}^N
\left(
\frac{ {\Kahlerm \ua ij} {\Kahlerm \wc j{i-1}} }{ {\Kahlerm \vb ij} {\Kahlerm \vb j{i-1}} }
\right)^{\ha{\fweight{\lambda^{(j)}} {i-j}}}
\end{align}
and
\begin{align}
{ \NekZp{\vlambda}{\vx} }
&=
{ \NekZps{\vlambda}{\vx} }
q^{\frac 14 \sum_{i=1}^{N}
\left(\sum_{j=1}^{N}\left(
{\fweight{\lambda^{(j)}}{i-j} }-{\fweight{\lambda^{(j)}}{1+i-j} }
\right)\right)^2}
\prod_{i,j=1}^N
\left(\frac{ 1 }{ {\Kahlerm \vb ij} {\Kahlerm \vb j{i-1}} } \right)^{\ha{\fweight{\lambda^{(j)}} {i-j}}}
\nonumber\\
&=
q^{\ha \Delta
}
{ \NekZps{\vlambda}{\vx} }
\prod_{i,j=1}^N
\left(\frac{ 1 }{ {\Kahlerm \vb ij} {\Kahlerm \vb j{i-1}} } \right)^{\ha{\fweight{\lambda^{(j)}} {i-j}}},
\end{align}
where we also used (\ref{eq:ZZshift}) and (\ref{eq:ZZqBorel}) for recasting the first line to the second.%

Since (\ref{eq:Quvw}) implies
\[
\frac{ {\Kahlerm \ua ij} {\Kahlerm \wc j{i-1}} }{ {\Kahlerm \vb ij} {\Kahlerm \vb j{i-1}} } = \frac{\ua_i}{\vb_i} \frac{\vb_{i-1}}{\wc_{i-1}},
\qquad
\frac{ 1 }{ {\Kahlerm \vb ij} {\Kahlerm \vb j{i-1}} } = \frac{\vb_{i-1}}{\vb_{i}} q{\tkappa},
\]
we finally obtain Proposition~\ref{Prop-C}.

\subsection{Inversion symmetry}

By using the partition function of Pochhammer type ${ \NekZ{\vlambda}{\vx} }$,
we can rewrite the non-sta\-tion\-ary~\smash{$\glN$} equation (\ref{eq:EigenEq}) without
$q$-Borel transformation \smash{$q^{ {\frac 12} \Delta}$}.
Suppose $q^n\neq 1$ for any integer $n$.
Let us define the $q$-exponential
without using the $q$-Pochhammer symbol
by
\be
e_q\bigl(xq^{\frac 12}\bigr)
:=
\exp\left(-\sum_{n=1}^\infty \frac{x^n}{n} \frac{1} 
{q^{\frac n2} -q^{-\frac n2} } \right),
\qquad q\in{\mathbb C}^{\times},
\ee
which is a formal power series in $x$. Then it satisfies
$e_q\bigl(xq^{\frac 12}\bigr) e_{q^{-1}}\bigl(xq^{-\frac 12}\bigr)=1$.

Let
$d_i := q\kappa {b_i}/{a_{i+1}}$,
$\overline{d}_i :={b_i}/{c_i}$,
\be
T
:=
\prod_{i=1}^N
\left(\frac{\kappa b_i}{b_{i+1}}
\right)^{\vartheta_i},
\qquad
\mathcal{S}:=\prod_{i=1}^N
\left(
\frac{q}{ d_i {\dbar{i}} }
\right)^{\vartheta_i},
\qquad
\mathcal{T}
:=
\prod_{i=1}^N
\left(\frac{d_i}{q {\dbar{i}} }
\right)^{\vartheta_i},
\ee
then
 \be
{ \NekZ{\vlambda}{\vx} }
=
\bigl(
q^{\Delta}
 \mathsf{T}/\mathcal{T}
 \bigr)^{\frac 12}
{ \NekZs{\vlambda}{\vx} },
\qquad
\sum_{\vlambda} { \NekZs{\vlambda}{\vx} }
=(T\mathcal{S})^{\frac 12}\psi 
\ee
with $\psi$ in Conjecture~\ref{con:1.6}.
Also let
\be
\mathcal{A}'_{C\pm}
:=
\prod_{i=1}^{N}
e_q
\Bigl(x_i \sqrt{ q ({ d_i} /{\overline{d}_i } )^{\pm1} }
\Bigr),
\qquad
\mathcal{A}'_{*}
:=
\mathcal{S}^{\frac 12}
\mathcal{A}_{*}
\mathcal{S}^{-\frac 12},
\qquad
*=L,C,R,
\ee
then
\be
\mathcal{A}'_C
=
\mathcal{A}'_{C+}
\mathcal{A}'_{C-} ,
\qquad
\mathcal{A}'_{R}
=
{:}\prod_{i=1}^N
e_q\Bigl(x_i \sqrt{q d_i \overline d_i} q^{\vartheta_i - \vartheta_{i-1}} \Bigr)^{-1}
{:}.
\ee

The non-stationary $\glN$ equation $(\ref{eq:EigenEq})$ can be rewritten by
\be
 \bigl(q^{\Delta} \mathsf{T}/\mathcal{T}
\bigr)^{\frac 12}
\mathcal{T}^{\frac 12}
\mathcal{A}'_L
\mathcal{A}'_{C-}
\mathcal{A}'_{C+}
\mathcal{A}'_R
\mathcal{T}^{\frac 12}
\bigl(
q^{\Delta}\mathsf{T}/\mathcal{T}
 \bigr)^{\frac 12}
\sum_{\vlambda} { \NekZs{\vlambda}{\vx} }
=
\sum_{\vlambda} { \NekZs{\vlambda}{\vx} }.
\ee
Let
\smash{$\psi^{\rm Poch}
:=
\sum_{\vlambda}{ \NekZ{\vlambda}{\vx}
=
\bigl(
q^{\Delta}
 \mathsf{T}/\mathcal{T}
 \bigr)^{\frac 12}
(T\mathcal{S})^{\frac 12}\psi
}$}.
Since
\be
\mathcal{A'}_{C-}^{-1}
=
\overline{
\mathcal{A}'_{C+}
},
\qquad
\mathcal{A'}_L^{-1}
=
\overline{
\mathcal{A}'_R
},
\ee
 we have
\begin{Proposition}
The non-stationary $\glN$ equation $(\ref{eq:EigenEq})$,
i.e.,
\smash{$
q^{\frac 12 \Delta}
\mathcal{A}_L
\mathcal{A}_C
\mathcal{A}_R
q^{\frac 12 \Delta}
 \mathsf{T}
\psi
=\psi$},
is equivalent to the following inversion symmetry
\be
\mathcal{A}'_{C+}
\mathcal{A}'_R
 \mathcal{T}^{\frac 12}
\psi^{\rm Poch}
=
\overline{
\mathcal{A}'_{C+}
\mathcal{A}'_R
\mathcal{T}^{\frac 12}
\psi^{\rm Poch}
 },
\label{eq:connectionPoch}
\ee
i.e.,
\begin{align}
& {:}\prod_{i=1}^N
\frac{
e_{q}\Bigl(x_i\sqrt{ q { d_i} /{\overline{d}_i } }\Bigr)
}{
e_{q}\Bigl({x_i } q^{\vartheta_i - \vartheta_{i-1}}
{\sqrt{ q d_i{\dbar{i}} }}\Bigr)
}{:} \,
 \mathcal{T}^{\frac 12}
\psi^{\rm Poch}\nonumber\\
&\qquad =
 {:}\prod_{i=1}^N
\frac{
e_{q^{-1}}\Bigl(x_i\sqrt{ {\overline{d}_i } / q{ d_i} } \Bigr)
}{
e_{q^{-1}}\Bigl({x_i }q^{-\vartheta_i + \vartheta_{i-1}}
/{\sqrt{ q d_i {\dbar{i}} }} \Bigr)
}{:}\,
\mathcal{T}^{-\frac 12}
\overline{\psi^{\rm Poch}
}.
\label{eq:InversionDbar}
\end{align}
\end{Proposition}

Note that \eqref{eq:connectionPoch} is also equivalent to the equation where $C+$ is replaced by $C-$.

\section{Instanton expansion with mass truncation}\label{App:Combi}

Let us examine the instanton expansion of the partition function with mass parameter truncation.
For the $\widehat{\mathfrak{gl}}_3$ case, the truncation condition is
\begin{equation}
d_1 = q^{-m_1}, \qquad d_2 = q^{-m_2}, \qquad d_3 = q^{-m_3}.
\end{equation}
Recall that the partition function is a summation over triplets of the Young diagrams$\vec{\lambda} = \smash{\bigl(\lambda^{(1)}, \lambda^{(2)}, \lambda^{(3)}\bigr)}$.
After the mass parameter truncation, the summation is restricted to the triplets such that the length of the first row of $\lambda^{(i)}$ is at most~$m_i$.
Set $M:= m_1 + m_2 + m_3$. In the main text we argued that the rank of the $q$-difference equation for the partition function is
$\frac{1}{2}(M+1)(M+2)$ which depends only on the sum of $(m_1, m_2, m_3)$. For each column of the Young diagram $\lambda^{(i)}$
we define its shifted residue by \smash{$\bigl(\bigr(\lambda^{(i)}\bigr)^\vee_k +i -1\bigr)$}, $k=1, \dots, m_i$, where \smash{$\bigl(\lambda^{(i)}\bigr)^\vee_k$} is the length of the $k$-th column
and $(\bullet)$ means the residue of the integer module~$3$. When the coloring of $\lambda^{(i)}$ is such that the color of the first row is~$i$,
the shifted residue agrees with the color (the number) of the end box of each column.

\subsection[The case M=2]{The case $\boldsymbol{M=2}$}

The rank of the $q$-difference system is $6$.
\begin{enumerate}\itemsep=0pt
\item $(m_1, m_2, m_3) = (2,0,0)$; There are 9 possibilities of the shifted resides of the first two columns of the first Young diagram.

In the table below, $r_k$ is the number of columns with the shifted residue $k$ (see Section~\ref{sec:Jackson-cocycle}).
Note that $r_0+r_1+r_2=2=M$ and the set of possible $(r_0, r_1,r_2)$ has ${\frac{1}{2}(M+1)(M+2)}$ elements,
which agrees with the rank of the truncated $q$-difference system.
We have introduced $z_1=x_1$, $z_2= x_1x_2$ and $z_3=x_1x_2x_3 \equiv 1$.
Then the Young diagram with $(r_0, r_1, r_2)$ contributes to $z_1^{r_1} z_2^{r_2} z_3^{r_0}$,
which is a monomial in $z_1$ with homogeneous degree $M=2$.
\begin{table}[h]\centering
\begin{tabular}{c|c|c|c}
Shifted residues & Contribues to & $(r_0, r_1, r_2)$ & homogeneous monomial \\ \hline
$(1,1)$ & $x_1^2 \Lambda^k$ & $(0,2,0$) & $z_1^2$ \\
$(2,1)$ & $x_1^2 x_2 \Lambda^k$ & $(0,1,1$) & $z_1 z_2$ \\
$(0,1)$ & $x_1 \Lambda^k$ & $(1,1,0$) & $z_1 z_3$ \\
$(1,2)$ & $x_1^2 x_2 \Lambda^k$ & $(0,1,1$) & $z_1 z_2$ \\
$(2,2)$ & $x_1^2 x_2^2 \Lambda^k$ & $(0,0,2$) & $z_2^2$ \\
$(0,2)$ & $x_1 x_2 \Lambda^k$ & $(1,0,1$) & $z_2 z_3$ \\
$(1,0)$ & $x_1 \Lambda^k$ & $(1,1,0$) & $z_1 z_3$ \\
$(2,0)$ & $x_1 x_2 \Lambda^k$ & $(1,0,1$) & $z_2 z_3$ \\
$(0,0)$ & $\Lambda^k$ & $(2,0,0$) & $ z_3^2$ \\
\end{tabular}
\end{table}

\item $(m_1, m_2, m_3) = (1,1,0)$; There are 9 possibilities of the shifted resides of the first columns of two Young diagrams.
The Young diagram with $(r_0, r_1, r_2)$ contributes to $x_1^{-1} z_1^{r_1} z_2^{r_2} z_3^{r_0}$.
Compared with the first case, the monomials are uniformly shifted by $x_1^{-1}$.
\begin{table}[h]\centering
\begin{tabular}{c|c|c|c}
Shifted residues & Contribues to & $(r_0, r_1, r_2)$ & homogeneous monomial \\ \hline
$(1;2)$ & $x_1 x_2 \Lambda^k$ & $(0,1,1$) & $x_1^{-1} z_1 z_2$ \\
$(2;2)$ & $x_1 x_2^2 \Lambda^k$ & $(0,0,2$) & $x_1^{-1} z_2^2$ \\
$(0;2)$ & $x_2 \Lambda^k$ & $(1,0,1$) & $x_1^{-1} z_2 z_3$ \\
$(1;0)$ & $ \Lambda^k$ & $(1,1,0$) & $x_1^{-1} z_1 z_3$ \\
$(2;0)$ & $ x_2 \Lambda^k$ & $(1,0,1$) & $x_1^{-1} z_2 z_3 $ \\
$(0;0)$ & $x_1^{-1} \Lambda^k$ & $(2,0,0$) & $x_1^{-1} z_3^2$ \\
$(1;1)$ & $x_1 \Lambda^k$ & $(0,2,0$) & $x_1^{-1} z_1^2$ \\
$(2;1)$ & $x_1 x_2 \Lambda^k$ & $(0,1,1$) & $x_1^{-1} z_1 z_2$ \\
$(0;1)$ & $\Lambda^k$ & $(1,1,0$) & $x_1^{-1} z_1 z_3$ \\
\end{tabular}
\end{table}

\end{enumerate}
Other four cases $(m_1, m_2, m_3)=(0,2,0),(0,0,2), (1,0,1), (0,1,1)$ are obtained by the cyclic permutation of
$(x_1, x_2, x_3)$.

\subsection[The case M=3]{The case $\boldsymbol{M=3}$}
The rank of the $q$-difference system is $10$. There are $10$ possibilities of $(m_1, m_2,m_3)$, which coincides with the rank.
They are $(3,0,0)$ and its cyclic permutations (3 cases), $(2,1,0)$ and its permutations (6 cases) and $(1,1,1)$.
In each case, there are $3^3 = 27$ possibilities of the three shifted residues, which are classified according to
$(r_0, r_1,r_2)$ with $r_0+r_1+r_2=3$ as follows:
\begin{enumerate}\itemsep=0pt
\item $(m_1,m_2,m_3)=(3,0,0)$:
The Young diagrams with $(r_0, r_1, r_2)$ contribute to
\[z_1^{r_1} z_2^{r_2} z_3^{r_0} \equiv z_1^{r_1} z_2^{r_2}.\]

\item $(m_1,m_2,m_3)=(2,1,0)$:
The Young diagrams with $(r_0, r_1, r_2)$ contribute to
\[
x_1^{-1} z_1^{r_1} z_2^{r_2} z_3^{r_0} \equiv z_1^{r_1-1} z_2^{r_2}.
\]

\item $(m_1,m_2,m_3)=(1,1,1)$:
The Young diagrams with $(r_0, r_1, r_2)$ contribute to
\[
x_1^{-2} x_2^{-1}z_1^{r_1} z_2^{r_2} z_3^{r_0} \equiv z_1^{r_1-1} z_2^{r_2-1}.
\]
\end{enumerate}

\begin{table}[h]\centering\vspace{-5mm}
\caption{The number of cases in the table is the number of terms
involved in the formula \eqref{N=3;cocycle} of the bases $\phi_{(r_0,r_1,r_3)}(z)$ of the cocycle function.}\label{M=3}
\begin{tabular}{c|c||c|c}
$(r_0,r_1,r_2)$ & Number of cases &$ (r_0,r_1,r_2)$ & Number of cases \\ \hline
$(3,0,0)$ & 1 & $(2,1,0)$ & 3 \\
$(2,0,1)$ & 3 & $(1,2,0)$ & 3 \\
$(1,1,1)$ & 6 & $(1,0,2)$ & 3 \\
$(0,3,0)$ & 1 & $(0,2,1)$ & 3 \\
$(0,1,2)$ & 3 & $(0,0,3)$ & 1
\end{tabular}
\end{table}

The allowed terms in the Laurent polynomial in $(x_1, x_2)$ are plotted in Figure \ref{Fig:mass-truncation}.
The fundamental triangle for the case $\mathbf{m} = (3,0,0)$ has the vertices $(0,0), (3,0), (3,3)$.
The triangle for the general case $\mathbf{m} = (m_1,m_2,m_3)$ is obtained from the fundamental triangle
by $-(m_2+m_3)$-shift in~$x_1$ direction and $-m_3$-shift in $x_2$ direction.
We also note that these vertices come from the Young diagrams whose shifted resides are the same, namely $(r_0, r_1, r_2)
= (3,0,0), (0,3,0), (0,0,3)$.

\begin{figure}[ht]
\centering\vspace*{21mm}
\begin{picture}(60,35)
\setlength{\unitlength}{0.9mm}
\thicklines
\put(-50,10){\vector(1,0){30}}
\put(-45,0){\vector(0,1){30}}
\put(-30,10){\line(0,1){15}}
\put(-45,10){\line(1,1){15}}
\put(-45,10){\circle{2}}
\put(-40,10){\circle{2}}
\put(-35,10){\circle{2}}
\put(-30,10){\circle{2}}
\put(-40,15){\circle{2}}
\put(-35,15){\circle{2}}
\put(-30,15){\circle{2}}
\put(-35,20){\circle{2}}
\put(-30,20){\circle{2}}
\put(-30,25){\circle{2}}

\put(-52,28){$x_2$}
\put(-27,5){$x_1$}
\put(-50,35){$\mathbf{m}=(3,0,0)$}
%
%
\put(-10,10){\vector(1,0){30}}
\put(0,0){\vector(0,1){30}}
\put(10,10){\line(0,1){15}}
\put(-5,10){\line(1,1){15}}
\put(-5,10){\circle{2}}
\put(0,10){\circle{2}}
\put(5,10){\circle{2}}
\put(10,10){\circle{2}}
\put(0,15){\circle{2}}
\put(5,15){\circle{2}}
\put(10,15){\circle{2}}
\put(5,20){\circle{2}}
\put(10,20){\circle{2}}
\put(10,25){\circle{2}}

\put(-7,28){$x_2$}
\put(18,5){$x_1$}
\put(-5,35){$\mathbf{m}=(2,1,0)$}

%
%
\put(30,10){\vector(1,0){30}}
\put(45,0){\vector(0,1){30}}
\put(35,5){\line(1,0){15}}
\put(50,5){\line(0,1){15}}
\put(35,5){\line(1,1){15}}
\put(35,5){\circle{2}}
\put(40,5){\circle{2}}
\put(45,5){\circle{2}}
\put(50,5){\circle{2}}
\put(40,10){\circle{2}}
\put(45,10){\circle{2}}
\put(50,10){\circle{2}}
\put(45,15){\circle{2}}
\put(50,15){\circle{2}}
\put(50,20){\circle{2}}

\put(38,28){$x_2$}
\put(58,5){$x_1$}
\put(35,35){$\mathbf{m}=(1,1,1)$}
\end{picture}
\caption{After the mass truncation $d_i=q^{-m_i}$, the affine Laumon partition function becomes
a Laurent polynomial in $(x_1, x_2)$, while it is still a formal power series in~$\Lambda$.
The circles represent the positions of allowed terms in the $(x_1,x_2)$-lattice.
The total number of the circles agrees with the rank of the truncated $q$-difference equation.}\label{Fig:mass-truncation}
\end{figure}

\subsection[General M]{General $\boldsymbol{M}$}
From the above examples we now see the general rule for the possible terms of the partition function after the mass truncation.
In the case $\mathbf{m}= (M,0,0)$, the vertices of the triangle is $(0,0)$, $(M,0)$ and $(M,M)$.
In general case the three vertices are determined by considering the Young diagrams with $\mathbf{r} = (M,0,0)$, $(0,M,0)$ and $(0,0,M)$.
Each case gives the following contribution:
\begin{enumerate}\itemsep=0pt
\item $\mathbf{r} = (M,0,0)$; the shifted residues are
$(\,\overbrace{0, \dots, 0}^{m_1}\,;\,\overbrace{0, \dots, 0}^{m_2}\,;\,\overbrace{0, \dots, 0}^{m_3}\,)$,
which gives the terms with $(x_2x_3)^{m_2} x_3^{m_3} \Lambda^{k}=x_1^{-m_2-m_3} x_2^{-m_3} \Lambda^{k+m_2+m_3}$.
\item $\mathbf{r} = (0,M,0)$; the shifted residues are
$(\,\overbrace{1, \dots, 1}^{m_1}\,;\,\overbrace{1, \dots, 1}^{m_2}\,;\,\overbrace{1, \dots, 1}^{m_3}\,)$,
which gives the terms with $(x_1)^{m_1} (x_3x_1)^{m_3} \Lambda^{k}=x_1^{m_1} x_2^{-m_3} \Lambda^{k+m_3}$.
\item $\mathbf{r} = (0,0,M)$; the shifted residues are
$(\,\overbrace{2, \dots, 2}^{m_1}\,;\,\overbrace{2, \dots, 2}^{m_2}\,;\,\overbrace{2, \dots, 2}^{m_3}\,)$,
which gives the terms with $(x_1x_2)^{m_1} (x_2)^{m_2} \Lambda^{k}=x_1^{m_1} x_2^{m_1+m_2} \Lambda^{k}$.
\end{enumerate}
Hence the vertices are $(-m_2-m_3, -m_3)$, $(m_1, -m_3)$ and $(m_1, m_1+m_2)$, We see that they are $(-m_2-m_3, -m_3)$-shift of
$(0,0)$, $(M,0)$ and $(M,M)$.
The boundary of the shifted triangle is $x_1=m_1$, $x_2 = - m_3$ and $x_2=x_1+m_2$ (see Figure \ref{Fig:shifted triangle}).

\begin{figure}[h]
\centering \vspace*{27mm}
\begin{picture}(40,30)
\setlength{\unitlength}{0.9mm}
\thicklines
\put(-10,15){\vector(1,0){50}}
\put(11,0){\vector(0,1){40}}
\put(25,5){\line(-1,0){30}}
\put(25,5){\line(0,1){30}}
\put(-5,5){\line(1,1){30}}
%
\put(27,18){$m_1$}
\put(3,22){$m_2$}
\put(0,8){$-m_2$}
\put(12,0){$-m_3$}
\put(27,35){$(m_1,m_1+m_2)$}
\put(27,0){$(m_1,-m_3)$}
\put(-30,0){$(-m_2-m_3,-m_3)$}
\put(5,35){$x_2$}
\put(38,10){$x_1$}
\end{picture}

\caption{The triangle on the $(x_1, x_2)$ lattice which indicates possible terms in the instanton expansion
after the mass truncation by $\mathbf{m}=(m_1,m_2,m_3)$.}\label{Fig:shifted triangle}
\end{figure}

\subsection[Generalization to widehat\{gl\}\_N]{Generalization to $\boldsymbol{\glN}$}

Now it is easy to figure out the combinatorics for $\glN$ case.
It is convenient to introduce the following coordinates;
\begin{equation}\label{z-coordinates}
z_1 = x_1, \quad z_2 = x_1 x_2, \quad \dots, \quad z_{N-1}=x_1x_2 \cdots x_{N-1}, \quad z_N = x_1 \cdots x_N = \Lambda \equiv 1.
\end{equation}
We also introduce the fundamental $(N-1)$-dimensional polyhedron \smash{$\Delta^{(N-1)}$ in $(x_1, \dots, x_{N-1})$} space.
The vertices of \smash{$\Delta^{(N-1)}$} are
\[
\mathbf{v}_k:= \smash{(\overbrace{M, \dots, M}^{k}, \overbrace{0, \dots, 0}^{N-1-k})},\qquad k=0,\dots, N-1.
\]
In terms of the coordinates \eqref{z-coordinates}, these vertices correspond to \smash{$z_N^M, z_1^M, \dots, z_{N-1}^M$}, respectively.
By the Pascal's relation
\begin{equation}
\binom{N+M-1}{M} - \binom{N+M-2}{M-1} = \binom{N+M-2}{M},
\end{equation}
one can check by induction that the number of the lattice points in \smash{$\Delta^{(N-1)}$} or on the boundary of \smash{$\Delta^{(N-1)}$} is \smash{$\binom{N+M-1}{N-1}$}
as it should be. Another way to see it is to note that under the identification $z_N \equiv 1$,
the lattice points in \smash{$\Delta^{(N-1)}$} are in one to one correspondence
with the monomials in $z_i$ with homogeneous degree $M$.
\begin{itemize}\itemsep=0pt
\item
When the mass truncation is given by $\mathbf{m} = (m_1, m_2, \dots, m_N)$, we make the shift by $(-m_{i+1}- \cdots -m_N)$
in $x_i$ coordinate to obtain the shifted polyhedron \smash{$\Delta^{(N-1)}(\mathbf{m})$}. The possible terms in the instanton expansion of
the partition function correspond to the lattice points in \smash{$\Delta^{(N-1)}(\mathbf{m})$} or on its boundary.
\item
If the $N$-tuple of the Young diagrams has the shifted residue $\mathbf{r}=(r_0, r_1, \dots, r_{N-1})$, it contribute to
the coefficient of \smash{$z_1^{r_1-m_2}z_2^{r_2-m_3} \cdots z_{N-1}^{r_{N-1}-m_N} z_N^{r_0}$}.
This follows from the fact that a column of $\lambda^{(i)}$ has the shifted residue $k$ gives the factor \smash{$z_{i-1}^{-1} z_k$} up to a power of $\Lambda = z_N$.
\end{itemize}

\subsection*{Acknowledgements}
We would like to thank M.~Bershtein, J.-E.~Bourgine, P.~Di Francesco,
Ken Ito, R.~Kedem, K.~Oshima and V.~Pasquier for useful discussions.
We are also grateful to the anonymous referees for useful comments.
Our work is supported in part by Grants-in-Aid for Scientific Research (Kakenhi);
25K06912 (K.H.), 23K03087 (H.K.), 21K03180 (R.O. and J.S.), 24K06753 (J.S.) and 23K22387 (Y.Y.).
The work of R.O. was partly supported by the Research Institute for Mathematical Sciences,
an International Joint Usage/Research Center located in Kyoto University.
The research of S.S. was carried out within the state assignment of Ministry of Science
and Higher Education of the Russian Federation for IITP RAS.

\pdfbookmark[1]{References}{ref}
\LastPageEnding

\end{document}